\documentclass[12pt]{article}
\usepackage{geometry}
\usepackage{amsmath,amsthm,amscd,amssymb,latexsym, enumerate,bm,lscape,epsfig,nicefrac,hyperref,amsfonts,cite,mathtools,mathrsfs}

\newcommand{\Tr}{\text{\rm Tr}}

\newtheorem{info}{}
\newtheorem{theorem}[info]{Theorem}
\newtheorem{corollary}[info]{Corollary}

\newtheorem{defin}[info]{Definition}

\newtheorem{lemma}[info]{Lemma}
\newtheorem{proposition}[info]{Proposition}
\newtheorem{remark}[info]{Remark}

\numberwithin{info}{section}

\numberwithin{equation}{section}
\renewcommand{\[}{\begin{equation}}
	\renewcommand{\]}{\end{equation}}
\makeatletter

\makeatletter
\g@addto@macro\normalsize{%
	\setlength\abovedisplayskip{5pt}
	\setlength\belowdisplayskip{5pt}
	\setlength\abovedisplayshortskip{4pt}
	\setlength\belowdisplayshortskip{4pt}}
\makeatother

\newcommand{\lam}{\lambda}

\renewcommand{\P}{\mathbb{P}}
\renewcommand{\cal}{\mathcal}

\newcommand{\sq}{\sqrt}
\newcommand{\sign}{\mathrm{Sign}}

\renewcommand{\sq}{\sqrt}

\newcommand{\Sum}{\mathrm{\Sum}}
\newcommand{\Var}{\mathrm{Var}}
\newcommand{\To}{\Rightarrow}
\renewcommand{\Im}{\mathrm{Im}}
\renewcommand{\Re}{\mathrm{Re}}
\renewcommand{\Im}{\mathrm{Im}}
\renewcommand{\Re}{\mathrm{Re}}
\newcommand{\Crit}{\cal{N}}
\newcommand{\n}{\textbf{n}}
\newcommand{\D}{\nabla}
\newcommand{\vol}{\mathrm{vol}}
\newcommand{\rank}{\mathrm{rank}}
\newcommand{\disteq}{\stackrel{d}{=}}

\newcommand{\Mb}{\bar{M}}
\newcommand{\Wb}{\bar{W}}
\newcommand{\Sb}{\bar{S}}
\newcommand{\Vb}{\bar{V}}
\newcommand{\Gb}{\bar{G}}
\newcommand{\Hb}{\bar{H}}
\renewcommand{\[}{\begin{equation}}
	\renewcommand{\]}{\end{equation}}

\newcommand{\N}{\mathbb{N}}
\newcommand{\Z}{\mathbb{Z}}

\newcommand{\R}{\mathbb{R}}
\newcommand{\C}{\mathbb{C}}
\newcommand{\E}{\mathbb{E}}

\renewcommand{\P}{\mathbb{P}}
\renewcommand{\cal}{\mathcal}

\renewcommand{\sq}{\sqrt}

\renewcommand{\Im}{\mathrm{Im}}
\renewcommand{\Re}{\mathrm{Re}}
\renewcommand{\Im}{\mathrm{Im}}
\renewcommand{\Re}{\mathrm{Re}}

\renewcommand{\b}{\textbf}

\newcommand{\SN}{S_N}
\newcommand{\Ezero}{E_{0;p,\tau}}
\newcommand{\uth}{u_{th;p,\tau}}
\newcommand{\Einf}{E_{\infty;p,\tau}}
\newcommand{\tauc}{\tau_c(p)}

\begin{document}
	
	\title{Concentration of Equilibria and Relative Instability in Disordered Non-Relaxational Dynamics}
	
	\author{
		Pax Kivimae \thanks{Courant Institute of Mathematical Sciences, New York University. Email: pax.kivimae@cims.nyu.edu. This research was partially supported by NSF grant DMS-2202720.}}

	\maketitle
	
	\begin{abstract}
		We consider a system of random autonomous ODEs introduced by Cugliandolo et al. \cite{original-glassy}, which serves as a non-relaxational analog of the gradient flow for the spherical $p$-spin model. The asymptotics for the expected number of equilibria in this model was recently computed by Fyodorov \cite{complexity-fyodorov-non-gradient} in the high-dimensional limit, followed a similar computation for the expected number of stable equilibria by Garcia \cite{complexity-xavier}. 
		
		We show that for $p>9$, the number of equilibria, as well as the number of stable equilibria, concentrate around their respective averages, generalizing recent results of Subag and Zeitouni \cite{pspin-second,subag-ofer-new} in the relaxational case. In particular, we confirm that this model undergoes a transition from relative to absolute instability, in the sense of Ben Arous, Fyodorov, and Khoruzhenko \cite{fyo-maywigner-2}.
	\end{abstract}
	
	\footnote{This research was partially supported by NSF grant DMS-2202720.}
	
	\section{Introduction}
	
	%General Introduction
	Understanding the typical behavior of dynamic equilibria (stationary points) in large disordered complex systems is a far-reaching problem motivated by numerous applications in fields such as machine learning, ecology, and economics, among others. The time evolution of such systems may often be modeled by a set of coupled first-order nonlinear autonomous ordinary differential equations (ODEs)
	\[\frac{dx_t}{dt}=\b{F}(x_t),\]
	where $\b{F}$ is a random vector field on some smooth high-dimensional manifold where $x_t$ takes its values.
	
	%By linearizing the system in a neighborhood of a fixed equilibrium, this study formed the basis for the local study of such equilibria.
	
	Much of the interest in this question began with the celebrated paper of May \cite{may}, who modeled the behavior of large ecological systems around an equilibrium point as a random linear system, which May then analyzed by employing tools developed in random matrix theory. They observed that as one varies the model parameters, the unique equilibrium transitions from being typically stable to typically unstable at a fixed value, a result which becomes deterministic in the high-dimensional limit. This result, known as the May-Wigner transition, has since been verified rigorously in a wide variety of models and formed the basis of the local study of equilibria in random systems (see \cite{may-review} for a modern review). 
	
	To study global questions about the dynamics however, one must instead work with non-linear systems, which may have many competing equilibria of varying stability. The simplest statistic available in this case is the number of equilibria, as well as the number of equilibria which are additionally stable. The computation of these expected counts has seen extensive development in recent years through application of the Kac-Rice formula (see \cite{fyo2} and the results therein).
	
	Foundational to recent work in the non-relaxational (i.e. where $\b{F}$ is not the gradient vector field of some potential function) case has been works of Ben Arous, Fyodorov and Khoruzhenko \cite{fyo-maywigner-1,fyo-maywigner-2}, who considered a non-linear analogue of May's model and found the model undergoes two distinct transitions. More specifically, the expected number of equilibria was computed in \cite{fyo-maywigner-1}, who found that for suitably strong noise, this count would become exponentially large in the system dimension, whereas for weak noise the system would typically have only a single equilibrium point. This transition is expected to indicate the onset of complex behavior in the dynamics. 
	
	The behavior within the complex phase was further studied in \cite{fyo-maywigner-2} by considering the expected number of stable equilibria. They found a further transition occurs for this quantity, which is specific to the non-relaxational case. In particular, they found that if the relative strength of the non-relaxational terms was brought above a certain threshold, the system appeared to transition from having an exponentially large number of stable equilibria, to simply having none at all. They designated these as the regimes of "relative" and "absolute" instability, respectively. Moreover, by computing the ratio of these expected counts in the relative instability phase, they predicted that while the number of stable equilibria typically would grow exponentially large, the fraction of equilibria which are stable would be exponentially small. 
	
	Similar results results for such expected counts have been obtained in other models \cite{complex-iso,complex-lacriox,complex-neu}, and in particular, the recent works of Fyodorov and Garcia \cite{complexity-fyodorov-non-gradient,complexity-xavier} found analogous transitions in the non-relaxational spherical $p$-spin model. This model was introduced by Cugliandolo et al. \cite{original-glassy} as a generalization of the spherical pure $p$-spin model, which has served as one of the paradigmatic models in complex systems. In \cite{complexity-fyodorov-non-gradient}, asymptotics for the expected number of equilibria were given, which were found to grow exponentially, independent of the relative strength of the non-relaxational term. The expected number of stable equilibria however, were found in \cite{complexity-xavier} to not only depend on this strength, but would transition from growing exponentially large to exponentially small below a certain threshold, indicating a transition from relative to absolute instability should occur at this point.
	
	The purpose of this paper is to confirm that such a transition occurs in this model. This is, to the knowledge of the author, the first rigorous confirmation that the transition from relative to absolute instability occurs in any of the models predicted above. To do this, we must understand the typical (quenched) behavior of these counts, which in general may deviate drastically from what one would predict from their expected (annealed) value (see \cite{chensen} and the discussion in \cite{fyo-maywigner-2}). 
	
	While powerful heuristics for computing the typical behavior exist in the physics literature \cite{quenched-complexity,complex-replica}, rigorously computing them, and understanding when they coincide with their annealed counterparts, has only seen progress quite recently. In particular, in the relaxational case, major progress was made by Subag \cite{pspin-second}, who showed that the number of local minima for the spherical $p$-spin model concentrates around its average. This result was then extended by Subag and Zeitouni \cite{subag-ofer-new} to the number of critical points. These results have proven important in establishing properties of the dynamics \cite{pspin-second-application3} as well as properties of corresponding potential landscape \cite{pspin-second-application1,pspin-second-application4}. Additionally, these works have since been extended to other models in the relaxational case  \cite{pspin-second-application2,kivimae,david-quench}. The main goal of this paper is an extension of these methods, and in particular the required framework and results in random matrix theory, the non-relaxational setting.
	
	We now recall the model, introduced in \cite{original-glassy} and based on the earlier model introduction in the $p=2$ case introduced by \cite{original-1}. Fix a choice of $p\ge 3$ and $\tau\in [-(p-1)^{-1},1]$. Let $\b{f}(x)=(f_k(x))_{k=1}^{N}$ be the smooth centered Gaussian vector-valued function on $\R^N$ whose components satisfy
	\[\E[f_k(x)f_l(y)]=\delta_{kl}p\left(x,y\right)_N^{p-1}+\tau p(p-1)\left(\frac{y_lx_k}{N}\right)\left(x,y\right)_N^{p-2},\;\;\;1\le k,l\le N,\label{eqn:model-def}\]
	where $(x,y)_N=N^{-1}\sum_{i=1}^Nx_iy_i$ denotes the Euclidean inner product. We then define $\b{F}(x)$ for $x\in \SN$ by
	\[\b{F}(x):=-\lam(x)x+\b{f}(x),\]
	where $\SN=\{x\in \R^N:\|x\|^2=N\}$ denotes the $(N-1)$-sphere of radius $\sqrt{N}$ and $\lam(x)=(x,\b{f}(x))_N$ is a Lagrange multiplier chosen so that 
	\[\b{F}(x)\in T_xS_N=\{v\in \R^N:(v,x)_N=0\}.\]
	One may alternatively define $\b{F}$ by specifying the terms in the following decomposition of $\b{f}$ 
	\[f_i(x)=c_s\D_i^{euc}H(x)+\frac{c_a}{N}\sum_{j=1}^{N}x_j A_{ij}(x),\label{eqn:model-def-explicit}\]
	where the matrix $A_{ij}(x)$ is anti-symmetric (i.e. $A_{ij}=-A_{ji}$), $\D_i^{euc}$ denotes the $i$-th standard Euclidean partial derivative, and $c_s,c_a$ are normalization constants. When $p=2$, this coincides with the decomposition of a matrix in symmetric and anti-symmetric parts. In addition, it is similar in form to the gradient-solenodial (Hodge) decomposition used to define the models in \cite{fyo-maywigner-2,fyo-maywigner-1}. In our case, we take $H$ to be a spherical $p$-spin model, and for each of $1\le i<j\le N$, takes $A_{ij}$ to be an independent $(p-2)$-spin model. Then taking $c_s=\sq{p^{-1}(1+(p-1)\tau)}$ and $c_a=\sq{(p-1)(1-\tau)}$, the components of (\ref{eqn:model-def-explicit}) satisfy (\ref{eqn:model-def}), as checked in Lemma \ref{lem:model-equivalence} of Appendix \ref{appendix:covariance-computation}.
	
	We note that in the case of $\tau=1$ we have that $c_a=0$ and that $c_s=1$ so that $\b{F}(x)=\D H(x)$. That is to say in this case, $\b{F}$ is not only relaxational, but given by the gradient vector field of the spherical $p$-spin model. This contrasts the case where $\tau=-(p-1)^{-1}$, where we have that $c_{s}=0$ so that $\b{F}$ is purely anti-symmetric.
	
	Another case of interest is when $\tau=0$, where one sees that (\ref{eqn:model-def}) becomes $\E[f_k(x)f_l(y)]=\delta_{kl}p\left(x,y\right)_N^{p-1}$, so that each of the $f_i$ are independent. Moreover, each coincides with a spherical $(p-1)$-spin model up to a normalization factor. In particular, in this case equilibria correspond to the zeros of a set of i.i.d. random functions.
	
	We will now prepare to precisely state our results, for which we will need notation. Let us denote the number of equilibria of the vector field $\b{F}$ as $\Crit_{N,tot}$, and similarly for denote the number of stable equilibria as $\Crit_{N,st}$. Asymptotics for $\E[\Crit_{N,tot}]$ were obtained by Fyodorov \cite{complexity-fyodorov-non-gradient}.
	
	\begin{theorem}
		\label{theorem:fyodorov 1st moment weaker}\cite{complexity-fyodorov-non-gradient}
		We have that for $|\tau|<1$ and $p\ge 3$ that
		\[\lim_{N\to \infty} N^{-1}\log(\E[\Crit_{N,tot}])=\frac{1}{2}\log(p-1).\]
	\end{theorem}
	
	Note that if one lets $\tau=1$, this coincides with the results of Auffinger, Ben Arous, and \v{C}ern\'{y} \cite{pspin-one} for the spherical $p$-spin model. Building upon this, Garcia \cite{complexity-xavier} computed asymptotics for $\E]\Crit_{N,st}]$.
	
	\begin{theorem}
		\label{theorem: xavier 1st moment}\cite{complexity-xavier}
		We have that for $|\tau|<1$ and $p\ge 3$ that
		\[\lim_{N\to \infty} N^{-1}\log(\E[\Crit_{N,st}])=\frac{1}{2}\log(p-1)-\frac{(p-2)(1+\tau)}{2(1+(p-1)\tau)}=:\theta_{p,\tau}.\]
	\end{theorem}
	
	With some mild algebraic manipulation, this shows that for each $p\ge 3$, there is a critical value
	\[\tauc:=\frac{(p-2)-\log(p-1)}{(p-1)\log(p-1)-(p-2)}> 0,\]
	such that for $\tau<\tauc$ it is exponentially rare that the system has even a single stationary equilibrium, while for $\tau>\tauc$ the expected number of stationary equilibria is exponentially large, but growing at a rate slower than the expected total number. Note this coincides exactly with the transition from relative to absolute instability found in \cite{fyo-maywigner-2}. Our main result will be the following concentration result, which confirms that this transition occurs. 
	
	\begin{theorem}
		\label{theorem:main-theorem-weak}
		Assume $p>9$ and $\tau\in (0,1)$. Then in probability as $N\to \infty$
		\[N^{-1}\log(\Crit_{N,tot})\To \frac{1}{2}\log(p-1).\label{eqn:theorem:weakmain1}\]
		If in addition $\tau\in (\tauc,1)$, then in probability as $N\to \infty$
		\[N^{-1}\log(\Crit_{N,st})\To \theta_{p,\tau},\label{eqn:theorem:weakmain2}\]
		so that in particular we have also in probability as $N\to \infty$
		\[N^{-1}\log(\Crit_{N,st}/\Crit_{N,tot})\To -\frac{(p-2)(1+\tau)}{2(1+(p-1)\tau)}.\]
	\end{theorem}
	
	We prove this by applying the second moment method. Before doing this, it will be useful to introduce the counts for equilibrium points with a given Lagrange multiplier. A key motivation for this is that the value taken by the Lagrange multiplier at an equilibrium point will essentially determine the equilibrium's stability. Explicitly, let us define the threshold
	\[\Einf=(1+\tau)\sq{p(p-1)}.\]
	Then a fixed equilibrium point $x\in \SN$ satisfying $\lam(x)>\Einf$ is typically stable, and similarly a equilibrium point satisfying $\lam(x)<\Einf$ is typically unstable. In particular, we will find it useful to denote for $u\in \R\cup\{-\infty\}$, the total number of equilibria $x$, such that $\lam(x)>u$ by $\Crit_N(u)$, and similarly for $\Crit_{N,st}(u)$. We now recall the following more general form of Theorem \ref{theorem:fyodorov 1st moment weaker}, also derived in \cite{complexity-fyodorov-non-gradient}.
	
	\begin{theorem}
		\label{theorem:fyodorov 1st moment}\cite{complexity-fyodorov-non-gradient}
		Assume $|\tau|<1$ and $p\ge 3$. For $u\in \R\cup\{-\infty\}$ we have that
		\[\lim_{N\to \infty} N^{-1}\log(\E[\Crit_N(u)])=\sup_{v>u}\Sigma^{p,\tau}(v),\]
		where $\Sigma^{p,\tau}$ is a fixed function recalled in (\ref{eqn:def:1st complexity function}) below.
	\end{theorem} 
	
	An important threshold for the function $\Sigma^{p,\tau}$ will be the unique value $\Ezero>0$, for which $\Sigma^{p,\tau}(v)<0$ for $v>\Ezero$. Note by Markov's inequality, it is exponentially rare for there to exist a single equilibrium such that the associated Lagrange multiplier is greater than $\Ezero+\epsilon$ for any $\epsilon>0$. We also observe that $\Sigma^{p,\tau}(\Einf)=\theta_{p,\tau}$, so that $\tauc$ may equivalently be defined as unique $\tauc$, such that $\Einf>\Ezero$ for $\tau<\tauc$ and $\Einf<\Ezero$ for $\tau>\tauc$. 
	
	With these thresholds defined, our primary result then is the following second moment estimates.
	
	\begin{theorem}
		\label{theorem:main-theorem constant order concentration bulk}
		Assume $\tau \in (0,1)$ and $p>9$. Then for $u\in [-\infty,\Ezero)$ with $u\neq \Einf$ we have that
		\[\lim_{N\to \infty}\frac{\E[\Crit_N(u)^2]}{\E[\Crit_N(u)]^2}=1.\label{eqn:theorem:main1}\]
		If in addition $\tau\in (\tauc,1)$ and $u>\Einf$
		\[\lim_{N\to \infty}\frac{\E[\Crit_{N,st}(u)^2]}{\E[\Crit_{N,st}(u)]^2}=1.\label{eqn:theorem:main2}\]
	\end{theorem}
	
	\begin{remark}
		Note then when $\tau\in (0,\tauc]$, $\Einf\ge \Ezero$ so the claim (\ref{eqn:theorem:main2}) would be vacuous in this case.
	\end{remark}
	
	By a standard argument, (\ref{eqn:theorem:main1}) implies that in probability as $N\to \infty$
	\[\log(\Crit_{N}(u))-\log(\E[\Crit_{N}(u)])\To 0.\]
	In particular, we observe that (\ref{eqn:theorem:main1}) and Theorem \ref{theorem:fyodorov 1st moment weaker} immediately imply (\ref{eqn:theorem:weakmain1}). Similarly, the identity $\Sigma^{p,\tau}(\Einf)=\theta_{p,\tau}$ (which may be verified using (\ref{eqn:def:1st complexity function}) below) we can essentially obtain (\ref{eqn:theorem:weakmain2}) from (\ref{eqn:theorem:main2}), Theorem \ref{theorem: xavier 1st moment}, and Markov's inequality.
	
	Finally, we also show the following result on the exponential scale, though with less restrictions on $(p,\tau,u)$, obtained in the course of proving Theorem \ref{theorem:main-theorem constant order concentration bulk}.
	
	\begin{theorem}
		\label{theorem:main-theorem exponential order concentration}
		For any $|\tau|<1$, $p\ge 3$, and $u\in \R\cup\{-\infty\}$ we have that
		\[\lim_{N\to \infty}N^{-1}\log \left(\frac{\E[\Crit_N(u)^2]}{\E[\Crit_N(u)]^2}\right)=0.\]
	\end{theorem}
	
	We note that Theorem \ref{theorem:main-theorem constant order concentration bulk} actually establishes a much stronger concentration result than is used by Theorem \ref{theorem:main-theorem-weak}. Indeed, we show that $N^{-1}\log(\Crit_{N}(u))$ fluctuates on the scale $O(N^{-1})$, while we only need a result of order $o(1)$ to establish the transition at $\tauc$. Moreover, it naively appears that one would only need a result like Theorem \ref{theorem:main-theorem exponential order concentration} to establish this level of concentration. To support this, in the relaxational case, analogues of Theorem \ref{theorem:main-theorem-weak} can be used to compute the potential function's minimum value (when suitably normalized), as discussed in \cite{pspin-second,kivimae}. This however, relies on knowing that the minimum value concentrates very quickly by the Borell-TIS inequality. If a similar general result were known for $N^{-1}\log(\Crit_{N}(u))$, Theorem \ref{theorem:main-theorem exponential order concentration} would be sufficient to establish the transition. 
	
	However, while such a result is widely believed to hold, no progress to establish it has yet been made. Moreover, in many of the applications of concentration of complexity listed above, concentration of $N^{-1}\log(\Crit_{N}(u))$ on the order $o(1)$ would not suffice, so stronger estimates like Theorem \ref{theorem:main-theorem constant order concentration bulk} seem to still be needed.
	
	\subsection{Methods}
	%Methods
	
	We will now discuss the methods we use to show our main result, with a more complete outline being the focus of Section \ref{section:outline-exponential}. The starting point will be the Kac-Rice formula, which will give us an explicit representation of $\E[\Crit_{N}(u)^2]$ (see Lemma \ref{lem:second-moment-exact}). The primary difficulty in analyzing this expression is a term involving the product of absolute determinants of a pair of conditional Jacobians.
	
	The corresponding expression for the first moment in \cite{complexity-fyodorov-non-gradient} expresses $\E[\Crit_{N}(u)]$ as a fixed multiple of the expected absolute determinant of a matrix sampled from the elliptic ensemble (recalled below in Definition \ref{defin:GEE}), up to a random multiple of the identity. Fixing this multiple, the problem of computing the first moment is reduced to producing asymptotics for the expected absolute characteristic polynomial of the elliptic ensemble.
	
	To compute the second moment however, we must deal with both the effect the conditioning has on the Jacobian, as well as the effect of the correlations between the Jacobians themselves. The first effectively augments the law of the Jacobians by a low-rank perturbation, so we are reduced to the problem of computing the expected characteristic polynomial for perturbations of a pair of correlated matrices both sampled from the elliptic ensemble.
	
	On the exponential scale, we do this by approximating the characteristic polynomial by the exponential of a more well-behaved linear statistic. Not only does this make the characteristic polynomial more stable under perturbation, but using suitable concentration estimates, we may show that this statistic may be replaced by its expectation, which is insensitive to correlations. In the symmetric case this method was explored quite fully in the work of Ben Arous, Bourgade and McKenna \cite{exponential}, who showed how one may employ this method under a fairly mild set of conditions. We follow this method, which entails establishing and modifying a variety of results from symmetric matrix theory to the non-symmetric setting, where the eigenvalues are less stable under perturbation. This will be the focus of Section \ref{section:elliptic-ensemble-background}. Many of our results are given specifically for the elliptic ensemble, but essentially only rely on concentration estimates for the matrix and rates of convergence for the expected empirical measures of shifted singular values. As such, we expect our general set-up should readily apply to most non-symmetric matrices for which the Kac-Rice formula has currently been applied.
	
	To obtain estimates on the determinant term to vanishing order, as needed for the proof of Theorem \ref{theorem:main-theorem constant order concentration bulk}, will be a more difficult task. The key problem is that in this case the correlations between the Jacobians are no longer negligible, as indeed the fluctuations of the determinant are typically on the order of a polynomial factor. Thus the primary difficulty is to control these fluctuations, for which our main tool will be a moment bound for the determinant term which is tight up to a constant. An important technical input for this will be the computation of the higher moments for the absolute characteristic polynomial of the elliptic ensemble up to constant order, which is obtained in our companion paper \cite{PaxComp}.
	
	Finally we comment on our restrictions on $(p,\tau)$. We believe one could relax the assumptions of Theorem \ref{theorem:main-theorem-weak} in sight of Theorem \ref{theorem:main-theorem exponential order concentration}. The restriction on $p>9$ is needed only in the proof of Lemma \ref{lem:delta-estimates-near-zero} and is similar in nature to the restriction to $p\ge 32$ in \cite{subag-ofer-new}, though removing it is beyond the scope of our methods. The primary reason our proof holds for this lower value of $p$ is that the characteristic polynomial when $\tau<1$ has weaker fluctuations then in the symmetric case, as well as the related increase in the typical spacing of eigenvalues from $N^{-1}$ in the symmetric case to $N^{-1/2}$.
	
	While we believe our methods could be easily generalized to the case $\tau=0$, significant changes would be needed to include $\tau<0$. The first problem is that while the largest real eigenvalue of the elliptic ensemble is typically $1+\tau$, the largest singular value is typically $1+|\tau|$. This renders a large portion of Section \ref{section:elliptic-ensemble-background} useless, as the largest singular value no longer accurately detects when a shift by the identity is positive definite. Moreover, we are unsure of how to obtain any constant order asymptotics for the higher moments of the absolute characteristic polynomial of the elliptic ensemble in this case.
	
	\subsection{Structure}
	
	We begin the paper with Section \ref{section:outline-exponential}, which outlines the general argument and reduces the proofs of our main results to a series of more technical results proven in the later sections.
	
	In Section \ref{section:KR-2 point-intro} we will provide our application of the Kac-Rice formula to the second moment, and give results which describe the structure of the conditional Jacobian at two points. In Section \ref{section:KR-2 point-old}, we compute the leading order asymptotics for this formula, which is given by Theorem \ref{theorem:KR-2point}. These sections primarily follow the standard structure of the Kac-Rice formalism, as in \cite{complexity-fyodorov-non-gradient,complexity-xavier,pspin-one,pspin-second,bipartite}. In particular, we tried to structure these sections similarly to Sections 4 and 5 of \cite{pspin-second}, so as to assist readers familiar with this work. 
	
	We then refine our results to the constant order in Section \ref{section:theorem-O1-section}, where we prove the tighter asymptotics for the correlated determinants mentioned above. Then in Section \ref{section:lemmas about functions} we prove some analytic results about the limiting complexity functions. 
	
	Section \ref{section:elliptic-ensemble-background} contains the results we need to treat the characteristic polynomial on the exponential scale, as well as outside of the bulk. The methods are quite general, and we expect them to apply to many non-symmetric matrix ensembles which are suitably well behaved (as in the sense of \cite{exponential} in the symmetric case), and are sufficient to obtain an exponential order results like Theorem \ref{theorem:main-theorem exponential order concentration}. The treatment of the characteristic polynomial in the bulk is more complicated, and is obtained in the companion paper \cite{PaxComp}, and is required to obtain the more delicate Theorem \ref{theorem:main-theorem constant order concentration bulk}.
	
	Finally, we provide two Appendices. Appendix \ref{appendix:covariance-computation} contains the majority of the raw technical computations involving the covariance structure, as well as the verification of a few miscellaneous results involving them. In Appendix \ref{appendix:Kac-Rice-Proof} we give the exact form of the Kac-Rice formula we use, which is a modification of the results in Chapter 6 of \cite{azais}.
	
	\subsection{Notation}
	
	We will need some notation involving matrices. For an $n$-by-$n$ matrix $A$ we will denote the (algebraic) eigenvalues of $A$ as
	$\lam_1(A),\dots, \lam_n(A)$, ordered lexicographically, so that $\Re(\lam_{i}(A))\le \Re(\lam_{i+1}(A))$ and if $\Re(\lam_{i}(A))= \Re(\lam_{i+1}(A))$ we also have that $\Im(\lam_{i}(A))\le \Im(\lam_{i+1}(A))$. We will denote its singular values as $0\le s_n(A)\le \dots \le s_1(A)$. We will denote the empirical (spectral) measure for the eigenvalues of $A$ by
	\[\mu_A=\frac{1}{n}\sum_{i=1}^n \delta_{\lam_i(A)}.\]
	In addition, for any choice of $z\in \R$ and $n$-by-$n$ matrix $A$ we define the empirical measure for the singular values of $A-zI$ by
	\[\nu_{A,z}=\frac{1}{n}\sum_{i=1}^n \delta_{s_i(A-zI)},\]
	and $\nu_{A}=\nu_{A,0}$. We will use the notation $\|A\|$ to denote the Frobenius norm on matrices (i.e. $\|A\|=\sqrt{\Tr(A^TA)}$).

	We now precisely define our counts. For a Borel subset $B\subseteq \R$ define
	\[\Crit_N(B)=\#\{\sigma \in \SN:\b{F}(\sigma)=0,\lam(\sigma)\in B\},\]
	\[\Crit_{N,st}(B)=\#\{\sigma \in \SN:\b{F}(\sigma)=0,\lam(\sigma)\in B, \Re(\lam_N(\D \b{F}(\sigma))\le  0\},\]
	where $\D \b{F}(\sigma)$ denotes the Jacobian at $\sigma$ taken with respect to the standard Riemannian metric on $\SN$.
	
	We observe that $\Crit_N(u)=\Crit_N((u,\infty))$ and $\Crit_{N,st}(u)=\Crit_{N,st}((u,\infty))$, and will henceforth abandon the notations $\Crit_N(u)$ and $\Crit_{N,st}(u)$ and instead work with $\Crit_{N}(B)$ and $\Crit_{N,st}(B)$. Next, for Borel subsets $B\subseteq \R$ and $I\subseteq [-1,1]$, we define
	\[\Crit_{N,2}(B,I)=\#\{(\sigma,\tau)\in \SN^2:\b{F}(\sigma)=\b{F}(\tau)=0,\lam(\sigma),\lam(\tau)\in B, (\sigma,\tau)_N\in I\}.\]
	We note that $\Crit_{N,2}(B,[-1,1])=\Crit_N(B)^2$. We will say a subset of $\R$ is nice if it is a finite union of open intervals, and similarly we will say a subset of $[-1,1]$ is nice if it is the intersection of a nice subset of $\R$ with $[-1,1]$.
	
	We now define our complexity functions. For $|\tau|<1$ we define the ellipse $\cal{E}_\tau\subseteq \C$
	\[\cal{E}_\tau=\left\{x+iy\in \C:\frac{x^2}{(1+\tau)^2}+\frac{y^2}{(1-\tau)^2}\le 1\right\},\label{eqn:ellipse}\]
	and denote the logarithmic potential of the uniform measure on this ellipse as
	\[\phi_\tau(u)=\frac{1}{(1-\tau^2)\pi}\int_{\cal{E}_\tau}\log(|x+iy-u|)dxdy.\]
	For $\tau=1$, we define
	\[\phi_1(u)=\lim_{\tau\to 1}\phi_\tau(u)=\frac{1}{2\pi}\int_{-2}^2\sq{4-x^2}\log(|x-u|)dx,\]
	the logarithmic potential of the standard semicircle law. It is well known that (see (4.1) of \cite{integral-on-ellipse})
	\[\phi_\tau(u)=\frac{u^2}{2(1+\tau)}-\frac{1}{2}-I_\tau(u)I(|u|\ge 1+\tau),\]
	\[I_\tau(u)=\begin{cases}\frac{1}{2(1+\tau)}u^2-\frac{|u|(|u|-\sq{u^2-4\tau})}{4\tau}-\log(\frac{|u|+\sq{u^2-4\tau}}{2});\;\;\tau\neq 0\\ -\log u+\frac{1}{2}u^2-\frac{1}{2};\;\;\tau=0\end{cases}.\]
	It will be convenient to define a rescaling of this as well
	\[\phi_{\tau,p}(u)=\phi_\tau\left(\frac{u}{\sq{p(p-1)}}\right).\]
	With this notation the annealed complexity function is given as
	\[\Sigma^{p,\tau}(u)=\frac{1}{2}+\frac{1}{2}\log(p-1)-\frac{u^2}{2p\alpha}+\phi_{\tau,p}\left(u\right)\label{eqn:def:1st complexity function}\]
	where here and elsewhere $\alpha=1+\tau(p-1)$. For the two-point complexity function, we first define an auxiliary function
	\[h(r):=\frac{1}{2}\log \left(\frac{1-r^2}{1-r^{2p-2}}\right),\label{def:h-def}\]
	and define
	\[\Sigma_{2}^{p,\tau}(r,u_1,u_2)=1+\log(p-1)+h(r)-\frac{1}{2}(u_1,u_2)\Sigma_{U}(r)^{-1}\begin{bmatrix}u_1\\ u_2 \end{bmatrix}+\phi_{\tau,p}\left(u_1\right)+\phi_{\tau,p}\left(u_2\right),\]
	where $\Sigma_U(r)=\Sigma_{U,p,\tau}(r)$ is a certain 2-by-2 matrix-valued function defined in (\ref{eqn:def:U-external field form}) below.
	
	Finally, we define the key classical ensemble occurring in this model, namely the (Gaussian) real elliptic ensemble.
	\begin{defin}
		\label{defin:GEE}
		For $-1\le \tau\le 1$ we will say a $N$-by-$N$ random matrix $A_N$ belongs to the real elliptic ensemble $GEE(\tau,N)$ if, when considered as a vector in $\R^{N^2}$, $A_N$ is a centered Gaussian random vector whose covariance is given by
		\[\E[A_{ij}A_{kl}]=\frac{1}{N}(\delta_{ik}\delta_{jl}+\tau \delta_{il}\delta_{jl}).\]
	\end{defin}
	
	The limiting law for the empirical measure (when $|\tau|<1$) for the elliptic ensemble was first identified by Girko \cite{girko1,girko2} as the uniform measure on the ellipse $\cal{E}_\tau$ defined above. More specifically, it was shown in \cite{spec-anti} (see as well \cite{naumov}) that for $A_N$ sampled from $GEE(\tau,N)$, we have as $N\to \infty$
	\[\mu_{A_N}\To \frac{1}{(1-\tau^2)\pi }I_{\cal{E}_\tau}\label{eqn:elliptic law}\]
	in probability, where $I_{\cal{E}_\tau}$ denotes the indicator function of the ellipse defined in (\ref{eqn:ellipse}). When $\tau=0$ this recovers the circular law, and when $\tau\to 1$ we recover the semi-circular law.
	
	\section{Outline of The Proofs for Theorems \ref{theorem:main-theorem constant order concentration bulk} and \ref{theorem:main-theorem exponential order concentration} \label{section:outline-exponential}}
	
	In this section, we outline the general structure of the proof of our main results. We begin by giving the exact expressions for the first two moments of $\Crit_N(B)$ and then introduce Theorem \ref{theorem:KR-2point}, which concerns the computation of $\E[\Crit_N(B)^2]$ up to the exponential scale. We then provide some analytic results about the resulting complexity functions, from which we derive Theorem \ref{theorem:main-theorem exponential order concentration}.
	
	After this we will proceed to the proof of Theorem \ref{theorem:main-theorem constant order concentration bulk}. Our first step will be to reduce this proof to Proposition \ref{prop:constant order concentration bulk}, which concerns the contribution of points which are roughly orthogonal. This makes the computation at this scale possible, as the components appearing in the Kac-Rice formula become roughly independent. In particular, we are able to prove Theorem \ref{theorem:main-theorem constant order concentration bulk} through a series of estimates for the absolute determinant term in the Kac-Rice formula.
	
	Our exact formula for $\E[\Crit_{N,2}(B,I)]$ which will follow from a standard application of the Kac-Rice formula. This expresses $\E[\Crit_{N,2}(B,I)]$ as an integral over $(\sigma,\sigma')\in\SN\times \SN$, though as the law of $\b{F}$ is isotropic, we will reduce this to an integral depending only on $r=(\sigma,\sigma')_N$ by the co-area formula. 
	
	\begin{lemma}
		\label{lem:second-moment-exact}
		Let $(M_N^k(r))_{k=1,2}$ and $(U^k_N(r))_{k=1,2}$ denote the random variables given in Lemma \ref{lem:jacobian-covariance-statement-no-energy}. These are, respectively, a pair of $(N-1)$-by-$(N-1)$ random matrices and a pair of random scalars with marginal law given by $(U^1_N(r),U^2_N(r))\disteq \cal{N}(0,N^{-1}\Sigma_U(r))$ where $\Sigma_U(r)$ is given by (\ref{eqn:def:U-external field form}). Then for nice $B\subseteq \R$ and nice $I\subseteq (-1,1)$
		\[\E[\Crit_{N,2}(B,I)]=\]\[C_N \int_I\left(\frac{1-r^2}{1-r^{2p-2}}\right)^{\frac{N-2}{2}}\cal{G}(r)\E \left[\prod_{k=1,2}|\det(M^k_N(r))|I(U^k_N(r)\in B)\right]dr.\label{eqn:intro-second-moment}\]
		where here $\cal{G}(r):=(1-r^2)^{-1/2}(1-(r^p-\tau (p-1)r^{p-2}(1-r^2))^2)^{-1/2}$ and
		\[C_N=\vol(S^{N-1})\vol(S^{N-2})\left((N-1)(2\pi )^{-1}(p-1)\right)^{(N-1)},\]
		where $\vol(S^{\ell-1})$ denotes the surface area of the $(\ell-1)$-dimensional unit sphere
		\[\vol(S^{\ell-1})=\frac{2\pi^{\ell/2}}{\Gamma(\ell/2)}.\]
	\end{lemma}
	
	This will be proven in Section \ref{section:KR-2 point-intro}, based on the form of the Kac-Rice formula derived in Appendix \ref{appendix:Kac-Rice-Proof}. From this expression, we see that the primary obstacle for obtaining asymptotics for $\E[\Crit_{N,2}(B,I)]$ is understanding the term
	\[\E \left[\prod_{k=1,2}|\det(M^k_N(r))|I(U^k_N(r)\in B)\right].\label{eqn:ignore-75}\]
	To express how this is done, it will be convenient to recall the corresponding exact expression for the annealed complexity, which occurs as Theorem 3.1 of \cite{complexity-fyodorov-non-gradient}.
	
	\begin{lemma}
		\label{lem:first-moment-exact}
		Let $A_{N-1}$ be sampled from $GEE(\tau,N-1)$ and let $X_N$ be an independent centered Gaussian random variable with variance $p\alpha/N$. Then for nice $B\subseteq \R$ we have that
		\[\E[\Crit_N(B)]=\bar{C}_N\E\left[|\det\left(A_{N-1}-\hat{X}_NI\right)|I(X_N \in B)\right],\]
		where here
		\[\hat{X}_N=\left(\frac{N}{(N-1)p(p-1)}\right)^{1/2}X_N,\]
		\[\bar{C}_N=\vol(S^{N-1})\left((N-1)(2\pi)^{-1}(p-1)\right)^{(N-1)/2}.\]
	\end{lemma}
	
	Conditioning on the value of $X_N$, we see that the computation of $\E[\Crit_N(B)]$ is similarly reduced to the computation of the expected value of the absolute characteristic polynomial of a matrix sampled from $GEE(\tau,N-1)$. To relate this to the term relevant to the second moment, we define $(\bar{M}^k_N(r,u_1,u_2))_{k=1,2}$, to be $M^k_N(r)$ conditioned on the event $(U^1_N(r)=u_1,U^2_N(r)=u_2)$. We will show in Section \ref{section:theorem-O1-section} that there are rank-$2$ matrices $E_N^k(r,u_1,u_2)$ such that we may write
	\[(\bar{M}_N^k(r,u_1,u_2))_{k=1,2}\disteq (A_{N-1}^k(r)-\hat{u}_{k,N}I+E_N^k(r,u_1,u_2))_{k=1,2}\label{eqn:decom-E}\]
	where $(A_{N-1}^1(r),A_{N-1}^2(r))$ are a pair of correlated $GEE(\tau,N-1)$ matrices and
	\[\hat{u}_{k,N}:=\left(\sq{\frac{N}{(N-1)p(p-1)}}\right)u_k.\]
	Thus to compare the first and second moments of $\Crit_N(B)$, we want to show that the contribution to the determinant coming from the perturbation term $E_N^k(r,u_1,u_2)$, as well as the effect of the correlations between the matrices $(A_{N-1}^1(r),A_{N-1}^2(r))$, are both negligible.
	
	On the exponential scale, which is sufficient to obtain the leading order term of $\E[\Crit_{N,2}(B,I)]$, this will follow by replacing the determinant with a suitably regularized analog. For brevity, we will only obtain an upper bound for the leading order behavior, as it is all we will need to demonstrate concentration. The corresponding lower bound may be obtained by applying the methods developed in \cite{exponential}. The following computation will be our main result on this scale, whose proof will be the focus of Section \ref{section:KR-2 point-old}.
	\begin{theorem}
		\label{theorem:KR-2point}
		For nice $B\subseteq \R$, and nice $I\subseteq (-1,1)$, we have that
		\[\limsup_{N\to \infty}N^{-1}\log(\E[\Crit_{N,2}(B,I)])\le \sup_{u_1,u_2\in B,r\in I}\Sigma_{2}^{p,\tau}(r,u_1,u_2).\]
	\end{theorem}
	
	We note that with Theorems \ref{theorem:fyodorov 1st moment} and \ref{theorem:KR-2point}, the proof of Theorem \ref{theorem:main-theorem exponential order concentration} is reduced to a comparison of the resulting complexity functions. To begin we need the following result which allows us to restrict our attention to points where $u_1=u_2$.
	\begin{lemma}
		\label{lem:diagonal-lemma-intro}
		For nice $B\subseteq \R$ that intersects $(-\Ezero,\Ezero)$ we have that
		\[\limsup_{N\to\infty}N^{-1}\log(\E[\Crit_{N}(B)^2])\le \sup_{u\in B,r\in (-1,1)}\Sigma_2^{p,\tau}(r,u,u).\]
	\end{lemma}
	
	This result will be proven in Section \ref{section:lemmas about functions}. We will also need the following purely analytic result, also proven in Section \ref{section:lemmas about functions}.
	
	\begin{lemma}
		\label{lem:analytic equality result}
		Fix a nice subset $B\subseteq \R$ which intersects $(-\Ezero,\Ezero)$ and further fix a nice subset $I\subseteq (-1,1)$. Then $0\in \bar{I}$ if and only if
		\[\sup_{v\in B ,r\in I }\Sigma_{2}^{p,\tau}(r,v,v)=2\sup_{v\in B}\Sigma^{p,\tau}(v),\]
		with equality replaced by $<$ otherwise. 
	\end{lemma}
	
	\begin{proof}[Proof of Theorem \ref{theorem:main-theorem exponential order concentration}]
		By Jensen's inequality we only need to show that
		\[\limsup_{N\to \infty}N^{-1}\log \left(\frac{\E[\Crit_N(u)^2]}{\E[\Crit_N(u)]^2}\right)\le 0.\label{eqn:gnore-1113}\]	
		By Lemmas \ref{lem:diagonal-lemma-intro} and \ref{lem:analytic equality result}, we see that
		\[\limsup_{N\to \infty}N^{-1}\log \E[\Crit_N(u)^2]\le 2\sup_{v>u}\Sigma^{p,\tau}(v).\]
		Combined with Theorem \ref{theorem:fyodorov 1st moment}, this demonstrates (\ref{eqn:gnore-1113}).
	\end{proof}
	
	We now turn to the proof of Theorem \ref{theorem:main-theorem constant order concentration bulk}, which concerns the refinement of our results to the scale $1+o(1)$. In particular, we must return our attention to understanding (\ref{eqn:ignore-75}) now up to a term of vanishing order. Even considering this up to a $O(1)$-term requires significantly more care, as we are now in the regime in which the correlations of $(A^1_{N-1}(r),A^2_{N-1}(r))$ may not only contribute non-trivially but may even diverge in $N$. To illustrate this, note that if $u=\hat{u}_{1,N}=\hat{u}_{2,N}$ then as $r\to 1$, this amounts to understanding the ratio
	\[\frac{\E[|\det(A_N-uI)|^2]}{\E[|\det(A_N-uI)|]^2}.\]
	In the case where $\tau\in (0,1)$ and $|u|<1+\tau$ it follows from the results of \cite{PaxComp} that this quantity grows like $\sqrt{N}$. 
	
	To deal with this, we first remove a number of negligible subsets of points from our counts. We show below (see Remark \ref{remark:strict increasing of phi}) that $\Sigma^{p,\tau}$ is strictly increasing on $(-\infty,0]$ and strictly decreasing on $[0,\infty)$. Using Lemmas \ref{lem:diagonal-lemma-intro} and \ref{lem:analytic equality result} we see for any $0\le u\le \Ezero$ and $0<\epsilon<\Ezero-u$,
	\[\limsup_{N\to \infty}N^{-1}\log(\E[\Crit_{N}((u+\epsilon,\infty))^2])\le 2\Sigma^{p,\tau}(u+\epsilon)<2\Sigma^{p,\tau}(u).\]
	However, using Theorem \ref{theorem:fyodorov 1st moment}, we have that
	\[2\Sigma^{p,\tau}(u)=\lim_{N\to \infty}N^{-1}\log(\E[\Crit_{N}((u,u+\epsilon))]^2)\le \liminf_{N\to \infty}N^{-1}\log(\E[\Crit_{N}((u,u+\epsilon))^2]).\]
	In particular, we conclude that
	\[\limsup_{N\to \infty}N^{-1}\log(\E[\Crit_{N}((u+\epsilon,\infty))^2])<\liminf_{N\to \infty}N^{-1}\log(\E[\Crit_{N}((u,u+\epsilon))^2]).\label{eqn:ignore-c-2}\]
	A similar argument shows that for $u<0$ we have that
	\[\limsup_{N\to \infty}N^{-1}\log(\E[\Crit_{N}((u,\infty)\setminus(-\epsilon,\epsilon))^2])<\liminf_{N\to \infty}N^{-1}\log(\E[\Crit_{N}((-\epsilon,\epsilon))^2]).\label{eqn:ignore-c-3}\]
	
	On the other hand, by Theorem \ref{theorem:KR-2point} and Lemmas  \ref{lem:analytic equality result} and \ref{lem:diagonal-lemma-old} we see that if $u\in (0,\Ezero)\setminus \{\Einf\}$ then for sufficiently small $\epsilon>0$ and $\rho>0$ we have that
	\[\limsup_{N\to \infty}N^{-1}\log(\E[\Crit_{N,2}((u,u+\epsilon),(-\rho,\rho)^c)])<2\Sigma^{p,\tau}(u),\]
	so as above we have that
	\[\limsup_{N\to \infty}N^{-1}\log(\E[\Crit_{N,2}((u,u+\epsilon),(-\rho,\rho)^c)])<\liminf_{N\to \infty}N^{-1}\log(\E[\Crit_N((u,u+\epsilon))]^2),\label{eqn:ignore-c-1}\]
	and similarly if $u<0$ for small $\epsilon>0$ we have that
	\[\limsup_{N\to \infty}N^{-1}\log(\E[\Crit_{N,2}((-\epsilon,\epsilon),(-\rho,\rho)^c)])<\liminf_{N\to \infty}N^{-1}\log(\E[\Crit_{N}((-\epsilon,\epsilon))^2]).\]
	
	In particular, we see that the dominant contribution to $\E[\Crit_N((u,\infty)^2)]$ comes from pairs of equilibria which are roughly orthogonal and which have roughly the same Lagrange multiplier. It is for these points that we compute the contribution up to a vanishing term.
	
	\begin{proposition}
		\label{prop:constant order concentration bulk}
		Let us assume that $\tau\in(0,1)$, $p>9$, and let us fix $u\in \R$ such that $u\neq \pm \Einf$. Then for any choice of $\rho_N>0$ and $a_N<b_N$ such that $\rho_N\to 0$ and $a_N,b_N\to u$
		\[\limsup_{N\to \infty}\frac{\E[\Crit_{N,2}((a_N,b_N),(-\rho_N,\rho_N))]}{\E[\Crit_N((a_N,b_N))]^2}\le 1.\]
	\end{proposition}

	Before discussing the proof of Proposition \ref{prop:constant order concentration bulk}, we first derive Theorem \ref{theorem:main-theorem constant order concentration bulk} as a consequence of it, and one final result, proven in Section \ref{section:theorem-O1-section}, which shows that the dominant contribution to $\E[\Crit_N((u,\infty))]$ for $u>\Einf$ comes from stable points.
	
	\begin{lemma}
		\label{lem:stable points}
		Assume $\tau\in (0,1)$. Then for any $u>\Einf$, we have that
		\[\limsup_{N\to \infty}N^{-1}\log \left(\frac{\E[\Crit_{N}((u,\infty))-\Crit_{N,st}((u,\infty))]}{\E[\Crit_{N}((u,\infty))]}\right)<0.\]
	\end{lemma}

	\begin{proof}[Proof of Theorem \ref{theorem:main-theorem constant order concentration bulk}]
		We first prove (\ref{eqn:theorem:main1}). By Jensen's inequality it suffices to show that
		\[\limsup_{N\to \infty}\frac{\E[\Crit_N((u,\infty))^2]}{\E[\Crit_N((u,\infty))]^2}\le 1.\label{eqn:ignore-pen}\]
		We will show this in the case that $0\le u<\Ezero$, with the other case of $u<0$ being similar. For this we will use the following elementary lemma which we state without proof.
		\begin{lemma}
			\label{lem:increasing}
			Let $a_{n,m}:\N\times \N\to (-\infty,0)$ be non-decreasing in $m$ and such that \[\limsup_{n\to \infty}a_{n,m}<0\]
			for each $m\in \N$. Then for any increasing sequence $b_n\to \infty$, there is some sequence $m_n\to \infty$ such that
			\[\limsup_{n\to \infty}b_n a_{n,m_n}=-\infty.\]
		\end{lemma}
		
		Thus as left hand side of (\ref{eqn:ignore-c-2}) is non-increasing in $\epsilon>0$, and the right hand side non-decreasing, we may use Lemma \ref{lem:increasing} to find some sequence $\epsilon_N>0$ such that $\epsilon_N\to 0$ sufficiently slowly and such that
		\[\limsup_{N\to \infty}N^{-1/2}\log \frac{\E[\cal{N}_{N}((u+\epsilon_N,\infty))^2]}{\E [\cal{N}_{N}((u,u+\epsilon_N))^2]}=-\infty,\]
		which gives in particular that
		\[\lim_{N\to \infty}\frac{\E[\cal{N}_{N}((u+\epsilon_N,\infty))^2]}{\E [\cal{N}_{N}((u,u+\epsilon_N))^2]}=0.\label{eqn:ignore-3}\]
		As we a.s. have that $\cal{N}_{N}((u,\infty))=\cal{N}_{N}((u,u+\epsilon_N))+\cal{N}_{N}((u+\epsilon_N,\infty))$, we see by using the inequality $\E (X+Y)^2\le \E X^2+\E Y^2 +2 (\E X^2 )^{1/2}(\E Y^2)^{1/2}$ and (\ref{eqn:ignore-3})
		\[1\le \limsup_{N\to \infty}\frac{\E[\cal{N}_{N}((u,\infty))^2]}{\E [\cal{N}_{N}((u,u+\epsilon_N))^2]}\le 1+\limsup_{N\to \infty}\frac{\E[\cal{N}_{N}((u+\epsilon_N,\infty))^2]}{\E [\cal{N}_{N}((u,u+\epsilon_N))^2]}\]\[
		+2\limsup_{N\to \infty}\frac{\E[\cal{N}_{N}((u+\epsilon_N,\infty))^2]^{1/2}}{\E [\cal{N}_{N}((u,u+\epsilon_N))^2]^{1/2}}=1.\]
		In particular, we see that
		\[\limsup_{N\to \infty}\frac{\E[\cal{N}_N(u,\infty)^2]}{\E[\cal{N}_N(u,\infty)]^2}=\limsup_{N\to \infty}\frac{\E[\cal{N}_N(u,u+\epsilon_N)^2]}{\E[\cal{N}_N(u,\infty)]^2}.\label{eqn:ignore-4}\]
		
		Similarly, by (\ref{eqn:ignore-c-1}) we may choose, for small fixed $\delta>0$ a sequence $\rho_N\to 0$ such that
		\[\lim_{N\to \infty}\frac{\E[\Crit_{N,2}((u,u+\delta),(-\rho_N,\rho_N)^c)]}{\E[\Crit_N((u,u+\delta))^2]}=0.\]
		In particular, we see that
		\[\lim_{N\to \infty}\frac{\E[\Crit_{N,2}((u,u+\epsilon_N),(-\rho_N,\rho_N)^c)]}{\E[\Crit_{N}((u,\infty))^2]}=0.\]
		Furthermore, note that
		\[\limsup_{N\to \infty}\frac{\E[\Crit_{N,2}((u,u+\epsilon_N),(-\rho_N,\rho_N))]}{\E[\Crit_N((u,\infty))]^2}\le \]
		\[\limsup_{N\to \infty}\frac{\E[\Crit_{N,2}((u,u+\epsilon_N),(-\rho_N,\rho_N))]}{\E[\Crit_N((u,u+\epsilon_N))]^2}.\]
		In particular,
		\[\limsup_{N\to \infty}\frac{\E[\Crit_N((u,u+\epsilon_N))^2]}{\E[\Crit_N((u,\infty))]^2}\le \limsup_{N\to \infty}\frac{\E[\Crit_{N,2}((u,u+\epsilon_N),(-\rho_N,\rho_N))]}{\E[\Crit_N((u,u+\epsilon_N))]^2},\]
		and so by (\ref{eqn:ignore-4}) we finally obtain that
		\[\limsup_{N\to \infty}\frac{\E[\Crit_N((u,\infty))^2]}{\E[\Crit_N((u,\infty))]^2}\le \limsup_{N\to \infty}\frac{\E[\Crit_{N,2}((u,u+\epsilon_N),(-\rho_N,\rho_N))]}{\E[\Crit_N((u,u+\epsilon_N))]^2}.\label{eqn:ignore-pen-k}\]
		Applying Proposition \ref{prop:constant order concentration bulk} to the right hand side completes the proof of (\ref{eqn:ignore-pen}) and thus of (\ref{eqn:theorem:main1}) as well.
		
		To prove (\ref{eqn:theorem:main2}) we note it again suffices to show
		\[\limsup_{N\to \infty}\frac{\E[\Crit_{N,st}((u,\infty))^2]}{\E[\Crit_{N,st}((u,\infty))]^2}\le 1.\label{eqn:ignore-pen-2}\]
		By Lemma \ref{lem:stable points} when $u>\Einf$ we have that
		\[\lim_{N\to \infty} \frac{\E[\Crit_N((u,\infty))]}{\E[\Crit_{N,st}((u,\infty))]}=1.\]
		Thus we have 
		\[\limsup_{N\to \infty}\frac{\E[\Crit_{N,st}((u,\infty))^2]}{\E[\Crit_{N,st}((u,\infty))]^2}\le \limsup_{N\to \infty}\frac{\E[\Crit_{N}((u,\infty))^2]}{\E[\Crit_{N,st}((u,\infty))]^2}=\]
		\[\limsup_{N\to \infty}\frac{\E[\Crit_{N}((u,\infty))^2]}{\E[\Crit_{N}((u,\infty))]^2},\]
		and so (\ref{eqn:ignore-pen-2}) follows from (\ref{eqn:ignore-pen}).
	\end{proof}
	
	The focus of the remainder of the section will be the proof of Proposition \ref{prop:constant order concentration bulk}. We denote the ratio of the absolute determinants as
	\[\Delta_N(r,u_1,u_2):=\frac{\E[\prod_{k=1,2}|\det(\bar{M}^k_N(r,u_1,u_2))|]}{\prod_{k=1,2}\E\left[|\det(A_{N-1}-\hat{u}_{N,k} I)|\right]},\]
	where here $A_{N-1}$ is sampled from $GEE(\tau,N-1)$. By expanding the expressions in Lemmas \ref{lem:second-moment-exact} and \ref{lem:first-moment-exact} around $r=0$, we obtain the following result, proven in Section \ref{section:theorem-O1-section}.
	
	\begin{lemma}
		\label{lem:ratio of second and first moment}
		For any choice of $u\in \R$, $\rho_N>0$ and $a_N<b_N$ such that $\rho_N\to 0$ and $a_N,b_N\to u$, we have that
		\[\frac{\E[\Crit_{N,2}((a_N,b_N),(-\rho_N,\rho_N))]}{\E[\Crit_N((a_N,b_N))]^2}\le\]\[ (1+o(1))\sup_{a_N<u_1,u_2<b_N}\left(\sq{\frac{N}{2\pi}}\int_{(-\rho_N,\rho_N)}e^{-N(\frac{r^2}{2}+o(r^3))}\Delta_N(r,u_1,u_2)dr\right).\]
	\end{lemma}
	
	This reduces us to producing estimates for the size of $\Delta_N(r,u_1,u_2)$. The first of these will give us a bound on (\ref{eqn:ignore-75}) up to a polynomial order in $N$.
	
	\begin{lemma}
		\label{lem:moments of M}
		Fix $\tau\in(0,1)$, even $\ell$, $u_1,u_2\in \R$, $r\in (-1,1)$ and $k=1,2$. Then there is $C$ such that if $|u_1|,|u_2|<\Einf$
		\[\frac{\E[|\det(\bar{M}_N^k(r,u_1,u_2))|^\ell]}{\E[|\det(A_{N-1}-\hat{u}_{k,N}I)|]^\ell}\le C N^{\ell(\ell-1)/4},\]
		and if $|u_1|,|u_2|>\Einf$
		\[\frac{\E[|\det(\bar{M}_N^k(r,u_1,u_2))|^\ell]}{\E[|\det(A_{N-1}-\hat{u}_{k,N}I)|]^\ell}\le C.\]
		Moreover, these bounds are uniform over compact subsets of $(r,u_1,u_2)\in (-1,1)\times (\R\setminus\{\pm \Einf\})^2$.
	\end{lemma}
	
	The proof of Lemma \ref{lem:moments of M}, given in Section \ref{section:theorem-O1-section}, is primarily reliant on the following bounds on moments of the absolute characteristic polynomial for the Gaussian elliptic ensemble.
	
	\begin{theorem}
		\label{theorem:moments of character main}
		Fix $\tau\in(0,1)$, $\ell\in \N$, and $x\in \R$. Then there is $C$ such that if $|x|<1+\tau$
		\[\frac{\E[|\det(A_N-xI)|^{\ell}]}{\E[|\det(A_N-xI)|]^{\ell}}\le CN^{\ell(\ell-1)/4},\]
		and such that if $|x|>1+\tau$
		\[\frac{\E[|\det(A_N-xI)|^{\ell}]}{\E[|\det(A_N-xI)|]^{\ell}}\le C.\]
		Moreover, these bounds are uniform in compact subsets of $\R\setminus\{\pm (1+\tau)\}$.
	\end{theorem}
	
	In the bulk, this follows directly from Theorem 1 of \cite{PaxComp}. Outside of the bulk, this follows from Proposition \ref{prop:moments of character outside bulk}, which is a main result of Section \ref{section:elliptic-ensemble-background}.
	
	Finally we will need a result which controls the correlations on the $O(1)$-scale when the matrices are almost independent.
	
	\begin{lemma}
		\label{lem:delta-estimates-near-zero}
		Assume $p>9$ and $\tau\in(0,1)$, and fix choices of $C>0$ and $u\in \R\setminus\{\pm \Einf\}$. Then for any $a_N<b_N$ with $a_N,b_N\to u$ we have that
		\[\limsup_{N\to \infty}\sup_{a_N<u_1,u_2<b_N}\sup_{|r|<C\sq{\log(N)/N}}\Delta_N(r,u_1,u_2)\le 1. \]
	\end{lemma}
	
	This lemma will also be proven in Section \ref{section:theorem-O1-section}, and relies on treating the correlation of the two matrices as a perturbation term and then applying coarse perturbative methods.
	
	Equipped with these preliminaries, we may give the proof of Proposition \ref{prop:constant order concentration bulk}.
	
	\begin{proof}[Proof of Proposition \ref{prop:constant order concentration bulk}]
		By Lemma \ref{lem:ratio of second and first moment} we are reduced to showing that
		\[\limsup_{N\to \infty}\sup_{a_N<u_1,u_2<b_N}\left(\sq{\frac{N}{2\pi}}\int_{(-\rho_N,\rho_N)}e^{-N(\frac{r^2}{2}+o(r^3))}\Delta_N(r,u_1,u_2)dr\right)\le 1.\]
		For sufficiently large $N$, we have that $(a_N,b_N)\cap \{\pm E_{\infty;p,\tau}\}=\varnothing$ and $\rho_N\le 1/2$. Thus using Lemma \ref{lem:moments of M} and the Cauchy-Schwarz inequality, we see there is large $C>0$ such that for $a_N\le u_1,u_2\le b_N$ and $|r|\le \rho_N$ we have that
		\[\Delta_N(r,u_1,u_2)\le \frac{\left(\E[|\det(\bar{M}_N^1(r,u_1,u_2))|^2]\E[|\det(\bar{M}_N^2(r,u_1,u_2))|^2]\right)^{1/2}}{\E[|\det(A_{N-1}-\hat{u}_{1,N}I)|]\E[|\det(A_{N-1}-\hat{u}_{2,N}I)|]}\le C\sq{N}.\]
		To use this, let us denote $J_N:=(-\rho_N,\rho_N)\setminus(-10\sq{\log(N)/N},10\sq{\log(N)/N})$. In particular, for $\epsilon\in (0,1/8)$ and $N$ large enough that $|\rho_N|\le 1/2$, we have that
		\[\int_{J_N}e^{-N(\frac{r^2}{2}-\epsilon|r|^3)}\Delta_N(r,u_1,u_2)dr\le C\sq{N}\int_{J_N}e^{-N(\frac{r^2}{2}-\epsilon|r|^3)}dr\]
		\[\le C\sq{N}\int_{J_N}e^{-\frac{Nr^2}{4}}dr\le  C2|\rho_N|N^{1/2-10^2/4},\label{eqn:ignore-141}\]
		where in the last step we have just bounded the integral by its maximal value. This final expression goes to zero uniformly in the choice of $u_1,u_2$. In particular, we are reduced to showing that
		\[\limsup_{N\to \infty}\sup_{a_N<u_1,u_2<b_N}\left(\sq{\frac{N}{2\pi}}\int_{-10\sq{\log(N)/N}}^{10\sq{\log(N)/N}}e^{-N(\frac{r^2}{2}+o(r^3))}\Delta_N(r,u_1,u_2)dr\right)\le 1.\label{eqn:ignore1641}\]
		Bounding this integral by taking the maximal value of $\Delta_N(r,u_1,u_2)$ over the domain of $r$, and then applying Lemma \ref{lem:delta-estimates-near-zero}, we see (\ref{eqn:ignore1641}) is bounded by
		\[\limsup_{N\to \infty}\left(\sq{\frac{N}{2\pi}}\int_{-10\sq{\log(N)/N}}^{10\sq{\log(N)/N}}e^{-N(\frac{r^2}{2}+o(r^3))}dr\right).\]
		
		To contend with this note that for $\epsilon>0$ we have that
		\[\sq{\frac{N}{2\pi}}\int_{-10\sq{\log(N)/N}}^{10\sq{\log(N)/N}}e^{-N(\frac{r^2}{2}-\epsilon|r^3|)}dr\le \]
		\[e^{N\epsilon 10^3\left(\frac{\log(N)}{N}\right)^{3/2}}\sq{\frac{N}{2\pi}}\int_{-\infty}^{\infty}e^{-\frac{Nr^2}{2}}dr=e^{N\epsilon 10^3\left(\frac{\log(N)}{N}\right)^{3/2}}=1+o(1).\]
		
		Altogether this completes the proof.
	\end{proof}
	
	\section{Proof of Lemma \ref{lem:second-moment-exact} \label{section:KR-2 point-intro}}
	
	In this section, we will give a proof of Lemma \ref{lem:second-moment-exact}. This will follow readily from a direct application of the Kac-Rice formula, as well as a variety of computations about the covariance structure of $\b{F}$ and its derivatives, which will be given in Appendix \ref{appendix:covariance-computation}. The exact form of the Kac-Rice formula we will apply is given by Lemma \ref{lem:formal-KR} below, which is proven in Appendix \ref{appendix:Kac-Rice-Proof}. This is a modification of the Kac-Rice formula for vector fields given in Chapter 6 of \cite{azais} to the case of a general Riemannian manifold. It is important to note that, even restricted to the relaxational case, this requires fewer non-degeneracy conditions than the corresponding formula of \cite{AT}, which simplifies the number of conditions we need to verify.
	
	To begin, we perform a normalization on $\b{F}$ to simplify the factors which appear in our analysis. For each $\ell$, let $S^{\ell}:=S^{\ell}(1).$ We define for $x\in S^{N-1}$ \[\b{g}(x):=\b{f}(\sq{N}x),\;\; \b{G}(x):=\b{F}(\sq{N}x),\;\; \zeta(x):=\lam(\sq{N}x).\]
	We observe that $\b{g}$, $\b{G}$ and $\zeta$ are, respectively, a vector-valued function, vector field, and function on $S^{N-1}$. The zero sets of $\b{G}$ and $\b{F}$ are related up to a rescaling, so it will suffice to prove a similar result for $\b{G}$. Indeed we have that
	\[\Crit_N(B)=\#\{\sigma \in S^{N-1}:\b{G}(x)=0,\zeta(x)\in B\},\]
	and similarly for $\Crit_{N,2}(B,I)$, so it will suffice to work with $(\zeta,\b{g},\b{G})$ instead of $(\lam,\b{f},\b{F})$.
	
	Given a (piecewise) smooth orthogonal frame field $E$ on $S^{N-1}$, we define \[\b{G}_{E}(x):=((\b{G}(x),E_i(x)))_{i=1}^{N-1},\;\;\;\D_{E} \b{G}(x):=(E_i(\b{G}(x),E_j(x)))_{i,j=1}^{N-1},\]
	where the inner-product here is the standard Riemannian metric on $S^{N-1}$ inherited from the standard embedding into $\R^N$. Lastly, we denote for $r\in [-1,1]$
	\[\n(r)=(0,\dots 0,\sq{1-r^2},r)\in S^{N-1},\]
	and the north pole as $\n=\n(1)$. 
	
	We now give the result of a direct application of the Kac-Rice formula, after which we will simplify the remaining terms to obtain Lemma \ref{lem:second-moment-exact}.
	\begin{lemma}
		\label{lem:Kac-Rice}
		Let $E$ be an arbitrary (piecewise) smooth orthonormal frame field on $S^{N-1}$. For any nice $B\subseteq \R$ and nice $I\subseteq (-1,1)$,
		\[\E[\Crit_{N,2}(B,I)]=\tilde{C}_N\int_I(1-r^2)^{\frac{N-3}{2}}\varphi_{\b{G}_E(\n),\b{G}_E(\n(r))}(0,0)\]
		\[\times \E\big[|\det \left(\frac{\D_E \b{G}(\n)}{\sq{(N-1)p(p-1)}}\right)||\det \left(\frac{\D_E \b{G}(\n(r))}{\sq{(N-1)p(p-1)}}\right)|\times\]
		\[ I(\zeta(\n),\zeta(\n(r))\in B)\big| \b{G}_E(\n)=\b{G}_E(\n(r))=0\big]dr,\label{eqn:kac-rice-usage-lem}\]
		where here $\varphi_{\b{G}_E(\sigma),\b{G}_E(\sigma')}$ is the density of the random vector $(\b{G}_E(\sigma),\b{G}_E(\sigma'))$ at zero and 
		\[\tilde{C}_N=\vol(S^{N-1})\vol(S^{N-2})((N-1)p(p-1))^{N-1}.\]
	\end{lemma}
	
	Before proving this, we note that the law of the function $(\zeta,\b{g}):S^{N-1}\to \R^{N+1}$ in isotropic in the following sense: for $R\in O(N)$ we have the distributional equality
	\[(\zeta,\b{g})\disteq(\zeta\circ R,R^{-1}\b{g}\circ R)\]
	where $\b{g}\circ R(x)=\b{g}(R(x))$ denotes precomposition by the rotation $R$. In particular, the same holds for $\b{G}$.
	
	\begin{proof}[Proof of Lemma \ref{lem:Kac-Rice}]
		We define $S^2_N(I)=\{(x,y)\in S^{N-1}\times S^{N-1}:(x,y)\in I\}$, considered as an open Riemannian submanifold of $S^{N-1}\times S^{N-1}$. In the notation of Lemma \ref{lem:formal-KR}, $\Crit_{N,2}(B,I)=N_{h,k}(B)$ with vector field on $S^2_N(I)$ given by $h(x,y)=(\b{G}(x),\b{G}(y))$, where the decomposition is with respect to the identification $T_{(x,y)}S^2_N(I)=T_xS^{N-1}\oplus T_yS^{N-1}$, and $k(x,y)=(\zeta(x),\zeta(y))$. 
		
		One may produce an orthonormal frame field on $\bar{E}$ on $S^2_N(I)$, by acting on each coordinate by $E$, which satisfies
		\[\D_{\bar{E}}h(x,y)=\begin{bmatrix}\D_E \b{G}(x)&0\\ 0&\D_E \b{G}(y)\end{bmatrix}\]
		with block structure with respect to the above decomposition of $T_{(x,y)}S^2_N(I)$. 
		
		We note that if $r\neq \pm 1$, then by Lemma \ref{lem:conditioning-density-statement} below $(\b{G}(\n),\b{G}(\n(r)))$ is non-degenerate. Employing the above invariance, we see that the vector $(\b{G}(x),\b{G}(y))$ is non-degenerate as long as $x\neq \pm y$. In particular, we may apply Lemma \ref{lem:formal-KR} to obtain that for any orthonormal frame field on $S^{N-1}$
		\begin{gather}
			\E[\Crit_{N,2}(B,I)]=\int_{S^2_N(I)}\varphi_{\b{G}_E(x),\b{G}_E(y)}(0,0) \E\big[|\det(\D_E \b{G}(x))||\det(\D_E \b{G}(y))|\times\\
			I(\zeta(x),\zeta(y)\in B)\big| \b{G}_E(x)=\b{G}_E(y)=0\big]\omega_{S^{N-1}}(dx)\omega_{S^{N-1}}(dy),
			\label{eqn:ignore-KR-temp}
		\end{gather}
		where $\omega_{S^{N-1}}$ is the standard surface measure on $S^{N-1}$. By the above invariance, the integrand of (\ref{eqn:ignore-KR-temp}) is invariant under applying a rotation to $S^{N-1}\times S^{N-1}$. By the co-area formulae (see as well pg. 10 of \cite{pspin-second}) for any isotropic function on $S^{N-1}\times S^{N-1}$
		\[\int_{S^2_N(I)} f(x,y)\omega_{S^{N-1}}(dx)\omega_{S^{N-1}}(dy)=\]
		\[\vol(S^{N-1})\vol(S^{N-2})\int_I (1-r^2)^{\frac{N-3}{2}}f(\n,\n(r))dr.\]
		Applying this identity to (\ref{eqn:ignore-KR-temp}) completes the proof.
	\end{proof}
	
	We now need the following results which describe the various terms appearing in (\ref{eqn:kac-rice-usage-lem}), beginning with the following expression for the conditional density.
	
	\begin{lemma}
		\label{lem:conditioning-density-statement}
		For any $r\in (-1,1)$, there is an orthonormal frame field $E=(E_i)_{i=1}^{N-1}$ such that the following holds: The Gaussian vector $(\b{G}_E(\n),\b{G}_E(\n(r)))$ is non-degenerate, with density at the origin given by
		\[\varphi_{\b{G}_E(\n),\b{G}_E(\n(r))}(0,0)=(2\pi p)^{-(N-1)}\left(1-r^{2p-2}\right)^{-\frac{N-2}{2}}\times\]\[(1-(r^p-\tau (p-1)r^{p-2}(1-r^2))^2)^{-1/2}=(2\pi p)^{-(N-1)}\left(1-r^{2p-2}\right)^{-\frac{N-1}{2}}\cal{G}(r),\]
		where $\cal{G}(r)$ is the function introduced in Lemma \ref{lem:Kac-Rice}.
		\\
		Moreover, conditional on the event $(\b{G}_E(\n)=\b{G}_E(\n(r))=0)$, $\sq{N}(\zeta(\n),\zeta(\n(r)))$ is a non-degenerate centered Gaussian vector, with $N$-independent covariance matrix $\Sigma_U(r)$, whose inverse is given by (\ref{eqn:def:U-external field form}) below.
	\end{lemma}
	
	The proof of this result is given in Appendix \ref{appendix:covariance-computation}. We will now describe the law of the conditional Jacobian, which we will break into two statements, the first giving the result after one conditions on the event that the values of $\b{G}$ vanish, and the second after conditioning additionally on the values of $\zeta$. This is done as the first is more convenient for the proof of Theorem \ref{theorem:KR-2point}, while in the proof of Proposition \ref{prop:constant order concentration bulk} we will find it convenient to employ the latter.
	
	\begin{lemma}
		\label{lem:jacobian-covariance-statement-no-energy}
		For $r\in (-1,1),$ and with the same choice of $(E_i)_{i=1}^{N-1}$ as in Lemma \ref{lem:conditioning-density-statement}, the following holds: Conditional on the event \[(\b{G}_E(\n)=0,\b{G}_E(\n(r))=0)\] let us denote the law of \[\left(\frac{\D_E \b{G}(\n)}{\sq{(N-1)p(p-1)}},\frac{\D_E \b{G}(\n(r))}{\sq{(N-1)p(p-1)}},\zeta(\n),\zeta(\n(r))\right)\]
		as \[(M^1_N(r),M^2_N(r),U^1_N(r),U^2_N(r)).\]
		For $k=1,2$, let us define 
		\[\hat{U}^k_N(r):=\left(\sq{\frac{N}{p(p-1)(N-1)}}\right)U^k_N(r),\]
		and write $M^k_N(r)$ in terms of square $(N-1)$-by-$1$ blocks as 
		\[M^k_N(r)=\begin{bmatrix}
			H^k_N(r)&  W^k_N(r)\\
			V^k_N(r)^t&S^k_N(r)\\
		\end{bmatrix}-\hat{U}^k_N(r) I.\]
		Then $(H^k_N(r),W^k_N(r),V^k_N(r),S^k_N(r),U^k_N(r))_{k=1,2}$ is a Gaussian random vector that has a covariance structure which satisfies the following properties:
		\begin{itemize}
			\item The variables $(S^k_N(r),U^k_N(r))_{k=1,2}$, $(V^k_N(r),W^k_N(r))_{k=1,2}$, and $(H^k_N(r))_{k=1,2}$ are independent and centered.
			\item We have the distributional equality
			\[\sqrt{\frac{N-1}{N-2}}\begin{bmatrix} H^1_N(r)\\ H^2_N(r)\end{bmatrix}\disteq\begin{bmatrix}\sq{1-|r|^{p-2}}\hat{G}^1_N+\sign(r)^p\sq{|r|^{p-2}}G_N\\
				\sq{1-|r|^{p-2}}\hat{G}^2_N+\sq{|r|^{p-2}}G_N\end{bmatrix},\]
			where $\hat{G}^1_N,\hat{G}^2_N,G_N$ are i.i.d matrices sampled from $GEE(\tau,N-2)$.
			\item The vectors $(V^k_N(r)_i,W^k_N(r)_i)_{k=1,2}$ for $1\le i<N-1$ are i.i.d and for fixed $i$ \[\sq{N-1}(V^1_N(r)_i,W^1_N(r)_i,V^2_N(r)_i,W^2_N(r)_i)\]
			has an $N$-independent covariance function given by the matrix $\Sigma_Z(r)$ defined in (\ref{eqn:def:SigmaZ}).
			\item $(\sq{N-1}S^k_N(r),\sq{N}U^k_N(r))_{k=1,2}$ has an $N$-independent covariance function.
		\end{itemize}
	\end{lemma}
	We note that as we will at no point need any properties of the covariance function of $(\sq{N-1}S^k_N(r),\sq{N}U^k_N(r))_{k=1,2}$ we do not derive a explicit formula for them.
	\begin{lemma}
		\label{lem:jacobian-covariance-statement-with-energy}
		Fix $r\in (-1,1)$, and $u_1,u_2\in \R$. With notation as in Lemma \ref{lem:jacobian-covariance-statement-no-energy}, let us denote \[(\Mb^1_N(r,u_1,u_2),\Mb^2_N(r,u_1,u_2))=\left(M^1_N(r),M^2_N(r)\right)\big|(U^1_N(r)=u_1,U^2_N(r)=u_2).\]
		For $k=1,2$, let us write $\Mb^k_N(r,u_1,u_2)$ in terms of square $(N-1)$-by-$1$ blocks as 
		\[\Mb^k_N(r,u_1,u_2)=\begin{bmatrix}
			\Gb^k_N(r)&  \Wb^k_N(r)\\
			\Vb^k_N(r)^t&\Sb^k_N(r,u_1,u_2)\\
		\end{bmatrix}-\hat{u}_{N,k}I.\]
		Then $(\Gb^k_N(r),\Wb^k_N(r),\Vb^k_N(r),\Sb^k_N(r,u_1,u_2))_{k=1,2}$ is the Gaussian vector determined by the following properties:
		\begin{itemize}
			\item$(\Gb^k_N(r), \Wb^k_N(r),\Vb^k_N(r))_{k=1,2}$ is independent of $(\Sb^k_N(r,u_1,u_2))_{k=1,2}$ and coincides in law with the random variable $(H^k_N(r),W^k_N(r),V^k_N(r))_{k=1,2}$ given in Lemma \ref{lem:jacobian-covariance-statement-no-energy}.
			\item  For $1\le k,l\le 2$
			\[(N-1)\E[\Sb^k_N(r,u_1,u_2)]=m^k_{1}(r)\hat{u}_{N,1}+m^k_{2}(r)\hat{u}_{N,2},\]
			\[(N-1)\E[\Sb^k_N(r,u_1,u_2)\Sb^l_N(r,u_1,u_2)]=[\Sigma_{S}(r)]_{k,l},\]
			where $m^k_l$ and $\Sigma_{S}(r)$ are given by (\ref{eqn:m-def}) and (\ref{eqn:S-def}).
		\end{itemize}
	\end{lemma}
	
	The proof of these lemmas will also be given in Appendix \ref{appendix:covariance-computation}. The proof of Lemma \ref{lem:second-moment-exact} follows immediately from Lemmas \ref{lem:Kac-Rice}, \ref{lem:conditioning-density-statement}, and \ref{lem:jacobian-covariance-statement-no-energy}.
	
	We remark that while the omission of the notation $(u_1,u_2)$ in various factors in Lemma \ref{lem:jacobian-covariance-statement-with-energy} is an abuse of notation, we have chosen this omission for terms that are unaffected by this conditioning, which will significantly simplify notation in latter sections.
	
	\section{Proof of Theorem \ref{theorem:KR-2point}\label{section:KR-2 point-old}}
	
	The primary purpose of this section is to give the proof of Theorem \ref{theorem:KR-2point}, which is our upper bound on $\E[\Crit_{N,2}(B,I)]$ on the exponential scale. This will follow analysis of the terms occurring in Lemma \ref{lem:second-moment-exact}, with the primary obstacle being the terms involving the absolute determinants.
	
	As outlined above, our method will be to rely on general concentration results, but this introduces an important technical distinction between the relaxational and non-relaxational cases, as we are not able to directly use the empirical spectral measure. The reason for this is that the eigenvalues of a general matrix are significantly less stable under perturbation than in the symmetric case (see, for example, the discussion in \cite{circular-law-exposition}).
	
	Instead we must rely on singular values, and in particular the following relationship which holds for any $n$-by-$n$ matrix $A$
	\[|\det(A-zI)|=\prod_{i=1}^{n}s_i(A-zI)=\exp\left(n\int_0^\infty \log(x)\nu_{A,z}(dx)\right).\label{eqn:girko-symmetrization}\]
	The measures $\nu_{A,z}$ satisfy few relations for distinct $z$, unlike in the case of eigenvalues where the corresponding measures are just translations of the empirical spectral measure. This complicates the analysis of $\nu_{A,z}$, and in particular their uniformity, which is explored further in Section \ref{section:elliptic-ensemble-background}.
	
	After this replacement though, treating the perturbation term $E^k_N(r,u_1,u_2)$ will be possible through the following deterministic matrix inequality. To state it, let $\epsilon>0$ and define the functions
	\[h_{\epsilon}(x)=\max(x,\epsilon),\;\;\;\log_{\epsilon}(x)=\log(h_{\epsilon}(x)).\]
	\begin{lemma}
		\label{lem:upper bound using G and W}
		Let $M$ be an $n$-by-$n$ matrix which may be written in terms of square $(n-1)$-by-$1$ blocks as
		\[M=\begin{bmatrix}
			H & W \\ V^t & S
		\end{bmatrix}.\]
		Then for any $\epsilon>0$
		\[|\det(M)|\le \epsilon^{-1}\sq{\|W\|^2+S^2}\left(\|V\|+\epsilon\right)\prod_{i=1}^{n-1}h_{\epsilon}(s_i(H)),\label{eqn:ignore-1020}\]
	\end{lemma}
	The proof of this result is delayed till the end of this section. We now note that
	\[\prod_{i=1}^{n}h_{\epsilon}(s_i(A-zI))=\exp \left(\int_0^{\infty} \log_{\epsilon}(x)\nu_{A,z}(dx)\right).\]
	We will show in Section \ref{section:elliptic-ensemble-background} that $\int_0^{\infty} \log_{\epsilon}(x)\nu_{A_N,z}(dx)$ satisfies a sub-exponential concentration estimate (see Lemma \ref{lem:log-sob bound}). Applying this, as well as some tightness results for the elliptic law, we will prove in Section \ref{section:elliptic-ensemble-background} the following result.
	
	\begin{lemma}
		\label{lem:main elliptic ensemble bound}
		For $|\tau|<1$, and any choice of $\ell,K>0$ we have that
		\[\lim_{\epsilon\to 0}\limsup_{N\to \infty}\sup_{|z|<K}\left|(N\ell)^{-1}\log(\E[\exp(\ell N \int_0^\infty \log_{\epsilon}(x)\nu_{A_N,z}(dx))])-\phi_\tau(z)\right|=0.\]
		In addition, for $0<\epsilon<1$ we have that
		\[\limsup_{N\to \infty}\sup_{|z|\ge 1}\bigg((N\ell)^{-1}\log(\E[\exp(\ell N \int_0^\infty \log_{\epsilon}(x)\nu_{A_N,z}(dx))])\]\[-(2+\log(|z|))\bigg)\le 0.\]
	\end{lemma}
	
	With these preliminaries stated, we now prepare for the proof of Theorem \ref{theorem:KR-2point}. From Lemma \ref{lem:upper bound using G and W} we see that for any $\epsilon>0$ we have that
	\[|\det(M^k_N(r))|\le \epsilon^{-1}\sq{\|W^k_N(r)\|^2+S^k_N(r)^2}\left(\|V^k_N(r)\|+\epsilon\right)\times\]\[\exp \left((N-2)\int_0^{\infty}\log_{\epsilon}(x)\nu_{H^k_N(r),\hat{U}_{k,N}(r)}(dx)\right).\label{eqn:ignore-440}\]
	
	We recall that $\sq{\frac{N-2}{N-1}}H^k_N(r)$ is distributed as $GEE(\tau,N-2)$. Thus Lemma \ref{lem:main elliptic ensemble bound} implies that for any $1>\delta>0$ and $\ell>0$
	\[\limsup_{\epsilon\to 0}\limsup_{N\to \infty}\sup_{u\in \R}\bigg(\frac{1}{N\ell }\log(\E[\exp(\ell(N-2)\int_0^{\infty}\log_{\epsilon}(x)\nu_{H^k_N(r),u}(dx))])\]
	\[-\phi_\tau^\delta(u)\bigg)\le 0,\label{eqn:main usage of elliptic}\]
	where here
	\[\phi_\tau^{\delta}(u)=\begin{cases}
		\phi_\tau(u);\;\;\;|u|\le \delta^{-1}\\
		2+\log(|u|);\;\;\; |u|>\delta^{-1}
	\end{cases}.\]
	The following result allows us to control the remaining terms of (\ref{eqn:ignore-440}), particularly as $r\to 1$.
	
	\begin{lemma}
		\label{lem:bound of edge elements, chi-stuff}
		For any $m\in \N$ there is $C$ such that for all $r\in (-1,1)$ and $N$
		\[\E\left[\|W^k_N(r)\|^{2m}\right]^{1/2m}\le C\sq{(1-r^2)},\;\;\;\E\left[|S^k_N(r)|^{2m}\right]^{1/2m}\le C\sq{(1-r^2)},\]
		\[\E\left[\|V^k_N(r)\|^{2m}\right]^{1/2m}\le C.\]
	\end{lemma}
	
	\begin{proof}[Proof of Lemma \ref{lem:bound of edge elements, chi-stuff}]
		Let $(X_i)_{i=1}$ denote a family of i.i.d standard Gaussian random variables. Then by Lemma \ref{lem:jacobian-covariance-statement-no-energy} we have that
		\[\E \left[\|W^k_N(r)\|^{2m}\right]^{1/2m}=\sigma_1(r)^{1/2}\E\left[\left(\frac{1}{N-1}\sum_{i=1}^{N-2}X_i^2\right)^{m}\right]^{1/2m},\]
		where $\sigma_1(r)$ is the covariance function of any entry of $\sq{N-1}W^k_N(r)$. An expression for this is derived in Appendix \ref{appendix:covariance-computation}, namely in equation (\ref{eqn:def:sigma1}) (see as well the computation of $\Sigma_Z$ following (\ref{eqn:ignore-plane})). We observe that
		\[\E[(\sum_{i=1}^{N-2}X_i^2)^{m}]^{1/2m}=\E[(\chi_{N-2}^2)^m]^{1/2m},\]
		where $\chi_{N-2}^2$ denotes a $\chi^2$-random variable of parameter $(N-2)$. Recalling that for $k>3$ \[\E[(\chi_{k}^2)^m]=(k-1)(k+1)\cdots (k-3+2m)\le (k+2m)^m\le  (2m)^m k^m,\label{eqn:chi-bound}\]
		we see that for $N>5$
		\[\E[\|W^k_N(r)\|^{2m}]^{1/2m}\le \sigma_1(r)^{1/2} (2m)^{1/2}.\]
		Thus to show the desired claim for $W^k_N(r)$ it suffices to show there is $C$ such that $\sigma_1(r)\le C\sq{1-r^2}$, which is among the statements of Lemma \ref{lem:appendix:covariance around 1}. A similar argument demonstrates the claims for $S^k_N(r)$ and $V^k_N(r)$.
	\end{proof}
	
	With these preliminaries we are ready to give the proof of Theorem \ref{theorem:KR-2point}.
	
	\begin{proof}[Proof of Theorem \ref{theorem:KR-2point}]	
		Let us fix $m\ge 2$ and $\epsilon>0$ and denote $\ell=\ell(m)=m/(m-1)$. By (\ref{eqn:ignore-440}) and H\"{o}lder's inequality we have that
		\[\E\left[\prod_{k=1,2}|\det(M^k_N(r))|I(U^k_N(r)\in B)\right]\le  \cal{F}_{N,\epsilon,m}(r)\cal{E}_{N,\epsilon,m}(r),\]
		where here $\cal{E}_{N,\epsilon,m}(r):=$
		\[\E\left[\prod_{k=1,2}\exp\left(\ell(N-2)\int_0^{\infty}\log_{\epsilon}(x)\nu_{H^k_N(r),\hat{U}_{k,N}(r)}(dx)\right)I(U^k_N(r)\in B)\right]^{1/\ell},\]
		\[\cal{F}_{N,\epsilon,m}(r):=\epsilon^{-1}\prod_{k=1,2}\E[(\|W^k_N(r)\|^2+S^k_N(r)^2)^{m}(\|V^k_N(r)\|+\epsilon)^{2m}]^{1/2m}.\]
		
		We observe as an immediate corollary of Lemma \ref{lem:bound of edge elements, chi-stuff} and the Cauchy-Schwarz inequality  there is $C:=C(\epsilon,m)$ such that for $r\in (-1,1)$
		\[\cal{F}_{N,\epsilon,m}(r)\le C\sq{1-r^2}.\]
		Noting as well that
		\[\lim_{r\to \pm 1}(r^p-\tau (p-1)r^{p-2}(1-r^2))^2=1,\]
		we see that the function $\cal{G}$ specified in Lemma \ref{lem:second-moment-exact} satisfies $\cal{G}(r)\le C(1-r^2)^{-1}$ for some $C$. Combining all of these facts, and possibly increasing $C$, Lemma \ref{lem:second-moment-exact} yields that
		\[\E[\Crit_{N,2}(B,I)]\le C C_N\int_{I}\left(\frac{1-r^2}{1-r^{2p-2}}\right)^{(N-2)/2}\cal{E}_{N,m,\epsilon}(r)\frac{dr}{\sq{1-r^2}}. \label{eqn:ignore-348}\]
		Moreover by H\"{o}lder's inequality (with respect to the finite measure $dr/\sqrt{1-r^2}$), and again possibly increasing $C$, this is less than
		\[CC_N\left(\int_{I}\left(\frac{1-r^2}{1-r^{2p-2}}\right)^{\ell(N-2)/2}\cal{E}_{N,m,\epsilon}(r)^{\ell}\frac{dr}{\sq{1-r^2}}\right)^{1/\ell}\]
		Conditioning on $(U^1_N(r),U^2_N(r))$, we see by (\ref{eqn:main usage of elliptic})  that for any fixed $1>\delta>0$, sufficiently small $\epsilon$ and large $N$, we have that
		\[\cal{E}_{N,m,\epsilon}(r)\le e^{N\delta}\E \left[\prod_{k=1,2}\exp\left(\ell(N-2)\phi_{\tau}^\delta(\hat{U}^k_N(r))\right)I(U^k_N(r)\in B)\right]^{1/\ell}.\label{eqn:ignore-349}\]
		With this we define the integral
		\[I_{N,\ell,\epsilon,\delta}:=\int_I \exp(\ell N h(r))\E[\prod_{k=1,2}\exp\left(\ell N\phi_{\tau}^\delta(\hat{U}^k_{N+2}(r))\right)I(U^k_{N+2}(r)\in B)]\frac{dr}{\sq{1-r^2}},\]
		where $h$ was defined in (\ref{def:h-def}) above. With this we see that
		\[\limsup_{N\to \infty}N^{-1}\log(\E[\Crit_{N,2}(B,I)])\le \]
		\[\delta+\limsup_{N\to \infty}N^{-1}\log(C_N)+\limsup_{N\to \infty}(N\ell)^{-1}\log(I_{N-2,\ell,\epsilon,\delta}).\label{eqn:ignore-upp-1}\]
		We note that by Stirling's approximation
		\[\lim_{N\to \infty}N^{-1}\log(C_N)=\log(p-1)+1.\label{eqn:ignore-stir-weak}\]
		Thus we need only obtain asymptotics for $I_{N,\ell,\epsilon,\delta}$, which will follow by Varadhan's Lemma (see Theorem 4.3.1 and Lemma 4.3.6 of \cite{Varadhan}). Let us denote by $(X_1,X_2)$ a pair of standard Gaussian random variables and $R$ an independent random variable on $[-1,1]$ with density $\pi^{-1}(1-r^2)^{-1/2}$. We define
		\[X_{i,N}=(N+2)^{-1/2}X_i,\;\;\hat{X}_{i,N}=(N+2)^{1/2}((N+1)p(p-1))^{-1/2}X_i.\]
		and observe the distributional equality
		\[(U^1_{N+2}(r),U^2_{N+2}(r))\stackrel{d}{=}\left(X_{1,N},X_{2,N}\right)\Sigma_U(r)^{1/2},\]
		and where $\Sigma_U(r)$ is defined in (\ref{eqn:def:U-external field form}). A similar result holds for $\hat{U}^k_{N+2}$. If we further denote
		\[W:=\{(u_1,u_2,r)\in \R^2\times I:(u_1,u_2)\Sigma_U(r)^{1/2}\in B^2\},\]
		then we see that
		\[I_{N,\ell,\epsilon,\delta}=\E\big[\exp\left(\ell N[h(R)+ \sum_{k=1,2}\phi_{\tau}^\delta\left(\left[(\hat{X}_{1,N},\hat{X}_{2,N})\Sigma_U(R)^{1/2}\right]_k\right)]\right)\times\]
		\[I\left((X_{1,N},X_{2,N},R)\in W\right)\big].\]
		We now observe that the random variable $(\hat{X}_{1,N},\hat{X}_{2,N},X_{1,N},X_{2,N},R)$ satisfies a LDP on $\R^4\times [-1,1]$ with good rate function \[J(v_1,v_2,u_1,u_2,r)=\begin{cases}
			\frac{u_1^2}{2}+\frac{u_2^2}{2};\;\;\;v_1=\sq{p(p-1)}u_1,\; v_2=\sq{p(p-1)}u_2\\
			\infty;\;\;\;\; \text{otherwise}
		\end{cases}.\]
		To apply Varadhan's Lemma, we will need to first verify a moment condition. For this we note that there is $C>0$ such that for $x\in \R$
		\[\phi^{\delta}_\tau(x)\le C(1+\log(1+|x|)).\]
		As the entries of $\Sigma_U(r)$ are uniformly bounded in $r\in [-1,1]$, we see that there is $C$ such that for any $(u_1,u_2,r)\in \R^2\times [-1,1]$
		\[\sum_{k=1,2}\phi_{p,\tau}^\delta([(u_1,u_2)\Sigma_U(r)^{1/2}]_k)\le C(1+\log(1+\|(u_1,u_2)\|)).\]
		Observing as well that $h(r)$ is bounded on $[-1,1]$, we see that there is $C$ such for any $\ell'>0$
		\[\limsup_{N\to \infty}N^{-1}\log(\E[\exp(\ell' N[ h(R)+\sum_{k=1,2}\phi_{\tau}^\delta\left(\left[(\hat{X}_{1,N},\hat{X}_{2,N})\Sigma_U(R)^{1/2}\right]_k\right)]\le\]\[ \limsup_{N\to \infty}N^{-1}\log(\E[\exp(\ell' CN(1+2 \log(1+\|(X_1,X_2)\|/\sq{N})))])<\infty.\]
		With this moment condition verified, we may now apply Varadhan's Lemma to conclude that
		\[\limsup_{N\to \infty}N^{-1}\log(I_{N,\ell,\epsilon,\delta})\le\]\[ \sup_{(u_1,u_2,r)\in W} \ell h(r)-\frac{u_1^2}{2}-\frac{u_2^2}{2}+\ell \sum_{k=1,2}\phi_{p,\tau}^\delta([(u_1,u_2)\Sigma_U(R)^{1/2}]_k)=\]
		\[\sup_{(u_1,u_2)\in B_{\epsilon}(B),r\in I} \ell h(r)-\frac{1}{2}(u_1,u_2)\Sigma_U(r)^{-1}\begin{bmatrix}u_1\\ u_2 \end{bmatrix}+\ell \phi_{p,\tau}^{\delta}(u_1)+\ell \phi_{p,\tau}^{\delta}(u_2).\label{eqn:ignore-upper-2}\]
		Thus taking $m\to \infty$ (so $\ell\to 1$) then $\epsilon\to 0$, and finally $\delta\to 0$ we see that equations (\ref{eqn:ignore-upp-1}), (\ref{eqn:ignore-stir-weak}) and (\ref{eqn:ignore-upper-2}) complete the proof.
	\end{proof}
	
	With the proof of Theorem \ref{theorem:KR-2point} completed, we finish this section with the proof of Lemma \ref{lem:upper bound using G and W}.
	
	\begin{proof}[Proof of Lemma \ref{lem:upper bound using G and W}]
		Let us denote
		\[Y=\begin{bmatrix}
			H & 0\\ V^t & 0
		\end{bmatrix}.\]
		The difference $M-Y$ has a single nonzero singular value given by $\sq{\|W\|^2+S^2}$. We note as well that $s_n(Y)=0$. Thus by applying Corollary \ref{corr:determinant of sum bound} we see that
		\[|\det(M)|\le \sq{\|W\|^2+S^2}\prod_{i=1}^{n-1}s_i(Y).\]
		We observe that
		\[\prod_{i=1}^{n-1}s_i(Y)^2=\lim_{\delta\to 0}\frac{1}{\delta}\det(Y^tY+\delta I)=\lim_{\delta\to 0}\det(H^tH+VV^t+\delta I)=\det(H^tH+VV^t).\]
		By applying Corollary \ref{corr:determinant of sum bound} again we see that 
		\[|\det(H^tH+VV^t)|\le (s_{n-1}(H)+\|V\|)\prod_{i=1}^{n-2}s_i(H)\le \epsilon^{-1}(\|V\|+\epsilon)\prod_{i=1}^{n-1}h_{\epsilon}(s_i(H)).\]
		Where in the last inequality we have used that for $x,y\in (0,\infty)$
		\[x+y\le h_{\epsilon}(x)\epsilon^{-1}(\epsilon+y).\]
		Together these statements establish (\ref{eqn:ignore-1020}).
	\end{proof}
	\section{Proof of Lemmas \ref{lem:stable points}, \ref{lem:ratio of second and first moment}, \ref{lem:moments of M} and \ref{lem:delta-estimates-near-zero}\label{section:theorem-O1-section}}
	
	In this section we will give the results needed to prove Proposition \ref{prop:constant order concentration bulk}, which involve the computation of the second moment at the constant order scale. We will begin by giving the proof of Lemmas \ref{lem:moments of M} and \ref{lem:delta-estimates-near-zero}, and then proceed to the proofs of Lemmas \ref{lem:stable points} and \ref{lem:ratio of second and first moment}.
	
	To begin we need the following basic concentration result, whose proof we defer to the end of the section.
	\begin{lemma}
		\label{lem:concentration lemma XY}
		Fix a Gaussian pair $(X,Y)\disteq \cal{N}(0,\Sigma)$. For each $N$, let $(X_N,Y_N)\in \R^N\times \R^N$ be the random vector whose entries are i.i.d copies of $N^{-1/2}(X,Y)$. Let $f:\R^3\to \R$ be any continuous function of polynomial growth. Then we have that
		\[\lim_{N\to \infty}\E[f(\|X_N\|,\|Y_N\|,(X_N,Y_N)^2)]=f(\Sigma_{11},\Sigma_{22},\Sigma_{12}^2).\]
	\end{lemma}
	
	We now proceed to the proof of Lemma \ref{lem:moments of M}. As in the proof of Lemma 8 in \cite{subag-ofer-new} in the relaxational case, we first separate the perturbation term by taking a cofactor expansion, and then simplify these terms by using rotation invariance. In the symmetric case, rotation invariance is quite strong, and this gives a direct recurrence in terms of the characteristic polynomials of smaller size. In our case however, we instead rely on direct comparison of the terms in this expansion for both $\bar{M}^k_N(r,u_1,u_2)$ and $A_{N-1}-\hat{u}_{N,k}I$.
	
	\begin{proof}[Proof of Lemma \ref{lem:moments of M}]
		We will fix and omit the choice of $(k,r,u_1,u_2)$ from the notation for clarity. Additionally, we denote $v=\hat{u}_{N,k}$. With this notation, the decomposition of $\Mb_N$ in Lemma \ref{lem:jacobian-covariance-statement-with-energy} in terms of square $(N-1)$-by-$1$ blocks is given as
		\[\Mb_N=\begin{bmatrix}
			\Hb_N& \Wb_N\\
			\Vb_N^t & \Sb_N
		\end{bmatrix}-vI.\]
		Let $A_{N-1}$ be sampled from $GEE(\tau,N-1)$. To compare the two, we will also write $A_{N-1}$ in terms of square $(N-1)$-by-$1$ blocks as
		\[A_{N-1}=\begin{bmatrix}
			\hat{H}_N& \hat{W}_N\\
			\hat{V}_N^t & \hat{S}_N
		\end{bmatrix}.\label{eqn:ignore-mad}\]
		By Lemma \ref{lem:jacobian-covariance-statement-with-energy} we see that $\Hb_N\disteq \hat{H}_N$, though the laws for the other elements are in general different. To compare these we will expand both in similar terms starting with $\Mb_N$. We first note that a.s.
		\[\det(\Mb_N-vI)=\det(\Hb_N-vI)\left(\Sb_N-v-\Vb_N^t(\Hb_N-vI)^{-1}\Wb_N\right).\]
		For convenience, let us denote $\Gb_N=(\Hb_N-vI)^{-1}$. Then
		using the inequality $(x+y)^\ell\le 2^\ell (x^\ell+y^\ell)$, and noting that $\Sb_N$ is independent of $(\Hb_N,\Vb_N,\Wb_N)$, we see that
		\[
		2^{-\ell}\E[|\det(\Mb_N-vI)|^{\ell}]\le \E[|\det(\Hb_N-vI)|^{\ell}]\E[|\Sb_N-v|^{\ell}]\]\[+\E[|\det(\Hb_N-vI)|^{\ell}|\Vb_N^t\Gb_N\Wb_N|^{\ell}].\label{eqn:ignore-plane-2}\]
		The first term is bounded by $C\E[|\det(\Hb_N-vI)|^{\ell}]$, so we focus on the second. Let us denote 
		\[\bar{R}_N=(\Vb_N,\Wb_N),\;\;\;\bar{Q}_N=\sq{\|\Vb_N\|^2\|\Wb_N\|^2-(\Vb_N,\Wb_N)^2}.\]
		As the law of $\Hb_N$ is invariant under orthogonal conjugation, and $(\Vb_N,\Wb_N)$ is independent of $\Hb_N$, we see we may choose $U\in O(N)$, independent of $\Hb_n$, such that
		\[U\Vb_N=(\|\Vb_N\|,0,\dots 0),\;\;\;\; U\Wb_N=\|\Vb_N\|^{-1}(\bar{R}_N,\bar{Q}_N,0,\dots 0).\]
		Applying $U$ and conjugation invariance of $\Hb_N$ (and thus $\Gb_N$) we see that
		\[\E[|\det(\Hb_N-vI)|^{\ell}|\Vb_N^t\Gb_N\Wb_N|^{\ell}]=\E[|\det(\Hb_N-vI)|^{\ell}|(U\Vb_N)^t\Gb_NU\Wb_N|^{\ell}]=\]
		\[\E[|\det(\Hb_N-vI)|^{\ell}|[\Gb_N ]_{11}\bar{R}_N+[\Gb_N]_{12}\bar{Q}_N|^\ell|].\]
		We will now use that $\ell$ is even, so that we may expand this using the binomial theorem as
		\[
		\sum_{k=0}^{\ell}\binom{\ell}{k}\E[|\det(\Hb_N-vI)|^{\ell}|[\Gb ]_{11}^{\ell-k}[\Gb]_{12}^{k}]\E[\bar{R}_N^{\ell-k}\bar{Q}_N^{k}],\label{eqn:ignore-diamond}
		\]
		where we have used the independence of $(\bar{V}_N,\bar{U}_N)$, and thus $(\bar{R}_N,\bar{Q}_N)$, from $\bar{G}_N$. As $(\bar{V}_N,\bar{W}_N)\disteq (-\bar{V}_N,\bar{W}_N)$ we see that the terms when $k$ is odd must vanish, so we may write (\ref{eqn:ignore-diamond}) as	\[\sum_{k=0}^{\ell/2}\binom{\ell}{2k}\E[|\det(\Hb_N-vI)|^{\ell}|[\Gb_N ]_{11}^{\ell-2k}[\Gb_N]_{12}^{2k}]\bar{\Delta}_N^k,\;\;\;\;\;\bar{\Delta}_N^k:=\E[\bar{R}_N^{\ell-2k}\bar{Q}_N^{2k}].\]
		Now the argument used above to show (\ref{eqn:ignore-plane-2}) may be easily modified to show
		\[2^{-\ell}\E[|\det(\hat{H}_N-vI)|^{\ell}|\hat{V}_N^t\hat{G}_N\hat{W}_N|^{\ell}]\le \]
		\[\E[|\det(A_{N-1}-vI)|^{\ell}]+\E[|\det(\hat{H}_N-vI)|^\ell]\E[(\hat{S}_N-u)^\ell]\label{eqn:ignore-104},\]
		where $\hat{G}_N$ is defined analogously. Then proceeding similarly again, we achieve as well
		\[
		\E[|\det(\hat{H}_N-vI)|^{\ell}|\hat{V}_N^t\hat{G}_N\hat{W}_N|^{\ell}]=\]\[
		\sum_{k=0}^{\ell/2}\binom{\ell}{2k}\E[|\det(\hat{H}_N-vI)|^{\ell}|[\hat{G}_N ]_{11}^{\ell-2k}[\hat{G}_N]_{12}^{2k}|]\hat{\Delta}_N^k,
		\]
		where again $\hat{\Delta}_N$ is defined analogously. Now as $\Hb_N\disteq\hat{H}_N$ we see that
		\[\E[|\det(\hat{H}_N-vI)|^{\ell}|[\hat{G}_N ]_{11}^{\ell-2k}[\hat{G}_N]_{12}^{2k}]=\E[|\det(\Hb_N-vI)|^{\ell}|[\Gb_N ]_{11}^{\ell-2k}[\Gb_N]_{12}^{2k}|].\]
		Moreover, as $\ell$ is even, we have that $\bar{\Delta}_N^k,\hat{\Delta}_N^k\ge 0$. Thus comparing the determinants has essentially been reduced to comparing $\bar{\Delta}_N^k$ and $\hat{\Delta}_N^k$. By applying Lemma \ref{lem:concentration lemma XY} we see that $\lim_{N\to \infty}\hat{\Delta}_N^k=\tau^{\ell-2k}(1-\tau^2)^{k}$ and 
		\[\lim_{N\to \infty}\bar{\Delta}_N^k=[\Sigma_Z(r)]_{12}^{\ell-2k}([\Sigma_Z(r)]_{11}[\Sigma_Z(r)]_{22}-[\Sigma_Z(r)]_{12}^2)^{k}.\]
		We note that $\Sigma_Z(r)$ is bounded in $r$, and so as $\tau\in(0,1)$ we may find $C>0$ such that
		\[\max_{1\le k\le \ell}\bar{\Delta}_N^k/\hat{\Delta}_N^k\le C.\]
		Enlarging $C$ we may also assume that 
		\[\E[|\bar{S}_N-v|^{\ell}]+\E[|\hat{S}_N-v|^{\ell}]+1\le C.\]
		Combing both of the expressions with the expression for $\det(\bar{M}_N)$ and $\det(A_{N-1}-vI)$ we see that
		\[4^{-\ell}\E[|\det(\bar{M}_N)|^\ell]\le C^2\left(\E[|\det(\hat{H}_N-vI)|^{\ell}]+\E[|\det(A_{N-1}-vI)|^{\ell}]\right).\label{eqn:ignore-109}\]
		The submatrix $\hat{H}_N$ is proportional to a matrix sampled from $GEE(\tau,N-2)$, so that
		\[\E[|\det(\hat{H}_N-vI)|^{\ell}]=\left(\frac{N-2}{N-1}\right)^{\ell(N-2)/2}\E[|\det(A_{N-2}-\sq{(N-1)/(N-2)}vI)|^\ell].\]
		By applying Theorem 1 of \cite{PaxComp} in the case where $|v|<\Einf$ and Proposition \ref{prop:moments of character outside bulk} in the case where $|v|>\Einf$, we see that there is $C$ such that
		\[\E[|\det(A_{N-2}-\sq{(N-1)/(N-2)}vI)|^\ell]\le C\E[|\det(A_{N-1}-vI)|^\ell].\]
		Applying this and Theorem \ref{theorem:moments of character main} to (\ref{eqn:ignore-109}) completes the proof.
	\end{proof}
	
	Before proceeding to the proof of Lemma \ref{lem:delta-estimates-near-zero}, we will need the following precise form of the decomposition given in (\ref{eqn:decom-E}). 
	
	\begin{lemma}
		\label{lem:perturbation lemma for M}
		Let us fix three i.i.d matrices sampled from $GEE(\tau,N-1)$ and denote them \[(A_{N-1}^1,A_{N-1}^2,A_{N-1}).\] Define for $r\in (-1,1)$ and $k=1,2$
		\[A_{N-1}^k(r):=\sq{1-|r|^{p-2}}A_{N-1}^k+\sign(r)^{pk}\sq{|r|^{p-2}}A_{N-1}.\]
		Then there are $(E_N^k(r,u_1,u_2))_{k=1,2}$, correlated to $(A_{N-1},A_{N-1}^1,A_{N-1}^2)$, such that a.s. $\rank(E_N^k(r,u_1,u_2))\le 2$ and which satisfy
		\[(\bar{M}_N^k(r,u_1,u_2))_{k=1,2}\disteq (A_{N-1}^k(r)+E_N^k(r,u_1,u_2)-\hat{u}_{N,k}I)_{k=1,2}.\label{eqn:E-perp-exact}\]
		Additionally there are $c,C>0$ such that
		\[\P \left(\|E_N^k(r,u_1,u_2)\|\ge C(|r|^{p-3}+N^{-1/2}|r|^{p-4})^{1/2}\right)\le Ce^{-cN^{1/2}},\label{eqn:E-van-rate}\]
		and moreover $c,C$ can be chosen so that this holds over any compact subset of $(r,u_1,u_2)\in (-1,1)\times \R^2$.
	\end{lemma}
	
	The construction of $E_N^k(r,u_1,u_2)$ is given at the end of this section, as is the proof of this lemma. It is primarily reliant on the technical computations from Appendix \ref{appendix:covariance-computation}.
	
	We now give the proof of Lemma \ref{lem:delta-estimates-near-zero}. As this provides asymptotics of the second moment up to vanishing order, we must quite carefully control the interplay between the fluctuations of the determinant and the strength of the correlations between the two matrices. The general strategy will be to introduce a series of rare events, and show that away from these events the determinant may be bounded by a perturbative argument using Corollary \ref{corr:determinant of sum bound}. We then show that the contributions coming from each of these events are negligible by estimating their probabilities, and then applying moment estimates. The primary tool will be Lemma \ref{lem:moments of M} as before.
	
	\begin{proof}[Proof of Lemma \ref{lem:delta-estimates-near-zero}]
		As in the proof of Lemma \ref{lem:moments of M}, it suffices to show this for a fixed sequence $(r_N,u_N^1,u_N^2)$, such that $u_N^1,u_N^2\to u$ and such that $|r_N|<C\sq{\log(N)/N}$, which we will omit from the notation. Through the proof we will take constants $C,c>0$, with $C$ assumed to be large, and $c$ taken to be sufficiently small, and both of which are allowed to become larger (or smaller) line by line, but always independently of $N$.
		
		With the notation of Lemma \ref{lem:perturbation lemma for M}, we first define the event
		\[\cal{E}_{\text{large}}=\{s_1(A_{N-1})>C\}\cup \]\[\bigcup_{k=1,2}\left\{\|E_N^k\|>C\left(|r_N|^{(p-3)}+N^{-1/2}|r_N|^{p-4}\right)^{1/2}\right\}\cup \{s_1(A_{N-1}^k)>C\}.\]
		We also define for $k=1,2$ the events
		\[\cal{E}_{\text{small},k}=\{s_{N-1}(A_{N-1}^k-\hat{u}_{N,k}I)<c |r_N|^{(p-3)/2-c}\}.\]
		Noting that
		\[s_{N-1}(A_{N-1}^k(r_N)-\hat{u}_{N,k}I)\ge s_{N-1}(A_{N-1}^k-\hat{u}_{N,k}I)\]
		\[-\left(1-(1-|r_N|^{p-2})^{1/2}\right)s_1(A_{N-1}^k)-|r_N|^{(p-2)/2}s_1(A_{N-1}),\]
		we see that on $\cal{E}_{\text{large}}^c\cap \cal{E}_{\text{small},k}^c$, and for large enough $N$, that
		\[s_{N-1}(A_{N-1}^k(r_N)-\hat{u}_{N,k}I)>c|r_N|^{(p-3)/2-c}.\]
		By Corollary \ref{corr:determinant of sum bound} we have that for $k=1,2$
		\[|\det(\bar{M}_N^k)|\le |\det(A_{N-1}^k(r_N)-\hat{u}_{N,k}I)|\left(1+\frac{\|E_N^k\|}{s_{N-1}(A_{N-1}^k(r)-\hat{u}_{N,k}I)}\right)^2,\]
		Thus conditional on the event $\cal{E}_{\text{large}}^c\cap \cal{E}_{\text{small},1}^c\cap \cal{E}_{\text{small},2}^c$ we have that
		\[|\det(\bar{M}_N^k)|\le |\det(A_{N-1}^k(r_N )-\hat{u}_{N,k}I)|\left(1+C|r_N|^{c/2}\right)^2=\]
		\[|\det(A_{N-1}^k(r_N)-\hat{u}_{N,k}I)|(1+o(1)).\]
		To further simplify the term on the right, we will introduce our final event. Let us denote $K=[2/c]^{-1}$ and $\epsilon=K^{-1}$ so that $\epsilon<c$. We define
		\[\cal{E}_{\text{medium}}=\bigcup_{k=1,2}\bigcup_{m=0}^{K-1}\{s_{N-[N^{\epsilon m}]+1}(A_{N-1}^k-\hat{u}_{N,k}I)>cN^{\epsilon m-1}\}.\]
		We observe that for $1\le m\le K$, the number of $i$ such that $N-[N^{\epsilon (m-1)}]<i\le  N-[N^{\epsilon m}]$ is less than $N^{\epsilon m}$. Then applying Corollary \ref{corr:determinant of sum bound} again we see that
		\[\frac{|\det(A_{N-1}^k(r_N)-\hat{u}_{N,k}I)|}{|\det(A_{N-1}^k-\hat{u}_{N,k}I)|}\le\prod_{i=1}^{N-1}\left(1+\frac{s_i(A_{N-1}^k(r_N)-A_{N-1}^k)}{s_{N-i}(A_{N-1}^k-\hat{u}_{N,k}I)}\right)\]\[ \prod_{i=1}^{N-1}\left(1+\frac{s_1(A_{N-1}^k(r_N)-A_{N-1}^k)}{s_{N-i}(A_{N-1}^k-\hat{u}_{N,k}I)}\right)\le\]\[\prod_{m=0}^{K-1}\left(1+\frac{s_1(A_{N-1}^k(r_N)-A_{N-1}^k)}{s_{N-[N^{\epsilon m}]+1}(A_{N-1}^k-\hat{u}_{N,k}I)}\right)^{N^{\epsilon (m+1)}}.\]
		We note that on $\cal{E}_{\text{large}}^c$ we have that
		\[s_1(A_{N-1}^k(r_N)-A_{N-1}^k)\le \left(1-(1-|r_N|^{p-2}_N)^{1/2}\right)s_1(A_{N-1}^k)\]
		\[+|r_N|^{(p-2)/2}s_1(A_{N-1})\le C|r_N|^{(p-2)/2}.\]
		Thus conditional on the event $\cal{E}_{\text{large}}^c\cap \cal{E}_{\text{medium}}^c\cap \cal{E}_{\text{small},k}^c$ we have that
		\[\frac{|\det(A_{N-1}^k(r_N)-\hat{u}_{N,k}I)|}{|\det(A_{N-1}^k-\hat{u}_{N,k}I)|}\le \]\[C\left(1+C|r_N|^{1/2+c}\right)^{N^{\epsilon}}\times \prod_{m=1}^{K-1} (1+C|r_N|^{(p-2)/2}N^{1-\epsilon m})^{N^{\epsilon (m+1)}}.\]
		As $|r_N|<C\sqrt{\log(N)/N}$, we see that the first term is $1+o(1)$. Similarly, as $(p-2)/4>1+\epsilon$, we see that $|r_N|^{(p-2)/2}N^{1+\epsilon}\to 0$, so that each term in the product on the right is $1+o(1)$. Combining these claims we see that for $k=1,2$ and on $\cal{E}_{\text{large}}^c\cap \cal{E}_{\text{medium}}^c \cap\cal{E}_{\text{small},1}^c\cap \cal{E}_{\text{small},2}^c$ 
		\[|\det(\bar{M}_N^k)|\le |\det(A_{N-1}^k-\hat{u}_{N,k}I)|(1+o(1)).\]
		In particular, we see that to prove both of our desired claims it is sufficient to bound the contribution to the determinant coming from each of these events. 
		
		To begin we will recall some bounds on the probabilities for these events coming from Section \ref{section:elliptic-ensemble-background}. By Lemma \ref{lem:misc tail results on elliptic} we see that $\P(s_1(A_{N-1})>C)\le Ce^{-Nc}$, and similarly for $A_{N-1}^k$. Combining this with the final claim of Lemma \ref{lem:perturbation lemma for M} we conclude that $\P(\cal{E}_{\text{large}})\le Ce^{-cN^{1/2}}$. By applying Lemma \ref{lem:appendix:weird rare outside} we see as well that $\P(\cal{E}_{\text{medium}})\le Ce^{-cN^\epsilon}$. Finally by Lemma \ref{lem:naomov smallest eigenvalue bound}, for $\ell=1,2$ we have that
		\[\P(\cal{E}_{\text{small},\ell})\le C\left(Nr^{(p-3)/2-c}_N+e^{-cN}\right).\label{eqn:ignore-1229}\]
		
		Now to first let us bound the contribution from an arbitrary event $\cal{A}$. We see from the Cauchy-Schwarz inequality that 
		\[\E\left[\prod_{k=1,2}|\det(\bar{M}_N^k)|I_{\cal{A}}\right]\le\P(\cal{A})^{1/2}\E[\prod_{k=1,2}|\det(\bar{M}_N^k)|^2]^{1/2}\le \]\[\P(\cal{A})^{1/2}\prod_{k=1,2}\E\left[|\det(\bar{M}_N^k)|^4\right]^{1/4}.\]
		Employing Lemma \ref{lem:moments of M} we have that
		\[\prod_{k=1,2}\E[|\det(\bar{M}_N^k)|^4]^{1/4}\le CN^{3/2}\prod_{k=1,2}\E[|\det(A_{N-1}^k-\hat{u}_{N,k}I)|].\]
		In particular, if $N^{3/2}\P(\cal{A})^{1/2}\to 0$ we have that
		\[\limsup_{N\to \infty}\left(\frac{\E[\prod_{k=1,2}|\det(\bar{M}_N^k)|I_{\cal{A}}]}{\prod_{k=1,2}\E[|\det(A_{N-1}^k-\hat{u}_{N,k}I)|]}\right)=0.\]
		This holds for both $\cal{A}=\cal{E}_{\text{large}}$ and $\cal{A}=\cal{E}_{\text{medium}}$ in either case, so all that remains is to treat the case of $\cal{E}_{\text{small},\ell}$ for $\ell=1,2$. For this note that for an event $\cal{A}$, which is only dependant on $A_{N-1}^1$, we have that
		\[\E\left[\prod_{k=1,2}|\det(\bar{M}_N^k)|I_{\cal{A}}\right]\le \E[|\det(\bar{M}_N^1)|^2]^{1/2}\E[|\det(\bar{M}_N^2)|^2I_{\cal{A}}]^{1/2}=\]
		\[\left(\prod_{k=1,2}\E[|\det(\bar{M}_N^k)|^2]^{1/2}\right)\P(\cal{A})^{1/2}\le\]\[ CN^{1/4}\P(\cal{A})^{1/2}\prod_{k=1,2}\E[|\det(A_{N-1}^k-\hat{u}_{N,k}I)|]\]
		where in the second to last equality we have employed the independence of $A_{N-1}^1$ from $\bar{M}_N^2$. A similar inequality holds for events only involving $A_{N-1}^2$. In particular, when the event only depends on a single $A_{N-1}^k$, it suffices to show that $N^{1/2}\P(\cal{A})\to 0$. In view of (\ref{eqn:ignore-1229}) this holds for $\cal{E}_{\text{small},k}$ as $\frac{1}{2}+1-\frac{(p-3)}{4}<0$ for $p>9$.
	\end{proof}
	
	With this done, we will now proceed to the proof of Lemma \ref{lem:ratio of second and first moment}, which gives the expansion of the expressions for the first and second moments around $r=0$.
	
	\begin{proof}[Proof of Lemma \ref{lem:ratio of second and first moment}]
		By employing the expression for $\E[\Crit_{N,2}(B,I)]$ given in Lemma \ref{lem:second-moment-exact} and the expression for $\E[\Crit_N(B)]$ given in Lemma \ref{lem:first-moment-exact} we see that
		\[\frac{\E[\Crit_{N,2}(B,I)]}{\E[\Crit_N(B)]^2}=\bar{C}_N^{-2}C_N\int_I\left(\frac{1-r^2}{1-r^{2p-2}}\right)^{\frac{N-2}{2}}\cal{G}(r) R_N(r)dr,\label{eqn:ignore-606-2},\]
		\[R_N(r):=\frac{\E\big[\prod_{k=1,2}|\det(M_N^k(r))|I(U_N^k(r)\in B)]}{\E[|\det(A_{N-1}-\hat{X}_N I)|I(X_N\in B)]^2}.\]
		We note \[\bar{C}_N^{-2}C_N=\vol(S^{N-2})/\vol(S^{N-1}).\]
		By Stirling's approximation we see that
		\[\vol(S^{N-2})/\vol(S^{N-1})=(1+O(N^{-1}))\sq{N/2\pi}.\]
		We also note that as $p\ge 3$, $\cal{G}(r)=1+O(r^2)$. Thus (\ref{eqn:ignore-606-2}) can be rewritten as
		\[(1+O(N^{-1}))\sq{\frac{N}{2\pi}}\int_{(-\rho_N,\rho_N)} e^{-N(\frac{r^2}{2}+O(r^3))+O(r^2)}R_N(r)dr.\label{eqn:ignore-720}\]
		To compute $R_N(r)$ we recall that $\sq{N}(U_N^1(r),U_N^2(r))$ is centered Gaussian vector with covariance matrix $\Sigma_U(r)$, whose inverse is computed in (\ref{eqn:def:U-external field form}). We note that $k_2(r)=O(r^p)$, $k_1(r)-b(r)=O(r^{2p-2})$, and $b(r)=1+O(r^{2p-2})$ so that \[\Sigma_U(r)^{-1}=\frac{1}{p\alpha b(r)}\begin{bmatrix}
			k_1(r)& k_2(r)\\ 
			k_2(r)& k_1(r)
		\end{bmatrix}=(p\alpha)^{-1}I+O(r^{p}).\]
		In particular the density of the event $(U_N^1(r)=u_1,U_N^2(r)=u_2)$ is given by
		\[\frac{1}{\sq{(2\pi)^2\det(\Sigma_U(r))}}\exp\left(-\frac{N}{2}(u_1,u_2)\Sigma_U(r)^{-1}\begin{bmatrix}u_1\\u_2\end{bmatrix}\right)=\]\[\frac{1+O(r^p)}{\sq{(2\pi)^2(p\alpha)^2}}\exp\left(-\frac{N}{2p\alpha}(u_1^2+u_2^2)(1+O(r^p))\right).\]
		So we see conditioning on the energy value in both the numerator and the denominator in $R_N(r)$ that
		\[R_N(r)= e^{-NO(r^3)}\frac{\int_{B\times B}\E\big[\prod_{k=1,2}|\det(\bar{M}_N^k(r,u_1,u_2))|]e^{-\frac{N}{2p\alpha}u_k^2}du_1du_2}{\int_{B\times B}\prod_{k=1,2}\E[|\det(A_{N-1}-\hat{u}_{N,k}I)|]e^{-\frac{N}{2p\alpha}u_k^2}du_1du_2}.\label{eqn:ignore-1800}\]
		We recall the fact that for two arbitrary positive functions $f$ and $g$ on some Euclidean subset $B$, we have that
		\[\frac{\int_{B}f(u)du}{\int_{B}g(u)du}\le \sup_{u\in B}\frac{f(u)}{g(u)}.\label{eqn:ignore-1801}\]
		Thus inserting (\ref{eqn:ignore-1800}) into (\ref{eqn:ignore-720}) and maximizing over the values of $(u_1,u_2)$ using (\ref{eqn:ignore-1801}) completes the proof.
	\end{proof}
	
	Now give a proof of Lemma \ref{lem:stable points}, which is relatively straightforward given Theorem \ref{theorem:moments of character main}.
	
	\begin{proof}[Proof of Lemma \ref{lem:stable points}]
		For the duration of this proof we will adopt the notation \[\Crit_{N,unst}(B)=\Crit_{N}(B)-\Crit_{N,st}(B)\]
		for the total number of unstable points, so that we must show that for $u>\Einf$
		\[\limsup_{N\to \infty}N^{-1}\log\left(\frac{\E[\Crit_{N,unst}((u,\infty))]}{\E[\Crit_{N}((u,\infty))]}\right) <0.\]
		We first observe that
		\[\frac{\E[\Crit_{N,unst}((u,\infty))]}{\E[\Crit_{N}((u,\infty))]}\le \frac{\E[\Crit_{N,unst}((u,\infty))]}{\E[\Crit_{N}((u,u+1))]}\le\]
		\[\frac{\E[\Crit_{N}([u+1,\infty))]+\E[\Crit_{N,unst}((u,u+1))]}{\E[\Crit_{N}((u,u+1))]}.\]
		To study the first term in the right-most expression, we note that as $\Sigma^{p,\tau}$ is strictly decreasing on $[0,\infty)$ (see Remark \ref{remark:strict increasing of phi}) we have by Theorem \ref{theorem:fyodorov 1st moment} that
		\[\liminf_{N\to \infty}N^{-1}\log\left(\frac{\E[\Crit_{N}((u,u+1))]}{\E[\Crit_{N}([u+1,\infty)]}\right)>0.\]
		This shows the first term is exponentially small, so in particular it suffices to show that
		\[\limsup_{N\to \infty}N^{-1}\log\left(\frac{\E[\Crit_{N,unst}((u,u+1))]}{\E[\Crit_{N}((u,u+1))]}\right)<0.\label{eqn:ignore-1218}\]
		By conditioning on $X_N$ in Lemma \ref{lem:first-moment-exact} we have that
		\[\E[\Crit_N((u,u+1))]=\bar{C}_N\sq{\frac{N}{p\alpha}}\int_{u}^{u+1} \E[|\det(A_{N-1}-\hat{v}_{N}I)|]e^{-N\frac{v^2}{2p\alpha}}dv.\]
		A similar application of the Kac-Rice formula for stable points gives
		\[\E[\Crit_{N,st}((u,u+1))]=\]
		\[\bar{C}_N\sq{\frac{N}{p\alpha}}\int_{u}^{u+1}\E[|\det(A_{N-1}-\hat{v}_{N}I)|I(\Re(\lam_{N-1}(A_{N-1}))\le   \hat{v}_N)]e^{-N\frac{v^2}{2p\alpha}}dv,\]
		where as above, $\lam_{N-1}(A_{N-1})$ is the eigenvalue with the largest real part. In particular, we may write 
		\[\E[\Crit_{N,unst}((u,u+1))]=\]
		\[\bar{C}_N\sq{\frac{N}{p\alpha}}\int_{u}^{u+1}\E[|\det(A_{N-1}-\hat{v}_{N}I)|I(\Re(\lam_{N-1}(A_{N-1}))>   \hat{v}_N)]e^{-N\frac{v^2}{2p\alpha}}dv.\]
		The event on the right hand side may equivalently be stated as saying that at least one eigenvalue of $A_{N-1}-\hat{v}_NI$ has positive real part.
		
		Taking the ratio these integrands we see that the left hand side of (\ref{eqn:ignore-1218}) is bounded by
		\[\sup_{u<v<u+1}\frac{\E[|\det(A_{N-1}-\hat{v}_{N}I)|I(\Re(\lam_{N-1}(A_{N-1}))>  \hat{v}_N)]}{\E[|\det(A_{N-1}-\hat{v}_{N}I)|]}\le \]
		\[\sup_{u<v<u+1}\P(\Re(\lam_{N-1}(A_{N-1}))>  \hat{v}_N)^{1/2}\left(\frac{\E[|\det(A_{N-1}-\hat{v}_{N}I)|^2]^{1/2}}{\E[|\det(A_{N-1}-\hat{v}_{N}I)|]}\right).\]
		The second term on the right is polynomial, so we focus on bounding the first. We note that for $u<v<u+1$
		\[\P(\Re(\lam_{N-1}(A_{N-1}))> \hat{v}_N)\le \P(\Re(\lam_{N-1}(A_{N-1}))>\hat{u}_N).\]
		Now using the inequalities 
		\[\Re(\lam_{N-1}(A_N-1))\le |\lam_{N-1}(A_{N-1})|\le \sigma_1(A_{N-1})\] 
		(recall that $\sigma_1(A_{N-1})$ is the largest singular value) we have that
		\[\P(\Re(\lam_{N-1}(A_{N-1}))>\hat{u}_N)\le \P(s_{1}(A_{N-1})>\hat{u}_N).\]
		Now as $u>\Einf$ we see that $\lim_{N\to \infty}\hat{u}_N>1+\tau$. In particular, by Lemma \ref{lem:misc tail results on elliptic} we see that there are $C,c>0$ such that
		\[\P(s_{1}(A_{N-1}))>\hat{u}_N)\le Ce^{-cN}.\]
		Combining this with Theorem \ref{theorem:moments of character main} completes the proof.
	\end{proof}
	
	We now give a proof of Lemma \ref{lem:concentration lemma XY}, which will be followed by basic concentration estimates.
	
	\begin{proof}[Proof of Lemma \ref{lem:concentration lemma XY}]
		We note that
		\[\E[(X_N,Y_N)^2]=N^{-1}(N-1)\Sigma_{12}^2+N^{-1}(\Sigma_{12}^2+2\Sigma_{11}\Sigma_{22})=\Sigma_{12}^2+O(N^{-1}),\]
		and similarly $\E[(X_N,Y_N)^4]=\Sigma_{12}^4+O(N^{-1})$, so that $\Var((X_N,Y_N)^2)=O(N^{-1})$. Similarly we have that $\Var(\|X_N\|^2),\Var(\|Y_N\|^2)=O(N^{-1})$, so if we define the event 
		\[\cal{A}_N=\left\{\min(|(X_N,Y_N)^2-\Sigma_{12}^2|,|\|X_N\|^2-\Sigma_{11}^2|,|\|Y_N\|^2-\Sigma_{22}^2|)>N^{-1/4}\right\}\]
		Chebysev's inequality shows that $\P(\cal{A}_N)\to 0$ as $N\to \infty$. We observe that
		\[\lim_{N\to \infty}\E[f(\|X_N\|^2,\|Y_N\|^2,(X_N,Y_N)^2)I_{\cal{A}^c_N}]=f(\Sigma_{11}^2,\Sigma_{22}^2,\Sigma_{12}^2).\]
		To show the contribution from $\cal{A}_N$ is negligible, we note that
		\[\E[f(\|X_N\|^2,\|Y_N\|^2,(X_N,Y_N)^2)I_{\cal{A}_N}]\le \]\[\P(\cal{A}_N)^{1/2}\E[f(\|X_N\|^2,\|Y_N\|^2,(X_N,Y_N)^2)^2]^{1/2},\]
		so it suffices to show the final term is bounded in $N$. For this, note that as $|(X_N,Y_N)|\le \|X_N\|\|Y_N\|$, so as $f^2$ is of polynomial growth, it suffices to show that for each $\ell\in \N$
		\[\limsup_{N\to \infty}\left(\E[\|X_N\|^{2\ell}]+\E[\|Y_N\|^{2\ell}]\right)<\infty.\]
		Observing that $\|X_N\|^2\disteq \chi_{N}^2$, where $\chi_N^2$ is a $\chi^2$-random variable, this follows from the estimate above (\ref{eqn:chi-bound}).
	\end{proof}
	
	We complete this section by providing the construction of $E^{k}_N(r,u_1,u_2)$ and a proof of Lemma \ref{lem:perturbation lemma for M}. To begin, we fix three i.i.d matrices sampled from $GEE(\tau,N-1)$ and denote them $A_{N-1}^1,A_{N-1}^2$ and $A_{N-1}$. For $k=1,2$ (using the matrix $\Sigma_S(r)$ defined around (\ref{eqn:S-def})), we define
	\[
	Z_{N-1,N-1}^k(r):=\]
	\[\left(\sign([\Sigma_S(r)]_{12})^k|[\Sigma_S(r)]_{12}|^{1/2}-\sign(r)^{pk}\sq{|r|^{p-2}}\right)[A_{N-1}]_{(N-1)(N-1)}
	\]
	\[
	+\left(\left([\Sigma_S(r)]_{11}-[\Sigma_S(r)]_{12}\right)^{1/2}-\sq{1-|r|^{p-2}}\right)[A_{N-1}^k]_{(N-1)(N-1)}.
	\]
	Note that as $\Sigma_S(r)$ is positive semi-definite, with $[\Sigma_{S}(r)]_{11}=[\Sigma_{S}(r)]_{22}$, we must have that $[\Sigma_S(r)]_{11}-[\Sigma_S(r)]_{12}\ge 0$, so all expressions here make sense.
	
	Now for $1\le i\le N-2$ we will be using the $4$-by-$4$ the matrix $\Sigma_Z(r)$ defined around (\ref{eqn:def:SigmaZ}). This is composed of square $2$-by-$2$ block matrices $\Sigma_Z^{k,l}(r)$. As such the definition will be structurally the same as $Z_{N-1,N-1}^k(r)$, but involving matrix-valued quantities, which we understand in the spectral sense. Primarily, we need note that as $\Sigma_Z(r)$ is positive semi-definite and $\Sigma^{1,1}_Z(r)=\Sigma^{2,2}_Z(r)$, one sees by evaluating on vectors of the form $(u,v,-u,-v)$ that $\Sigma^{1,1}_Z(r)-\Sigma^{1,2}_Z(r)$ is positive semi-definite. With these prerequisites, we define
	\[
	\begin{bmatrix}Z_{N-1,i}^k(r)\\Z_{i,N-1}^k(r) \end{bmatrix}:=\]\[
	\left(\sign(\Sigma^{1,2}_Z(r))^k|\Sigma^{1,2}_Z(r)|^{1/2}
	\begin{bmatrix}
		1&\tau\\
		\tau& 1
	\end{bmatrix}^{-1/2}
	-\sign(r)^{pk}\sq{|r|^{p-2}}I\right)\begin{bmatrix}[A_{N-1}]_{(N-1)i}\\ [A_{N-1}]_{i(N-1)}\end{bmatrix}\]
	\[
	+\left((\Sigma^{1,1}_Z(r)-\Sigma^{1,2}_Z(r))^{1/2}
	\begin{bmatrix}
		1&\tau\\
		\tau& 1
	\end{bmatrix}^{-1/2}
	-\sq{1-|r|^{p-2}}I\right)\begin{bmatrix}[A_{N-1}^k]_{(N-1)i}\\ [A_{N-1}^k]_{i(N-1)}\end{bmatrix}.
	\]
	We finally define an $(N-1)$-by-$(N-1)$ matrix $E_N^k(r,u_1,u_2)$ by defining for $1\le i\le N-2$
	\[[E_N^k(r,u_1,u_2)]_{i(N-1)}=Z^k_{i,N-1}(r),\;\;\;[E_N^k(r,u_1,u_2)]_{(N-1)i}=Z^k_{N-1,i}(r),\]
	\[[E_N^k(r,u_1,u_2)]_{(N-1)(N-1)}=m^k_{1}(r)\hat{u}_{N,1}+m^k_{2}(r)\hat{u}_{N,2}+Z^k_{N-1,N-1}(r),\]
	and letting all other entries be zero.
	
	We first verify the claim (\ref{eqn:E-perp-exact}). As both sides are Gaussian, it will suffices to verify that the Gaussian vector $(A_{N-1}^k(r)+E_N^k(r,u_1,u_2)-\hat{u}_{N,k}I)_{k=1,2}$ has the mean and covariance structure specified by Lemma \ref{lem:jacobian-covariance-statement-with-energy}. The nontrivial claims are those involving the covariance matrices of $(V_N^k(r),W_N^k(r))_{k=1,2}$ and $(S_N^k(r,u_1,u_2))_{k=1,2}$. For this, we note that for $1\le i\le N-2$
	\[\begin{bmatrix}
		1&\tau\\
		\tau& 1
	\end{bmatrix}^{-1/2}\begin{bmatrix}[A_{N-1}]_{(N-1)i}\\ [A_{N-1}]_{i(N-1)}\end{bmatrix}\disteq \cal{N}(0,I),\]
	and similarly for entries involving $A_{N-1}^k$. From this we see that the covariance matrix of $([A_{N-1}^k(r)+E_N^k(r,u_1,u_2)]_{i(N-1)},[A_{N-1}^k(r)+E_N^k(r,u_1,u_2)]_{(N-1)i})_{k=1,2}$ for $1\le i\le N-2$ is given by
	\[\begin{bmatrix}
		\Sigma^{1,1}_Z(r)& \Sigma^{1,2}_Z(r)\\
		\Sigma^{1,2}_Z(r)^t& \Sigma^{1,1}_Z(r)\\
	\end{bmatrix}.\]
	Thus recalling that $\Sigma_Z^{1,1}(r)=\Sigma_Z^{2,2}(r)$ and $\Sigma_Z^{1,2}(r)=\Sigma_Z^{2,1}(r)^t$ verifies the first claim. The claim for $(S_N^k(r,u_1,u_2))_{k=1,2}$ is similar, which establishes (\ref{eqn:E-perp-exact}).
	
	To complete the proof of Lemma \ref{lem:perturbation lemma for M}, we need to prove the statements about $E_N^k(r,u_1,u_2)$. The claim involving the rank is obvious, so we focus on the latter. The primary tool is the following result, proven in Appendix \ref{appendix:covariance-computation}, which quantifies the rate at which $E_N^k(r,u_1,u_2)$ vanishes as $r\to 0$.
	
	\begin{lemma}
		\label{lem:vanishing rate of E} 
		For $\epsilon>0$ there is $C>0$ such that for $|r|<1-\epsilon$, $k,l=1,2$, and $1\le i\le N-2$ 
		\[\E[(Z_{N-1,i}^k(r))^2]\le CN^{-1}r^{p-3},\;\;\E[(Z_{i,N-1}^k(r))^2]\le CN^{-1}r^{p-3},\label{eqn:ignore-druid-1}\]
		\[\E[(Z_{N-1,N-1}^k(r))^2]\le CN^{-1}r^{p-4},\;\; |m^k_{l}(r)|\le Cr^{p-3}.\label{eqn:ignore-druid-2}\]
	\end{lemma}
	
	\begin{proof}[Proof of Lemma \ref{lem:perturbation lemma for M}]
		
		As noted, we only need to show (\ref{eqn:E-van-rate}). We note that
		\[\|E^k_N(r,u_1,u_2)\|^2_F=\sum_{i=1}^{N-2}\left(Z_{i,N-1}^k(r)^2+Z_{N-1,i}^k(r)^2\right)+\]
		\[(m_1^k(r)\hat{u}_{N,1}+m_2^k(r)\hat{u}_{N,2}+Z^k_{N-1,N-1}(r))^2.\]
		By Lemma \ref{lem:vanishing rate of E} and the independence of these elements, we see that for large enough $C$, $\|E^k_N(r)\|^2_F$ is stochastically dominated by
		\[CN^{-1}\chi_{N-1}^2r^{p-3}+C(\hat{u}_{N,1}+\hat{u}_{N,2})^2r^{p-2}+CN^{-1}r^{p-4}X^2\]
		where $\chi^2_N$ is a $\chi^2$-random variable of parameter $N$ and $X$ is a standard Gaussian random variable. By appealing to Cram\'{e}r's Theorem, we see that for $C>1$ there is $c>0$ such that $\P(N^{-1}\chi_{N-1}^2\ge C)\le Ce^{-cN}$. More simply note that there are $c,C$ such that $\P(X^2\ge N^{1/2})\le Ce^{-cN^{1/2}}$. Combining these statements yields the final claim.
		
	\end{proof}
	\section{Proofs of Lemmas \ref{lem:diagonal-lemma-intro} and \ref{lem:analytic equality result}\label{section:lemmas about functions}}
	
	The purpose of this section is to give proofs of the main analytic lemmas through which we understand the complexity functions. This will primarily involve understanding the points which achieve their supremums, and as such will primarily be shown through explicit but non-trivial calculus.
	
	The results are analogs of results of \cite{pspin-second,subag-ofer-new} for the case $\tau=1$. Our main method to establish them for general $\tau$ will be to reduce to the cases of $\tau=1$ and $\tau=-(p-1)^{-1}$, with the latter often being surprisingly simple. We will begin with the following prerequisite for the proof of Lemma \ref{lem:diagonal-lemma-intro}
	\begin{lemma}
		\label{lem:diagonal-lemma-old}
		Define $\Sigma_2^{p,\tau}(r,u):=\Sigma_2^{p,\tau}(r,u,u)$. Then if $B$ is a connected open subset of either $(-\infty, -\Einf)$, $(-\Einf,\Einf)$, or $(\Einf,\infty)$, we have for any $r\in (-1,1)$ that
		\[\sup_{u_1,u_2\in B}\Sigma_2^{p,\tau}(r,u_1,u_2)= \sup_{u\in B}\Sigma_2^{p,\tau}(r,u).\]
	\end{lemma}
	
	\begin{proof}
		This will follow from studying the behavior of the function $\Sigma_2^{p,\tau}(r,u_1,u_2)$ in terms of the variable $\nu=(u_1-u_2)$. Let us first treat the case of $B\subset (-\infty,-\Einf)$. Note that $\phi_\tau$ is concave on $(-\infty,-1-\tau)$, as this set is outside of $\cal{E}_\tau$. Noting as well that $\Sigma_U(r)$ is positive-definite by construction, we see that $\Sigma_2^{p,\tau}(r,u_1,u_2)$ is concave in $(u_1,u_2)$ in the region $(-\infty,-\Einf)^2$, and in particular, it is concave in the variable $\nu=u_1-u_2$, restricted to this region. Now as the function $\Sigma_2^{p,\tau}(r,u_1,u_2)$ is symmetric in $(u_1,u_2)$, we see that $\frac{d}{d\nu}\Sigma_2^{p,\tau}(r,u,u)=0$, so that for $u_1,u_2\in (-\infty,-\Einf)$, we have that \[\Sigma_2^{p,\tau}(r,u_1,u_2)\le \Sigma_2^{p,\tau}\left(r,\frac{u_1+u_2}{2},\frac{u_1+u_2}{2}\right).\]
		This demonstrates the claim in the case that $B\subset (-\infty,-\Einf)$, and an identical argument demonstrates the claim for $B\subset (\Einf,\infty)$.
		
		In the case where $B\subset (-\Einf,\Einf)$, we see by the above argument that it is sufficient to show for $u_1,u_2\in (-\Einf,\Einf)$ that
		\[\frac{d^2}{d\nu^2}\Sigma_2^{p,\tau}(r,u_1,u_2)< 0.\label{eqn:ignore-416}\]
		We note that
		\[\frac{d^2}{d\nu^2}\Sigma_2^{p,\tau}(r,u_1,u_2)=-[\Sigma_U(r)^{-1}]_{11}-[\Sigma_U(r)^{-1}]_{22}+2[\Sigma_U(r)^{-1}]_{12}+\]\[\frac{1}{p(p-1)}\phi_\tau''\left(\frac{u_1}{\sq{p(p-1)}}\right)+\frac{1}{p(p-1)}\phi_\tau''\left(\frac{u_2}{\sq{p(p-1)}}\right)=\]
		\[\frac{2}{p\alpha b(r)}(k_2(r)-k_1(r))+\frac{2}{p(p-1)(1+\tau)}\]
		where in the last step we have used (\ref{eqn:def:U-external field form}). The function $b$ occurs as the determinant of the matrix (\ref{eqn:ignore-vic}), which itself is the inverse of a matrix block in the inverse of the matrix (\ref{eqn:ignore-vector}) by the Schur complement formula. As this is a covariance matrix, we see that $b$ is positive. Thus by the positivity of $b$, (\ref{eqn:ignore-416}) is equivalent to the claim
		\[j_{\tau,p}(r):=\alpha b(r)+(1+\tau)(p-1)(k_2(r)-k_1(r))< 0.\]
		We first show that for $|\tau|<1$ and $|r|<1$ we have that
		\[1-r^{2p-2}+\tau (p-1) r^{p-2}(1-r^2)>0.\label{eqn:ignore-423}\]  
		For this, we observe that
		\[1-r^{2p-2}-(p-1)|r|^{p-2}(1-r^2)=(1-r^2)\bigg(\sum_{i=0}^{p-2}r^{2i}-(p-1)|r|^{p-2}\bigg)>0,\label{eqn:ignore-939}\]
		as we have that 
		\[\frac{1}{p-1}\sum_{i=0}^{p-2}r^{2i}> |r|^{p-2},\label{eqn:AM-GM}\]
		by the AM-GM inequality applied to the sequence $\{r^{2i}\}_{i=0}^{p-2}$. This establishes (\ref{eqn:ignore-423}).
		
		Now we observe that
		\[k_1(r)-k_2(r)=b(r)+r^p(1+r^{p-2})(1-r^{2p-2}+\tau(p-1)r^{p-2}(1-r^2)),\]
		\[b(r)=(1-r^{2p-2}-\tau (p-1)r^{p-2}(1-r^2))(1-r^{2p-2}+\tau (p-1)r^{p-2}(1-r^2)).\label{eqn:ignore-699}\]
		We have just shown that the common factor of these two terms is positive, so it suffices to show that
		\[\bar{j}_{\tau,p}(r):=(1-r^{2p-2}+\tau (p-1)r^{p-2}(1-r^2))^{-1}j_{\tau,p}(r)=\]
		\[-(p-2)(1-r^{2p-2}-\tau (p-1)r^{p-2}(1-r^2))-(1+\tau)(p-1)r^p(1+r^{p-2})<0.\]
		Now we observe that $\bar{j}_{\tau,p}(r)$ is linear in $\tau$, so it suffices to verify this claim for $\tau\in \{-(p-1)^{-1},1\}$. For $\tau=-(p-1)^{-1}$ we have that
		\[\bar{j}_{-(p-1)^{-1},p}(r)=-(p-2)(1+r^{p-2})< 0.\]
		The case of $\tau=1$ is treated by \cite{subag-ofer-new}. Specifically, $-\bar{j}_{1,p}(r)=h(r)$, where $h$ is introduced in (2.10) of \cite{subag-ofer-new}, which they then show is positive.
	\end{proof}
	\begin{remark}
		\label{remark:strict increasing of phi}
		We note that when $r=0$ we have that $\Sigma_2^{p,\tau}(0,u_1,u_2)=\Sigma^{p,\tau}(u_1)+\Sigma^{p,\tau}(u_2)$. In particular, the concavity in $u_1-u_2$ established in the above proof shows that the function $\Sigma^{p,\tau}(u)$ is strictly concave for $u\in (\Einf,\infty)$ and $u\in (-\Einf,\Einf)$. As the function is even, this implies it is strictly decreasing on $[0,\Einf)$. As in addition $\phi_\tau'(1+\tau)=1$, we see that
		\[(\Sigma^{p,\tau})'(\Einf)=-\frac{(1+\tau)\sq{p(p-1)}}{p\alpha}+\frac{\phi_\tau'(1+\tau)}{\sq{p(p-1)}}=-\frac{(p-2)}{\alpha\sq{p(p-1)}}<0,\]
		so that we see that in fact $\Sigma^{p,\tau}$ is strictly decreasing on $(0,\infty)$. In addition, as the function is even, it must be strictly increasing on $(-\infty,0)$.
	\end{remark}
	
	\begin{proof}[Proof of Lemma \ref{lem:diagonal-lemma-intro}]
		We observe that as $\Crit_N(\{\pm\Einf\})=0$ a.s, it suffices to treat the case where we can write $B=\coprod_{i=1}^{m}B_i$ as a disjoint union of intervals which satisfy the conditions of Lemma \ref{lem:diagonal-lemma-old}. We note that as 
		\[\Crit_N(B)=\sum_{i=1}^m\Crit_N(B_i),\]
		we may apply the Cauchy-Schwartz inequality to obtain that
		\[\E[\Crit_{N}(B)^2]\le m^2\sum_{i=1}^m\E[\Crit_{N}(B_i)^2],\]
		so that 
		\[\limsup_{N\to \infty}N^{-1}\log \E[\Crit_{N}(B)^2]\le \max_{1\le i\le m}\limsup_{N\to \infty}N^{-1}\log\E[\Crit_{N}(B_i)^2].\]
		We also note that $\Crit_{N,2}(B_i,\{\pm 1\})\le 2\Crit_{N}(B_i)$, so by Theorems \ref{theorem:fyodorov 1st moment} and \ref{theorem:KR-2point}, as well as Lemma \ref{lem:diagonal-lemma-old}
		\[\limsup_{N\to \infty} N^{-1}\log(\E[\Crit_{N}(B_i)^2])\le \max\left(\sup_{u\in B_i,\; r\in (-1,1)}\Sigma_2^{p,\tau}(r,u),\sup_{u\in B_i}\Sigma^{p,\tau}(u))\right).\]
		Thus we see that
		\[\limsup_{N\to \infty}N^{-1}\log(\E[\Crit_{N}(B)^2])\le \]\[\max_{1\le i\le m}(\sup_{u\in B_i,\; r\in (-1,1)}\max(\Sigma_2^{p,\tau}(r,u),\Sigma^{p,\tau}(u)))=\]
		\[\sup_{u\in B,\;r\in (-1,1)}\max(\Sigma_2^{p,\tau}(r,u),\Sigma^{p,\tau}(u)).\]
		Now we note that $\Sigma^{p,\tau}_2(0,u)=2\Sigma^{p,\tau}(u)$, which is positive somewhere on $B$ by assumption. In particular,
		\[\sup_{u\in B,\; r\in (-1,1)}\max(\Sigma_2^{p,\tau}(r,u),\Sigma^{p,\tau}(u))=\sup_{u\in B,\; r\in (-1,1)}\Sigma_2^{p,\tau}(r,u),\]
		which completes the proof.
	\end{proof}
	
	To prove Lemma \ref{lem:analytic equality result}, we need the following identification for the maximizers of $\Sigma_2^{p,\tau}(r,u)$.
	\begin{lemma}
		\label{lem:sup-iden}
		Let us define 
		\[\uth:=\sq{(1+\tau)\frac{\log(p-1)(p-1)p\alpha}{(p-2)}}.\]
		Then we may extend $\Sigma_2^{p,\tau}(r,u)$, for fixed $u$, to a continuous function of $r\in (-1,1]$, which we denote by $\bar{\Sigma}_2^{p,\tau}(r,u)$, and such that:
		\begin{itemize}
			\item For $|u|=\uth$ and $r\in (-1,1)\setminus\{0\}$ we have that
			\[\Sigma_2^{p,\tau}(r,u)<\Sigma_2^{p,\tau}(0,u)=2\Sigma^{p,\tau}(0,u).\]
			\item For $|u|<\uth$ and $\epsilon>0$, we have $\sup_{|r|\ge \epsilon}\Sigma_2^{p,\tau}(r,u)<\Sigma_2^{p,\tau}(0,u)$.
			\item For $|u|>\uth$ and $\epsilon>0$, we have $\sup_{|r|\le 1-\epsilon}\Sigma_2^{p,\tau}(r,u)<\bar{\Sigma}_2^{p,\tau}(1,u)$.
		\end{itemize}
	\end{lemma}
	
	\begin{proof}[Proof of Lemma \ref{lem:sup-iden}]
		We begin by simplifying the expressions in $\Sigma_2^{p,\tau}$. Employing (\ref{eqn:def:U-external field form}) we may write
		\[\frac{1}{2}(1,1)\Sigma_U(r)^{-1}\begin{bmatrix}1 \\ 1 \end{bmatrix}=\frac{k_1(r)+k_2(r)}{p\alpha b(r)}.\label{eqn:ignore-353}\]
		Using the expression for $k_1,k_2$ and $b$ given after (\ref{eqn:def:U-external field form}) we see that
		\[k_1(r)+k_2(r)-b(r)=(r^{2p-2}-r^p)(1-r^{2p-2}-(p-1)\tau r^{p-2}(1-r^2)).\]
		By this and (\ref{eqn:ignore-699}) we see we may cancel a common factor in the expression (\ref{eqn:ignore-353}), so that
		\[\frac{k_1(r)+k_2(r)}{b(r)}=1-g^{p,\tau}(r),\;\; g^{p,\tau}(r):=\frac{r^{p}(1-r^{p-2})}{1-r^{2p-2}+(p-1)\tau r^{p-2}(1-r^2)}.\]
		If we define then
		\[\;\;Q^{p,\tau}(r,u):=h(r)+\frac{u^2}{p\alpha}g^{p,\tau}(r),\]
		we see that
		\[\Sigma_2^{p,\tau}(r,u)=Q^{p,\tau}(r,u)+1+\log(p-1)-\frac{u^2}{p\alpha}+2\phi_{\tau,p}\left(u\right).\]
		We see it suffices to prove the corresponding statements for $Q^{p,\tau}(r,u)$, as the other terms are constant in $r$. To prove the existence of the desired continuous extension, we note that
		\[\lim_{r\to 1}Q^{p,\tau}(r,u)=-\frac{1}{2}\log(p-1)\left(1-\frac{u^2}{\uth^2}\right).\label{eqn:ignore-Qp}\]
		For notational ease, we will denote the extension of $Q^{p,\tau}$ to $r\in (-1,1]$ again by $Q^{p,\tau}$. 
		
		We next verify that it suffices to show the claims when $r\in [0,1)$. When $p$ is even, this follows from the fact that $Q^{p,\tau}$ is even in $r$. When $p$ is odd, it will suffice to show that for $r\in (0,1)$, $g^{p,\tau}(r)>0$ and for $r\in (-1,0)$, $g^{p,\tau}(r)<0$. On the other hand by (\ref{eqn:AM-GM}) the denominator is positive, so this follows by inspection of the numerator. In either case, we may now restrict to $r\in [0,1)$.
		
		Next we will show that $g^{p,\tau}$ is increasing in $r\in (0,1)$. We first note that
		\[\frac{d}{dr}g^{p,\tau}(r)=\frac{r^{p-1}\bar{h}^{p,\tau}(r)}{(1-r^{2p-2}+(p-1)\tau r^{p-2}(1-r^2))^2},\]
		where here
		\[\bar{h}^{p,\tau}(r):=(p-(2p-2)r^{p-2})(1-r^{2p-2}+(p-1)\tau r^{p-2}(1-r^2))+\]\[(r-r^{p-1})((2p-2)r^{2p-3}-(p-1)\tau r^{p-3}((p-2)-pr^2)).\]
		To show that $g^{p,\tau}$ is increasing it suffices to show that $\bar{h}^{p,\tau}(r)>0$ for $r\in (0,1)$. As $\bar{h}^{p,\tau}(r)$ is linear in $\tau$, it suffices to verify this in the cases that $\tau=1$ and $\tau=-(p-1)^{-1}$. Routine calculation gives that
		\[\bar{h}^{p,1}(r)=p(1-r^{2p-2}-(p-1)r^{2p-4}(1-r^2))\ge p(1-r^{2p-2}-(p-1)|r|^{p-2}(1-r^2))>0,\]
		and similarly that
		\[\bar{h}^{p,-(p-1)^{-1}}(r)=p(1-r^{p-2})^2>0\]
		This completes the verification that $g^{p,\tau}$ is increasing.
		
		We note from (\ref{eqn:ignore-Qp}) that $Q^{p,\tau}(1,\uth)=0=Q^{p,\tau}(0,\uth)$. We will now reduce to the case of $u=\uth$. Indeed, given the result in this case, we note that for $|u|<\uth$, and $r\in (0,1)$, we have that
		\[Q^{p,\tau}(r,u)<Q^{p,\tau}(r,\uth)\le Q^{p,\tau}(0,\uth)= Q^{p,\tau}(0,u),\]
		where in the first inequality we have used that $g^{p,\tau}(r)>0$. Similarly if $|u|>\uth$ then for $r\in (0,1)$
		\[Q^{p,\tau}(1,u)=Q^{p,\tau}(1,\uth)+\frac{(u^2-\uth^2)}{p\alpha}g^{p,\tau}(1)\ge\]\[ Q^{p,\tau}(r,\uth)+\frac{(u^2-\uth^2)}{p\alpha}g^{p,\tau}(1)=\]
		\[Q^{p,\tau}(r,u)+\frac{(u^2-\uth^2)}{p\alpha}(g^{p,\tau}(1)-g^{p,\tau}(r))\ge Q^{p,\tau}(r,u),\]
		where in the last step we have used that $g^{p,\tau}(r)$ is increasing in $r$. This completes the reduction to the case where $u=\uth$.

		Let us denote $\bar{Q}^{p,\tau}(r):=Q^{p,\tau}(r,\uth)$, so that we now only need to show that for $r\in (0,1)$, we have that $\bar{Q}^{p,\tau}(r)<0$. Our final reduction will be to the case where $\tau=1$. To do this, we may compute
		\[\frac{d}{d\tau}\bar{Q}^{p,\tau}(r)=\frac{d}{d\tau}\bigg(\frac{\uth^2}{p\alpha}g^{p,\tau}(r)\bigg)=\frac{(p-1)\log(p-1)}{(p-2)}\frac{d}{d\tau}\bigg((1+\tau)g^{p,\tau}(r)\bigg)=\]
		\[\frac{(p-1)\log(p-1)}{(p-2)}\bigg(\frac{r^p(1-r^{p-2})(1-r^{2p-2}-\tau(p-1)r^{p-2}(1-r^2))}{(1-r^{2p-2}+(p-1)\tau r^{p-2}(1-r^2))^2}\bigg).\]
		In particular, we see from this and (\ref{eqn:ignore-939}) we see that $\frac{d}{d\tau}\bar{Q}^{p,\tau}(r)>0$, so it suffices to verify the case where $\tau=1$. This coincides with the case addressed in \cite{pspin-second}. In particular, our $\bar{Q}^{p,1}(r)$ coincides with their $Q_p^{u_{th}(p)}$, which they define in their Lemma 7 and Eqn. (6.12). In particular, the verification that $\bar{Q}^{p,1}(r)<0$ for $r\in (0,1)$ now follows from their verification that $Q^{u_{th}(p)}_p(r)<0$, given in their proof of Lemma 7 (more specifically, following their Eqn. 6.15).
	\end{proof}
	
	We are now able to give the proof of Lemma \ref{lem:analytic equality result}.
	
	\begin{proof}[Proof of Lemma \ref{lem:analytic equality result}]
		To begin, we treat the values outside of $(-\Ezero,\Ezero)$ beginning with those outside $(-\uth,\uth)$. We see by Lemma \ref{lem:sup-iden} and evenness of $\bar{\Sigma}^{p,\tau}_2$ that
		\[\sup_{|v|\ge \uth ,r\in (-1,1), }\Sigma_2^{p,\tau}(r,v)=\sup_{v\ge\uth }\bar{\Sigma}^{p,\tau}_2(1,v).\]
		We will show for $v\ge \uth$ that
		\[\bar{\Sigma}^{p,\tau}_2(1,v)<0.\label{eqn:ignore-239}\]
		Before establishing this fact, we will complete the proof assuming (\ref{eqn:ignore-239}). Using (\ref{eqn:ignore-239}) we see that
		\[\sup_{|v|\ge\uth, r\in (-1,1)}\bar{\Sigma}^{p,\tau}_2(r,v)\le 0.\label{eqn:ignore-alex}\]
		By applying Lemma \ref{lem:sup-iden} again, and the fact that $\Sigma_2^{p,\tau}(0,v)=2\Sigma^{p,\tau}(v)$ is negative when $|v|>\Ezero$, we further see that 
		\[\sup_{|v|\ge \Ezero, r\in (-1,1) }\Sigma^{p,\tau}_2(r,v)\le 0. \label{eqn:ignore-alex-2}\]
		Now a final application of Lemma \ref{lem:sup-iden} shows that if $0\in \bar{I}$ than
		\[\sup_{v\in B\cap (-\Ezero,\Ezero),r\in I}\Sigma_2^{p,\tau}(r,v)=\sup_{v\in B\cap (-\Ezero,\Ezero) }\Sigma_2^{p,\tau}(0,v),\]
		and as $B$ intersects $(-\Ezero,\Ezero)$, we see further that this is a strict inequality if $0\notin \bar{I}$. Moreover, as $\Sigma_2^{p,\tau}(0,v)$ is positive if $v\in (-\Ezero,\Ezero)$, we see if $0\in \bar{I}$ that
		\[\sup_{v\in B,r\in I}\Sigma_2^{p,\tau}(r,v)=\sup_{v\in B\cap (-\Ezero,\Ezero),r\in I}\Sigma_2^{p,\tau}(r,v),\]
		and if $0\notin \bar{I}$ that
		\[\sup_{v\in B,r\in I}\Sigma_2^{p,\tau}(r,v)\le \max\left(\sup_{v\in B\cap (-\Ezero,\Ezero),r\in I}\Sigma_2^{p,\tau}(r,v),0\right)<2\sup_{v\in B}\Sigma^{p,\tau}(v).\]
		In particular, this and the previous result completes the proof once we establish (\ref{eqn:ignore-239}).
		
		To do this we first need a preliminary fact about $\phi_\tau$. We note that the function $\phi_\tau$ is concave for $u\in (1+\tau,\infty)$. We note as well that $\phi_\tau(1+\tau)=\tau/2$ and $\phi_\tau'(1+\tau)=1$, so by concavity we have that for $u>1+\tau$
		\[\phi_\tau(u)\le \frac{\tau}{2}+(u-(1+\tau))= \frac{u^2}{2(1+\tau)}-\frac{1}{2}-\frac{1}{2(1+\tau)}(u-(1+\tau))^2\le \frac{u^2}{2(1+\tau)}-\frac{1}{2}.\]
		As this bound is an equality for $|u|\le 1+\tau$, and the function $\phi_\tau$ is even, we see that for any $u\in \R$
		\[\phi_\tau(u)\le \frac{u^2}{2(1+\tau)}-\frac{1}{2}.\label{eqn:ignore-430}\]
		Employing (\ref{eqn:ignore-Qp}) we see that we may write
		\[\bar{\Sigma}_2^{p,\tau}(1,u)=\frac{\log(p-1)}{2}\left(\frac{u^2}{\uth^2}-1\right)+1+\log(p-1)-\frac{u^2}{p\alpha}+2\phi_{\tau,p}\left(u\right).\]
		By (\ref{eqn:ignore-430}) we have that
		\[2\phi_{\tau,p}\left(u\right)\le-1+\frac{u^2}{p(p-1)(1+\tau)}=-1+\frac{\log(p-1)\alpha}{(p-2)}\frac{u^2}{\uth^2},\]
		so we see that
		\[\bar{\Sigma}_2^{p,\tau}(1,u)\le \frac{1}{2}\log(p-1)+\frac{u^2}{\uth^2}\bigg(\frac{\log(p-1)}{2}-\frac{\log(p-1)(p-1)(1+\tau)}{(p-2)}+\]\[\frac{\log(p-1)\alpha}{(p-2)}\bigg)=\frac{\log(p-1)}{2}\left(1-\frac{u^2}{\uth^2}\right),\]
		which verifies (\ref{eqn:ignore-239}).
	\end{proof}
	
	\section{Random Matrix Theory and Elliptic Ensembles \label{section:elliptic-ensemble-background}}
	
	In this section we will prove the majority of the random matrix results we need concerning the elliptic ensembles. Our first goal will be to give a proof of Lemma \ref{lem:main elliptic ensemble bound}, which gives the behavior of the characteristic polynomial on the exponential scale. Our other main result will be Proposition \ref{prop:moments of character outside bulk}, which gives our bounds on the characteristic polynomial outside of the bulk. In the course of proving these, we will also need an overcrowding result to control the behavior of the smallest singular value. This is given by Lemma \ref{lem:naomov smallest eigenvalue bound}. The majority of the results in this section will follow from applying standard Gaussian concentration inequalities to the forms of the elliptic law given in \cite{naumov,rourke-elliptic-2}. 
	
	To begin we recall some general facts from linear algebra. The first of these is the Hoffman-Wielandt inequality (see Lemma 2.1.19 of \cite{cupbook}). This states that for $n$-by-$n$ symmetric matrices $A$ and $B$ we have that
	\[\sum_{i=1}^n |\lam_i(A)-\lam_i(B)|^2\le \Tr (A-B)^2=\|A-B\|^2.\]
	We also recall that for an $n$-by-$n$ matrix $M$, the eigenvalues of the matrix
	\[\Mb:=\begin{bmatrix}0& M\\ M^t & 0 \end{bmatrix}\label{eqn:def-symmetrization-matrix}\]
	are given by $-s_1(M),\dots ,-s_n(M),s_n(M),\dots, s_1(M)$. Combining these two statements we see that for general $n$-by-$n$ matrices $A$ and $B$ we have that
	\[\sum_{i=1}^n|s_i(A)-s_i(B)|^2\le \|A-B\|^2.\label{eqn:hoffman-2}\]
	From this statement, we conclude the following result, which may be compared to Lemma 2.3.1 of \cite{cupbook}.
	\begin{lemma}
		\label{lem:lipschitz-norm} \cite{cupbook}
		Let $f:[0,\infty)\to \R$ be a Lipschitz function of norm less than $C$. Then for any $z\in \R$ the function on $N$-by-$N$ matrices given by
		\[A\mapsto \int_0^\infty f(x)\nu_{A,z}(dx)\]
		has Lipschitz norm bounded by $C/\sq{N}$ (with respect to the Frobenius norm)
	\end{lemma} 
	\begin{proof}
		If $A$ and $B$ are two $N$-by-$N$ matrices then
		\[|\int_0^\infty f(x)\nu_{A,z}(dx)-\int_0^\infty f(x)\nu_{B,z}(dx)|\le  \frac{C}{N}\sum_{i=1}^N |s_i(A-zI)-s_i(B-zI)|\le\]\[ \frac{C}{\sq{N}}\left(\sum_{i=1}^N |s_i(A-zI)-s_i(B-zI)|^2\right)^{1/2}\le \frac{C}{\sq{N}}\|A-B\|^2.\]
	\end{proof}
	We will also need a bound on the determinant of the sum of two matrices. In the case of symmetric matrices this follows from the main result of \cite{bounds-on-det-sum}, and for non-symmetric case it follows by applying the symmetric case to (\ref{eqn:def-symmetrization-matrix}).
	\begin{corollary}
		\label{corr:determinant of sum bound} \cite{bounds-on-det-sum}
		Let $A$ and $B$ be $n$-by-$n$ matrices. Then
		\[|\det(A+B)|\le \prod_{i=1}^{n}(s_i(A)+s_{n+1-i}(B)).\]
		In particular, if $\rank(B)\le d$ then
		\[|\det(A+B)|\le |\det(A)|\left(1+\frac{s_1(B)}{s_n(A)}\right)^d.\]
	\end{corollary}
	
	Next we will need to recall some results for the empirical measures $\nu_{A_N,z}$. In the course of proving a general form for the elliptic law, it was shown in \cite{naumov} (more specifically, in their proof of Theorem 5.2) that for fixed $z$, the sequence $\nu_{A_N,z}$ converges a.s. to some deterministic limit $\nu_z$. This measure is studied more extensively in Appendix C of \cite{rourke-elliptic-2}. We will collect some of their results that we will need as the following lemma.
	\begin{lemma}
		\label{lem:limit measure of v_N}
		\cite{naumov,rourke-elliptic-2}
		For any $|\tau|<1$, and fixed $z\in \R$ we have that
		\[\E[\nu_{A_N,z}]\To \nu_z\label{eqn:convergence of nu_z}\]
		where $\nu_z$ is a certain ($\tau$-dependent) measure, characterized by Lemma C.6 of \cite{rourke-elliptic-2}. Moreover, the convergence in (\ref{eqn:convergence of nu_z}) is uniform for $z$ in any compact subset of $\R$, and the measure $\nu_z$ is compactly-supported with a continuous density. Finally, this measure is uniformly bounded in the following sense: for any $K>0$ we may find $C>0$ such that if we denote the density of $\nu_z$ as $\bar{\nu}_z$ then
		\[\sup_{|z|\le K,x\in \R}|\bar{\nu}_z(x)|\le C.\label{eqn:bound of nu}\]
	\end{lemma}
	\begin{proof}
		Except for the uniformity claims, all of these results directly follow from Lemma C.7 of \cite{rourke-elliptic-2} and the discussion immediately prior. Uniform boundedness of the density follows from observing that all bounds used in their proof of Lemma C.7 hold uniformly in compact subsets of $z$. Finally, to establish the uniformity claim for the convergence in (\ref{eqn:convergence of nu_z}) we observe that for $x,y\in\R$
		\[\frac{1}{N}\sum_{i=1}^N\left| s_i(A_N-xI)-s_i(A_N-yI)\right|\le\]
		\[\frac{1}{\sq{N}}\left(\sum_{i=1}^N(s_i(A_N-xI)-s_i(A_N-yI))^2\right)^{1/2}\le |x-y|\frac{\|I\|}{\sq{N}}=|x-y|,\label{eqn:hoff-to}\]
		where in the last step we have used (\ref{eqn:hoffman-2}). Taking expectations, this statement easily gives the desired uniformity.
	\end{proof}
	
	We now note that the limiting analog of the identity
	\[|\det(A-zI)|=\sq{|\det((A-zI)^t(A-zI))|},\]
	yields the following continuum relation (see as well Lemma 2.3 of \cite{naumov})
	\[\int_0^{\infty} \log(x)\nu_z(dx)=\phi_\tau(z).\label{eqn:continuum-girko}\]
	We will need the following result which will follow from this and the uniform boundedness of $\nu_z$.
	\begin{lemma}
		\label{lem:uniformity of v_z in epsilon}
		For any $|\tau|<1$ and $K>0$ we have that
		\[\lim_{\epsilon\to 0}\sup_{|z|\le K}\left|\int_0^{\infty} \log_{\epsilon}(x)\nu_z(dx)- \phi_\tau(z)\right|=0.\label{eqn:uniformity of v_z in epsilon}\]
	\end{lemma}
	\begin{proof}
		We observe that by (\ref{eqn:continuum-girko}) we have that
		\[|\int_0^{\infty} \log_{\epsilon}(x)\nu_z(dx)- \phi_\tau(z)|=|\int_0^{\epsilon}\log(x/\epsilon)\nu_z(dx)|.\]
		By (\ref{eqn:bound of nu}) we see that there is $C>0$ such that for $|z|<K$
		\[|\int_0^{\epsilon}\log(x/\epsilon)\nu_z(dx)|\le C\epsilon\int_0^1\log(|x|)dx.\]
		Together these establish (\ref{eqn:uniformity of v_z in epsilon}).
	\end{proof}
	
	We will now give a concentration estimate for $\int_0^{\infty}\log_{\epsilon}(x)\nu_{A_N,z}(dx)$.
	\begin{lemma}
		\label{lem:log-sob bound}
		For any $\epsilon>0$ and $\ell>0$ we have that
		\[
		\E[\exp(N\ell \int_0^{\infty}\log_{\epsilon}(x)\nu_{A_N,z}(dx))]\le \exp\left(N\ell\int_0^\infty \log_{\epsilon}(x)\E[\nu_{A_N,z}](dx)+8\ell^2\epsilon^{-2}\right).\label{eqn:ignore-proof-311}
		\]
	\end{lemma}
	\begin{proof}
		We observe that all entries of $A_N$ are independent except for the pairs $([A_N]_{ij},[A_N]_{ji})$. For each $i,j$, the pair $([A_N]_{ij},[A_N]_{ji})$ satisfies the log-Sobolev inequality with constant $4/N$, by the Bakry-Emery criterion \cite{Log-Sobolev}. Thus $A_N$, considered as a random vector in $\R^{N^2}$, also satisfies the log-Sobolev inequality with constant $4/N$. We also note that $\log_{\epsilon}$ has Lipschitz constant $1/(2\epsilon)$, so by Lemma \ref{lem:lipschitz-norm} the function
		\[A\mapsto \int_0^\infty \log_{\epsilon}(x)\nu_{A,z}(dx)\]
		has Lipschitz constant bounded by $1/(2\epsilon\sq{N})$.
		Thus (\ref{eqn:ignore-proof-311}) now follows by Herbst's argument (see section 2.3 of \cite{Log-Sobolev}).
	\end{proof}
	
	To make use of this, we will need some tail estimates for the largest singular value. By Corollary 2.3 of \cite{rourke-elliptic-2} $s_1(A_N)$ converges to $1+|\tau|$ a.s.. From this, we are able to obtain the following result by the Borell-TIS inequality (see Theorem 2.1.1 of \cite{AT}).
	
	\begin{lemma}
		\label{lem:misc tail results on elliptic}
		For $|\tau|<1$ and any $\epsilon>0$ we have that
		\[\lim_{N\to \infty}\E[s_1(A_N)]=1+|\tau|,\]	
		\[\lim_{N\to \infty}\E[I(s_1(A_N)>1+|\tau|+\epsilon)s_1(A_N)]e^{N\epsilon^2/32}=0.\]
		In particular, we have that
		\[\lim_{N\to \infty}\P(s_1(A_N)>1+|\tau|+\epsilon)e^{N\epsilon^2/32}=0.\]
	\end{lemma}
	\begin{proof}
		To begin we note that
		\[s_1(A_N)=\sup_{x,y\in S^{N-1}}|x^tA_Ny|.\]
		The function $x^tA_Ny$ is a centered Gaussian process on $S^{N-1}\times S^{N-1}$ with covariance
		\[\E[(x_1^tA_Ny_1)(x_2^tA_Ny_2)]=N^{-1}\left[(x_1,x_2)(y_1,y_2)+\tau(x_1,y_1)(x_2,y_2)\right].\]
		We observe that on $S^{N-1}\times S^{N-1}$ this is bounded by $2/N$. Thus by applying the Borell-TIS inequality we see that for any $u>0$
		\[\P(|s_1(A_N)-\E[s_1(A_N)]|\ge u)\le 2\exp(-Nu^2/4).\]
		In particular by the Borel-Cantelli lemma, for any $\epsilon>0$ we have a.s.
		\[\limsup_{N\to \infty}|s_1(A_N)-\E[s_1(A_N)]|<\epsilon.\]
		As we also know that $s_1(A_N)$ converges to $(1+\tau)$ a.s. we obtain the first claim.
		
		Now fix $\epsilon>0$ and take $N$ sufficiently large that $\E[s_1(A_N)]\le 1+|\tau|+\epsilon/2$. Then
		\[\E[I(s_1(A_N)>1+|\tau|+\epsilon)s_1(A_N)]=\int_{1+|\tau|+\epsilon}^{\infty}\P(s_1(A_N)>x)dx\le\]
		\[ \int_{\epsilon/2}^{\infty}\P(|s_1(A_N)-\E[s_1(A_N)]|>x)dx\le 2\int_{\epsilon/2}^\infty \exp(-Nx^2/4)dx,
		\]
		from which the second bound immediately follows. Lastly the final claim follows by Markov's inequality.
	\end{proof}
	
	With these two results, we are able to obtain the following uniform result.
	
	\begin{lemma}
		\label{lem:weaker elliptic bound}
		For any choices of $|\tau|<1$, $\ell>0$, $\epsilon>0$ and $K>0$ we have that
		\[\lim_{N\to \infty}\sup_{|z|<K}\bigg|(N\ell)^{-1}\log(\E[\exp(N\ell \int_0^{\infty}\log_{\epsilon}(x)\nu_{A_N,z}(dx))])\]
		\[-\int_0^{\infty}\log_{\epsilon}(x)\nu_{z}(dx)\bigg|= 0.\label{eqn:lemma:stronger upper bound of det-new}\]
		Moreover for $0<\epsilon<1$ we have that
		\[\lim_{N\to \infty}\sup_{|z|\ge 1}\left((N\ell)^{-1}\log(\E[\exp(N\ell
		\int_0^{\infty}\log_{\epsilon}(x)\nu_{A_N,z}(dx))])-(2+\log(|z|))\right)\le 0.\label{eqn:lemma:weak upper bound of det}\]
	\end{lemma}
	\begin{proof}
		By Jensen's inequality we have that
		\[\int_0^{\infty}\log_{\epsilon}(x)\E[\nu_{A_N,z}](dx)\le (N\ell)^{-1}\log(\E[\exp(N\ell \int_0^{\infty}\log_{\epsilon}(x)\nu_{A_N,z}(dx))]).\]
		Combining this with Lemma \ref{lem:log-sob bound} we see that
		\[\lim_{N\to \infty}\sup_{|z|<K}\bigg|(N\ell)^{-1}\log(\E[\exp(N\ell \int_0^{\infty}\log_{\epsilon}(x)\nu_{A_N,z}(dx))])\]\[-\int_0^{\infty}\log_{\epsilon}(x)\E[\nu_{A_N,z}](dx)\bigg|= 0.\]
		Thus to show (\ref{eqn:lemma:stronger upper bound of det-new}) it suffices to show that
		\[\limsup_{N\to \infty}\sup_{|z|<K}\left|\int_0^{\infty}\log_{\epsilon}(x)\E[\nu_{A_N,z}](dx)-\int_0^{\infty}\log_{\epsilon}(x)\nu_{z}(dx)\right|=0.\label{eqn:lemma:stronger upper bound of det-old}\]
		For this we define for $C>\epsilon$,
		\[\log_\epsilon^C(x)=\min(\log_\epsilon(x),\log(C)).\]
		We observe that	
		\[\left|\int_0^{\infty}\log_{\epsilon}(x)\E[\nu_{A_N,z}](dx)-\int_0^\infty \log_{\epsilon}^C(x)\E[\nu_{A_N,z}](dx)\right|\le\] \[\E[I(s_1(A_N-zI)>C)\log(s_1(A_N-zI)/C)]\le\]
		\[\E[I(s_1(A_N)>C-|z|)(s_1(A_N)+|z|-\log(C))],\]
		where in the last line we have used that $s_1(A_N-zI)\le s_1(A_N)+|z|$ and that $\log(x)\le x$. Taking $C$ sufficiently large we see by Lemma \ref{lem:misc tail results on elliptic} that
		\[\limsup_{N\to \infty}\sup_{|z|<K}\left|\int_0^{\infty}\log_{\epsilon}(x)\E[\nu_{A_N,z}](dx)-\int_0^{\infty}\log_{\epsilon}^C(x)\E[\nu_{A_N,z}](dx)\right|=0.\label{eqn:ignore-111}\]
		This as well as the uniform convergence in Lemma \ref{lem:limit measure of v_N} completes the proof of (\ref{eqn:lemma:stronger upper bound of det-old}), and thus of the first claim.
		
		To show (\ref{eqn:lemma:weak upper bound of det}) we note that for $|z|>\max(1,\epsilon)$,
		\[\int_0^\infty \log_{\epsilon}(x)\E[\nu_{A_N,z}](dx)\le \E[\log_{\epsilon}(s_1(A_N-zI))]\le \E[\log_{\epsilon}(s_1(A)+|z|)]=\]\[\E[\log(s_1(A_N)+|z|)]\le \E[s_1(A_N)]+\log(|z|).\]
		where in the last inequality we have used that for $x,y\in (0,\infty)$
		\[\log(x+y)=\log(x)+\log(1+y/x)\le y/x+\log(x).\]
		Using \ref{lem:misc tail results on elliptic} we may conclude (\ref{eqn:lemma:weak upper bound of det}) as before.
	\end{proof}
	
	The proof of Lemma \ref{lem:main elliptic ensemble bound} now immediately follows from Lemmas \ref{lem:limit measure of v_N} and \ref{lem:weaker elliptic bound}.
	
	Our next step will be to provide an overcrowding result. The method is essentially Theorem 3.1 of \cite{naumov}, which is based on the ideas of Mark Rudelson and Roman Vershynin \cite{overcrowding1, overcrowding2}. As we are in the case where the entries are Gaussian though, we are able to get a significantly stronger result by modifying the proof at a key point.
	
	\begin{lemma}
		\label{lem:naomov smallest eigenvalue bound}
		For any $\tau\in (0,1)$, there are fixed $C,c>0$ such that for any $\epsilon>0$ and $u\in \R$
		\[\P\left(s_N(A_N-uI)\le \epsilon \right)\le C (\epsilon N+e^{-cN}).\label{eqn:ignore-129292}\]
	\end{lemma}
	
	\begin{proof}
		In this proof $C,c>0$ will be, respectively, a large and small constant, which is allowed to be increased (resp. decreased) line by line, but always assumed to be independent of the choice of $\epsilon$, $u$, and $N$.
		
		We first deal with the case of $u$ is far outside of the bulk. For this note that
		\[s_N(A_{N}-uI)\ge s_N(uI)-s_1(A_N)=|u|-s_1(A_N). \label{ignore-1}\]
		Thus if $|u|>4$, we have that
		\[\P(s_N(A_N-uI)\le 1)\le \P(|u|-s_1(A_N)\le 1)\le \P(3\le s_1(A_N))\le Ce^{-cN}\]
		where the last inequality follows from Lemma \ref{lem:misc tail results on elliptic}. This demonstrates (\ref{eqn:ignore-129292}) for $\epsilon\in (0,1]$ and $u\in \R\setminus [-4,4]$. However, as long as $C\ge 1$, (\ref{eqn:ignore-129292}) is trivially true for $\epsilon>1$ and any $u$. Thus we are reduced to reduced to showing the case for $u\in [-4,4]$.
		
		Next we will recall a number of results from \cite{naumov} which they use to obtain their small singular value result (their Theorem 3.1). To ease translation between our notations, we note that in their proof of Theorem 3.1 in section 3 they consider a matrix $A=X-zI$. Their $X$ corresponds to $\sqrt{N}A_N$ in our notation (note also, we are in the Gaussian case, and they are considering the general Wigner case). They then obtain results for $z=O(\sqrt{N})$. So considering $z=\sqrt{N}u$, we may write their $A$ as $\sqrt{N}(A_N-uI)$. We also note that their $n$ is our $N$.
		
		With this noted, we first combine Lemmas 3.7 and 3.8 of \cite{naumov}. For this, let us denote the $i$-th column vector of $\sqrt{N}(A_N-uI)$ as $A_i$, and denote by $H_i\subset \R^N$ the subspace spanned by all columns of $\sqrt{N}(A_N-uI)$ except the $i$-th. Then following the proofs in Section 3 of \cite{naumov} \footnote{Note that the term $n^{-1}$ occurring in the first term of (3.2) in Lemma 3.8 is a typo and should be $n^{-1/2}$, as this is what the original Lemma 3.5 of \cite{overcrowding1} claims, which is what they are citing.} up to Lemma 3.8 implies that for any $K>1$ we have that
		\[\P \left(\sqrt{N}s_N(A_N-uI)\le \epsilon N^{-1/2},\sqrt{N}s_1(A_N-uI)\le 3K\sqrt{N}\right)\le\]\[C\left(e^{-cN}+\sum_{i=1}^{N}\P(d(A_i,H_i)\le C\epsilon)\right)\label{eqn:ingore-29332}\]
		where $d(A_i,H_i)$ denotes the distance of the vector $A_i$ to the subspace $H_i$. By Lemma \ref{lem:misc tail results on elliptic}, and for any $u\in [-4,4]$, we have that
		\[\P(s_1(A_N-uI)\ge 7)\le \P(s_1(A_N)\ge 7-|u|)\le \P(s_1(A_N)\ge 3)\le  Ce^{-cN},\label{eqn:ignorefk-fs}\]
		so we may obtain from (\ref{eqn:ingore-29332}) that
		\[\P(s_N(A_N-uI)\le \epsilon N^{-1})\le C\left(e^{-cN}+\sum_{i=1}^{N}\P(d(A_i,H_i)\le C\epsilon)\right).\label{eqn:masdinf}\]
		
		We now bound the second term on the right, taking $i=1$ for simplicity. If we denote a unit vector orthogonal to $H_1$ as $h$, they note that
		\[d(A_1,H_1)\ge |(h,A_1)|.\]
		To understand better $h$, we notate the components of $\sqrt{N}(A_N-uI)$ in terms of $1$-by-$(N-1)$ blocks as
		\[\sqrt{N}(A_N-uI)=\begin{bmatrix} a_{11}& V^T\\ U & B  \end{bmatrix}.\]
		It is here our proof diverges from \cite{naumov}. Let us condition on the event $(V=v,B=b)$ for some $v\in \R^{N-1}$ and matrix $b\in \R^{(N-1)^2}$, and denote the conditional law of $(A_1,a_{11},U)$ as $(\bar{A}_1,\bar{a}_{11},\bar{U})$. First note that the entries of the vector $\bar{A}_1=(\bar{a}_{11},\bar{U})$ remain independent Gaussian random variables. Moreover, we note that $a_{11}$ is unaffected by this conditioning, so that $\bar{a}_{11}\disteq \cal{N}(-\sqrt{N}u,1+\tau)$. As are correlated $U_i$ and $V_i$, an easy calculation shows that $\bar{U}_i\disteq \cal{N}(\tau v_i, 1-\tau^2)$. We also observe that after we have made this conditioning, $h$ is a deterministic unit vector. Thus in particular, writing out the inner-product as
		\[(\bar{A}_1,h)=a_{11}h_1+\sum_{i=1}^{N-1}\bar{U}_ih_{i+1}\]
		we see that $(\bar{A}_1,h)$ is a Gaussian random variable. In particular, if we define $\mu(h):=-\sqrt{N}u+\sum_{i=1}^{N-1}\tau v_ih_{i+1}$ and $\sigma(h)^2:=(1+\tau)h_1^2+(1-\tau^2)\sum_{i=2}^{N}h_{i}^2$, then 
		$(\bar{A}_1,h)\disteq \cal{N}(\mu(h),\sigma(h)^2)$. In particular, if we denote the probability with respect to our conditioning as $\P_{v,b}$, and let $X$ be in independent standard Gaussian random variable, then 
		\[\P_{v,b}(|(\bar{A}_1,h)|\le \epsilon)=\P(|\sigma(h)X+\mu(h)|\le \epsilon).\]
		This is a small ball probability that is easily computed. As $h$ is a unit vector, we see that $\sigma(h)^2\ge (1-\tau^2)$. Moreover, as $X$ has a density bounded by $(2\pi)^{-1/2}\le 2^{-1}$, we see that
		\[\P(|\sigma(h)X+\mu(h)|\le \epsilon)\le \sup_{y\in \R}\P\left(|X-y|\le \frac{\epsilon}{\sigma(h)}\right)\le  \frac{\epsilon}{\sigma(h)}\le \frac{\epsilon}{(1-\tau^2)^{1/2}}.\]
		In particular, we have that
		\[\P_{v,b}(|(\bar{A}_1,h)|\le \epsilon)\le \frac{\epsilon}{(1-\tau^2)^{1/2}}.\]
		As the right is simply a constant, taking expectations over $(V,B)$, we get the bound
		\[\P(d(A_1,H_1)\le \epsilon)\le \P(|(A_1,h)|\le \epsilon)\le \frac{\epsilon}{(1-\tau^2)^{1/2}}.\]
		Combined with (\ref{eqn:masdinf}) this completes the proof.
	\end{proof}
	
	With all of these results at hand, we may obtain the following bound on the characteristic polynomials outside of the bulk.
	\begin{proposition}
		\label{prop:moments of character outside bulk}
		For $\tau\in (0,1)$, $\ell>0$, and $|x|>1+\tau$, there is $C>0$ such that
		\[C^{-1}\le \frac{\E[|\det(A_N-xI)|^{\ell}]}{\E[|\det(A_N-xI)|]^{\ell}}\le C.\label{eqn:first-bound}\]
		Moreover, for any fixed $k\in \N$, we can also find $C>0$ such that
		\[C^{-1}\le \frac{\E[|\det(A_N-xI)|^{\ell}]}{\E[|\det(A_{N-k}-xI)|^{\ell}]}\le C.\label{eqn:second-bound}\]
		Lastly, both $C$ may be chosen to be uniform in compact subsets of $x$.
		
	\end{proposition}
	\begin{proof}
		To begin we take a fixed $\epsilon>0$ and recall that
		\[\E[|\det(A_N-xI)|^{\ell}]\le \E \exp\left(N\ell\int_0^\infty \log_{\epsilon}(y)\nu_{A_N,x}(dy)\right).\]
		By Jensen's inequality and Lemma \ref{lem:log-sob bound} for $C\ge \exp(8\ell^2\epsilon^{-2})$, we have that
		\[\exp\left(N\ell\int_0^\infty \log_{\epsilon}(y)\E[\nu_{A_N,x}](dy)\right)\le \E \exp\left(N\ell\int_0^\infty \log_{\epsilon}(y)\nu_{A_N,x}(dy)\right)\le\]
		\[ C \exp\left(N\ell\int_0^\infty \log_{\epsilon}(y)\E[\nu_{A_N,x}](dy)\right).\]
		Now using (\ref{ignore-1}),
		\[
		\P(\nu_{A_N,x}([0,\epsilon])\neq 0)=\P(s_N(A_N-xI)\le \epsilon)\le \P(|x|-\epsilon \le s_1(A_N)).
		\]
		Thus if we choose $\epsilon<|x|-1-\tau$, we see by Lemma \ref{lem:misc tail results on elliptic} there is $C,c>0$ such that
		\[\E[\nu_{A_N,x}]([0,\epsilon])\le \P(\nu_{A_N,x}([0,\epsilon])\neq 0)\le \P(|x|-\epsilon \le s_1(A_N))\le Ce^{-cN}.\]
		In particular, for the limiting measure we see that
		\[\nu_x([0,\epsilon])=\lim_{N\to \infty}\E[\nu_{A_N,x}]([0,\epsilon])=0.\] 
		Thus by the above results we have that
		\[\lim_{N\to \infty}\int_0^\infty \log_{\epsilon}(y)\E[\nu_{A_N,x}](dy)=\int_0^\infty \log(y)\nu_{x}(dy)=\phi_\tau(x).\]
		In particular, this shows that on the exponential scale, making the replacement of $\log$ by $\log_\epsilon$ without taking $\epsilon\to 0$ should yield a tight upper-bound, which makes sense as no singular values should greatly exceed $1+\tau$. However to push this to the constant scale will be more technical, and we first must replace $\epsilon$ with a much smaller value. 
		
		For this first note that for $\delta\in (0,\epsilon)$, we have the crude bound
		\[\left|\int_0^\infty \log_{\epsilon}(y)\E[\nu_{A_N,x}](dy)-\int_0^\infty \log_{\delta}(y)\E[\nu_{A_N,x}](dy)\right|\le \log(\epsilon/\delta)\E[\nu_{A_N,x}]([0,\epsilon]).\]
		In particular, if we take $\delta_N:=e^{-N^{10}}$, then we have (potentially for larger $C$) that
		\[\log(\epsilon/\delta_N)\E[\nu_{A_N,x}]([0,\epsilon])\le CN^{10}e^{-cN}.\]
		In particular, adjusting $C,c>0$, we have that
		\[\left|\int_0^\infty \log_{\epsilon}(y)\E[\nu_{A_N,x}](dy)-\int_0^\infty \log_{\delta_N}(y)\E[\nu_{A_N,x}](dy)\right|\le Ce^{-cN}, \label{eqn:krie-2}\]
		so enlarging $C>0$ again we have that
		\[\exp\left(N\ell\int_0^\infty \log_{\delta_N}(y)\E[\nu_{A_N,x}](dy)\right)\le \exp\left(N\ell\int_0^\infty \log_{\epsilon}(y)\E[\nu_{A_N,x}](dy)\right)\le\]
		\[ C \exp\left(N\ell\int_0^\infty \log_{\delta_N}(y)\E[\nu_{A_N,x}](dy)\right).\label{eqn:krie}\]
		In particular for larger $C>0$ we have that
		\[\E[|\det(A_N-xI)|^{\ell}]\le C \exp\left(N\ell\int_0^\infty \log_{\delta_N}(y)\E[\nu_{A_N,x}](dy)\right).\]
		We will show that we also have that
		\[\E[|\det(A_N-xI)|^{\ell}]\ge C^{-1} \exp\left(N\ell\int_0^\infty \log_{\delta_N}(y)\E[\nu_{A_N,x}](dy)\right). \label{ignore-viet}\]
		For this we define the event 
		\[\cal{E}_N=\{s_N(A_N-uI)>\delta_N\}.\]
		By Lemma \ref{lem:naomov smallest eigenvalue bound} there is $C,c>0$ such that $\P(\cal{E}_N^c)\le Ce^{-cN}$. We note that by excluding the event $\cal{E}_N^c$, we get the bound
		\[\E[|\det(A_N-xI)|^{\ell}]\ge \E [\exp\left(N\ell \int_0^\infty \log_{\delta_N}(y)\nu_{A_N,x}(dy)\right)I_{\cal{E}_N}].\]
		We may rewrite the right hand side as 
		\[\E \left[\exp\left(N\ell \int_0^\infty \log_{\delta_N}(y)\nu_{A_N,x}(dy)\right)\right]-\E \left [\exp\left(N\ell\int_0^\infty \log_{\delta_N}(y)\nu_{A_N,x}(dy)\right)I_{\cal{E}_N^c}\right].\label{eqn:icofof}\]
		We show that the second term is negligible compared to the first. Indeed, by Jensen's inequality we have that
		\[\E \left[\exp\left(N\ell \int_0^\infty \log_{\delta_N}(y)\nu_{A_N,x}(dy)\right)\right]\ge \exp\left(N\ell\int_0^\infty \log_{\delta_N}(y)\E[\nu_{A_N,x}](dy)\right),\label{eqn:fwiuw1}\] and by the Cauchy-Schwarz inequality we have that
		\[\E \left [\exp\left(N\ell\int_0^\infty \log_{\delta_N}(y)\nu_{A_N,x}(dy)\right)I_{\cal{E}_N^c}\right]\le \]\[\E \left [\exp\left(N2\ell\int_0^\infty \log_{\delta_N}(y)\nu_{A_N,x}(dy)\right)\right]^{1/2}\P(\cal{E}_N^c)^{1/2}\le \]
		\[\E \left [\exp\left(N2\ell\int_0^\infty \log_{\epsilon}(y)\nu_{A_N,x}(dy)\right)\right]^{1/2}\P(\cal{E}_N^c)^{1/2}.\label{eqn:ifjijfe} \]
		Proceeding as in the above work, we see that
		\[\lim_{N\to \infty}N^{-1}\log \E \left [\exp\left(N2\ell\int_0^\infty \log_{\epsilon}(y)\nu_{A_N,x}(dy)\right)\right]^{1/2}=\ell\int_0^\infty \log(y)\nu_{x}(dy),\]
		and similarly that
		\[\lim_{N\to \infty} N^{-1}\log \exp\left(N\ell\int_0^\infty \log_{\delta_N}(y)\E[\nu_{A_N,x}](dy)\right)=\ell\int_0^\infty \log(y)\nu_{x}(dy).\]
		In particular (\ref{eqn:fwiuw1}), (\ref{eqn:ifjijfe}), and the bound $\P(\cal{E}_N^c)\le Ce^{-cN}$ together show that
		\[\limsup_{N\to \infty}\frac{\E \left [\exp\left(N\ell \int_0^\infty \log_{\delta_N}(y)\nu_{A_N,x}(dy)\right)I_{\cal{E}_N^c}\right]}{\exp\left(N\ell \int_0^\infty \log_{\delta_N}(y)\E[ \nu_{A_N,x}](dy)\right)}=0.\]
		Thus with (\ref{eqn:icofof}) we have shown (\ref{ignore-viet}), so altogether,
		\[C^{-1}\le \frac{\E[|\det(A_N-xI)|^{\ell}]}{ \exp\left(N\ell \int_0^\infty \log_{\delta_N}(y)\E [\nu_{A_N,x}](dy)\right)}\le C.\label{eqn:12345}\]
		Note this suffices to establish (\ref{eqn:first-bound}) as employing (\ref{eqn:12345}) in the case of $\ell=1$ we obtain that
		\[C^{-\ell}\le \frac{\E[|\det(A_N-xI)|]^{\ell}}{ \exp\left(N\ell \int_0^\infty \log_{\delta_N}(y)\E [\nu_{A_N,x}](dy)\right)}\le C^\ell,\]
		so using (\ref{eqn:12345}) again in the case of our chosen $\ell$ gives (\ref{eqn:first-bound}).
		
		Next, we show how the second claim, (\ref{eqn:second-bound}), follows if we show that for fixed $0<\epsilon<1-|\tau|$,
		\[|N\int_0^{\infty}\log_{\epsilon}(y)\E[\nu_{A_N,x}](dy)-(N-k)\int_0^{\infty}\log_{\epsilon}(y)\E[\nu_{A_{N-k},x}](dy)|\le C. \label{eqn:ignore-2}\]
		Note that by using (\ref{eqn:12345}) twice, we see that (\ref{eqn:second-bound}) is equivalent to the claim that
		\[C^{-1}\le \frac{\exp\left(N\ell \int_0^\infty \log_{\delta_{N}}(y)\E [\nu_{A_N,x}](dy)\right)}{ \exp\left((N-k)\ell \int_0^\infty \log_{\delta_{N-k}}(y)\E [\nu_{A_{N-k},x}](dy)\right)}\le C,\]
		or equivalently that
		\[|N\int_0^{\infty}\log_{\delta_{N}}(y)\E[\nu_{A_N,x}](dy)-(N-k)\int_0^{\infty}\log_{\delta_{N-k}}(y)\E[\nu_{A_{N-k},x}](dy)|\le C.\]
		However, by using (\ref{eqn:krie-2}) twice, we see that 
		\[|N\int_0^{\infty}\log_{\delta_{N}}(y)\E[\nu_{A_N,x}](dy)-(N-k)\int_0^{\infty}\log_{\delta_{N-k}}(y)\E[\nu_{A_{N-k},x}](dy)|\le\]
		\[2C+|N\int_0^{\infty}\log_{\epsilon}(y)\E[\nu_{A_N,x}](dy)-(N-k)\int_0^{\infty}\log_{\epsilon}(y)\E[\nu_{A_{N-k},x}](dy)|.\]
		Thus (\ref{eqn:second-bound}) follows immediately from this if we show (\ref{eqn:ignore-2}).
		
		To prepare to show (\ref{eqn:ignore-2}), note that by employing the tails estimates of Lemma \ref{lem:misc tail results on elliptic}, we see that for $K>0$ sufficiently large
		\[\lim_{N\to \infty}|N\int_K^{\infty}\log_{\epsilon}(y)\E[\nu_{A_N,x}](dy)|=0.\]
		Let us denote the principal $(N-k)$-by-$(N-k)$ submatrix of $A_N$ by $\bar{A}_N$. We note by the interlacing property for the singular values of a submatrix, we have for any $a\in \R$
		\[|N\nu_{A_N,x}([0,a])-(N-k)\nu_{\bar{A}_N,x}([0,a])|\le k.\]
		In particular, by bounding $\log_{\epsilon}(y)$ by its maximum, we see that
		\[\left|N\int_0^{K}\log_{\epsilon}(y)\E[\nu_{A_N,x}](dy)-(N-k)\int_0^{K}\log_{\epsilon}(y)\E[\nu_{\bar{A}_N,x}](dy)\right|\le \]
		\[2k\left|\log_{\epsilon}(0)+\log_{\epsilon}(K)\right|.\]
		Now we note the distributional equality
		\[\eta_{N,k}\bar{A}_N\disteq A_{N-k},\;\;\; \eta_{N,k}=\sq{\frac{N}{N-k}}.\]
		In particular, we see that
		\[\int_0^{K}\log_{\epsilon}(y)\E[\nu_{\bar{A}_N,x}](dy)=\int_0^{\eta_{N,k}K}\log_{\epsilon}(\eta_{N,k}^{-1}y)\E[\nu_{A_{N-k},\eta_{N,k}x}](dy).\]
		Using (\ref{eqn:hoff-to}) and noting that $\eta_{N,k}=1+O(N^{-1})$ one may check that for large $C$
		\[N\left|\int_0^{\eta_{N,k}K}\log_{\epsilon}(\eta_{N,k}^{-1}y)\E[\nu_{A_{N-k},\eta_{N,k}x}](dy)-\int_0^{K}\log_{\epsilon}(y)\E[\nu_{A_{N-k},x}](dy)\right|\le C.\]
		Combining these observations completes the proof.
	\end{proof}
	
	We end this section by noting the following result which follows immediately from the proof of Lemma 4.4 of \cite{naumov}.
	
	\begin{lemma}
		\label{lem:appendix:weird rare outside}\cite{naumov}
		For any $|\tau|<1$, $0<\gamma<1$, and $u\in \R$, as well as any choice of $u_N\to u$, there is $C,c>0$ such that
		\[\P \left(s_{N-i}(A_{N}-u_NI)>cN^{\gamma-1} \text{ for all }N^{\gamma}\le i<N\right)\le Ce^{-cN^\gamma}.\]
	\end{lemma}
	\appendix
	
	\section{Computation of Covariances and Conditional Laws \label{appendix:covariance-computation}}
	
	The primary purpose of this appendix will be to study the covariance structure of the centered Gaussian vector \[(\zeta(\n),\b{G}_E(\n),\D_E \b{G}(\n),\zeta(\n(r)),\b{G}_E(\n(r)),\D_E \b{G}(\n(r))).\label{eqn:appendix-main-vector}\]
	The main result is given by Lemma \ref{lem:main-covariance-appendix}. The remainder of this section will then consist of computing various conditional expectations and densities involving (\ref{eqn:appendix-main-vector}), as well as establishing some basic results about them.
	
	We define functions 
	\[\Phi_1(x)=px^{p-1},\;\;\;\Phi_2(x)=\tau p(p-1) x^{p-2},\]
	so that
	\[\E[g_i(x)g_j(y)]=\delta_{ij}\Phi_1((x,y))+x_jy_i\Phi_2((x,y)).\label{eqn:appendix-general-model}\]
	The proof we give for Lemma \ref{lem:main-covariance-appendix} will actually hold for any centered Gaussian random field with covariance given by (\ref{eqn:appendix-general-model}), and with $\b{G}$ and $\zeta$ defined identically. In particular, it holds for the more general models considered in \cite{complexity-fyodorov-non-gradient,complexity-xavier}. We believe in addition that this generality may assist in the readability and verification of Lemma \ref{lem:main-covariance-appendix}. 
	
	We further define the function $\Phi_3(r)=\Phi_1(r)+r\Phi_2(r)$. With $\delta$ denoting the standard Kronecker $\delta$, we will further denote $\delta_{ijk}=\delta_{ij}\delta_{jk}$, $\delta_{ijkl}=\delta_{ijk}\delta_{kl}$, and $\delta_{ij\neq N-1}=\delta_{ij}(1-\delta_{i(N-1)})$. We will also use the short-hand $x_*=\sq{1-x^2}$. Finally, we fix and omit the choice of $E$ given by the following lemma from the notation.
	
	\begin{lemma}
		\label{lem:main-covariance-appendix}
		For any $r\in [-1,1]$, there exists a choice of orthonormal frame field $E=(E_i)_i$, such that for $1\le i,j,\alpha,\beta\le N-1$
		\[\E[G_\alpha(\n)G_\beta(\n(r))]=\delta_{\alpha\beta\neq N-1}\Phi_{1}(r)+\delta_{\alpha \beta (N-1)}(r\Phi_1(r)-r_*^2\Phi_2(r))\]
		\[\E[G_{\alpha}(\n)\zeta(\n(r))]=-\E[\zeta(\n)G_{\alpha}(\n(r))]=\delta_{\alpha(N-1)}r_*\Phi_3(r)\]\[\E[\zeta(\n)\zeta(\n(r))]=r\Phi_3(r)\]
		\[\E[E_iG_\alpha(\n)G_{\beta}(\n(r))]=-\E[G_\beta(\n)E_iG_\alpha(\n(r))]=\]\[r_*[\delta_{\alpha\beta\neq N-1}\delta_{i(N-1)}\Phi_1'(r)+\delta_{\beta i\neq N-1}\delta_{\alpha(N-1)}\Phi_2(r)+\delta_{\beta(N-1)}\delta_{\alpha i}\Phi_3(r)\]
		\[\delta_{\alpha\beta i(N-1)}(r\Phi_3'(r)-\Phi_2'(r))]\]
		\[\E[E_iG_\alpha(\n)\zeta(\n(r))]=\E[\zeta(\n)E_iG_\alpha(\n(r))]=\delta_{\alpha i(N-1)}r_*^2\Phi_3'(r)-\delta_{\alpha i}r\Phi_3(r)\]
		\[\E[E_iG_{\alpha}(\n)E_jG_\beta(\n(r))]=r_*^4\delta_{ij\alpha\beta (N-1)}\Phi_2''(r)\]
		\[-\Phi'_2(r)r_*^2(\delta_{ij\alpha\beta(N-1)}5r+\delta_{i\alpha\neq (N-1)}\delta_{\beta j(N-1)}r+\delta_{i\alpha (N-1)}\delta_{\beta j \neq (N-1)}r\]
		\[+\delta_{\alpha j (N-1)}\delta_{i\beta \neq (N-1)}+\delta_{\alpha j \neq (N-1)}\delta_{i\beta (N-1)}+\delta_{\alpha \beta (N-1)}\delta_{ij\neq (N-1)})\]
		\[-r_*^2\Phi_1''(r)(\delta_{ij\alpha\beta (N-1)}r+\delta_{\alpha\beta \neq (N-1)}\delta_{ij (N-1)})\]
		\[-(\delta_{j\beta(N-1)}\delta_{\alpha i}+\delta_{j\beta}\delta_{\alpha i(N-1)})r_*^2(\Phi_2(r)+\Phi_1'(r))\]
		\[+\Phi_1'(r)(\delta_{ij\neq N-1}+r\delta_{ij(N-1)})(\delta_{\alpha \beta\neq N-1}+r\delta_{\alpha \beta(N-1)})\]
		\[+(\delta_{i\beta\neq N-1}+r\delta_{i\beta(N-1)})(\delta_{\alpha j\neq N-1}+r\delta_{\alpha j(N-1)})\Phi_2(r)+r\Phi_3(r)\delta_{i\alpha}\delta_{j\beta}.\]
	\end{lemma}
	
	Before proceeding to the proof we observe that in the purely relaxational case  (i.e. $\Phi_1'=\Phi_2$) we recover Lemma 54 of \cite{complexity-second-moment-general} with $\Phi_1=\nu'$ and $q_1=q_2=1$ (up to constant differences resulting from their normalization $\nu(1)=1$), and in the case of $r=1$ we recover the covariance structure described on pg. 13 of \cite{complexity-xavier}.
	\begin{proof}
		We will take the orthonormal frame field constructed in the proof of Lemma 30 of \cite{pspin-second}, though we will need to recall its construction to establish the precise properties of it that we will use.
		
		We let $P_N:S^{N-1}\to \R^{N-1}$ denote projection given by omitting the last coordinate. We let $R_{r,N}:S^{N-1}\to S^{N-1}$ denote the rotation mapping
		\[R_{r,N}(x_1,\dots x_N)=(x_1,\dots x_{N-2},rx_{N-1}-r_*x_N,r_*x_{N-1}+rx_N).\]
		We observe that $P_N(\n)=P_N\circ R_{r,N}(\n(r))=0$. Then taking small enough neighborhoods of $\n$ and $\n(r)$ in $S^{N-1}$ that the restrictions of $P_N$ and $P_N\circ R_{r,N}$ are diffeomorphisms onto their image, we define $\hat{g}_{i}^1=g_{i}\circ P_N^{-1}$ and $\hat{g}_{i}^r=g_{i}\circ (P_N\circ R_{r,N})^{-1}$, and similarly for $\hat{\zeta}^1$ and $\hat{\zeta}^r.$
		
		We now claim that there is an orthonormal frame field $E=(E_i)_{i=1}^{N-1}$ such that for any $1\le i,j\le N-1$,
		\[(E_i\zeta(\n),E_iG_j(\n))=(\partial_i \hat{\zeta}^1(0),\partial_{i}\hat{g}_{j}^1(0)-\delta_{ij}\hat{g}^1_N(0))\label{eqn:local-frame-1}\]
		and such that for $1\le j<N-1$ and $1\le i\le N-1$,
		\[(E_i\zeta(\n(r)),E_iG_j(\n(r)))=(\partial_i \hat{\zeta}^r(0),(\partial_{i}\hat{g}_{j}^r)(0)-\delta_{ij}(r_*\hat{g}_{N-1}^r(0)+r\hat{g}_N^r(0))),\label{eqn:local-frame-2}\]
		\[E_{i}G_{N-1}(\n(r))=r\partial_i\hat{g}_{N-1}^r(0)-r_*\partial_i \hat{g}_N^r(0)-\delta_{i(N-1)}(r_*\hat{g}_{N-1}^r(0)+r\hat{g}_N^r(0)),\label{eqn:local-frame-3}\]
		where in all the expressions above, $\partial_i$ denotes the $i$-th standard partial derivative on $\R^{N-1}$.
		
		As stated above, we may infact take the orthonormal frame field constructed in the proof of Lemma 30 in \cite{pspin-second}, or more precisely, in footnote 5 on their pg. 36. For clarity, we now recall it's construction. Let $(\frac{\partial}{\partial x_i})_{i=1}^{N-1}$ denote the pullback of the standard frame field $(\frac{d}{dx_i})_{i=1}^{N-1}$ on $\R^{N-1}$ by $P_N\circ R_{r,N}$. If we now denote the standard frame field on $\R^N$ as $(\frac{\partial^{euc}}{\partial^{euc} x_i})_{i=1}^{N}$, we see that with respect to the inclusion $T_{\n(r)}S^{N-1}\subseteq T_{\n(r)}\R^N$ induced by the standard embedding $S^{N-1}\subseteq \R^N$, we have that for $1\le i<N-1$,
		\[\frac{\partial}{\partial x_i}(\n(r))=\frac{\partial^{euc}}{\partial^{euc} x_i},\;\;\;\;\; \frac{\partial}{\partial x_{N-1}}(\n(r))=r\frac{\partial^{euc}}{\partial^{euc} x_{N-1}}-r_*\frac{\partial^{euc}}{\partial^{euc} x_N},\]
		In particular, we observe that $(\frac{\partial}{\partial x_i}(\n(r)))_{i=1}^{N-1}$ is an orthonormal basis of $T_{\n(r)}S^{N-1}$. Thus we may form an orthonormal frame field on a neighborhood of $\n(r)$ by taking its value at any point $x$ to be given by the parellel transport of $(\partial_i)_{i=1}^{N-1}$ along the unique geodesic of $S^{N-1}$ connecting $\n(r)$ to $x$. If $r=1$, we extend this orthonormal frame field arbitrarily to the remainder of the sphere to obtain $E$. Otherwise, perform the same construction with $r=1$ to obtain a field around $\n$, and potentially shrinking our neighborhood, we may extend these to obtain our desired $E$. We must now verify the claims that $E$ satisfies (\ref{eqn:local-frame-1}-\ref{eqn:local-frame-3}). We observe that by our construction (\ref{eqn:local-frame-1}) is a special case of (\ref{eqn:local-frame-2}-\ref{eqn:local-frame-3}), so we are reduced to studying our frame field around $\n(r)$.
		
		For this, we note that by routine computation, and $x\in \R^{N-1}$ sufficiently small we have that \[\frac{\partial}{\partial x_i} ((P_N\circ R_{r,N})^{-1}(x))=\partial_i^{euc}-x_i( r_*\partial_{N-1}^{euc}+r\partial_N^{euc})+O(\|x\|^2),\label{eqn:ignore-temp-frame-1}\]
		\[\frac{\partial}{\partial x_{N-1}}((P_N\circ R_{-\theta})^{-1}(x))=r\partial_{N-1}^{euc}-r_*\partial_N^{euc}-x_{N-1}( r_*\partial_{N-1}^{euc}+r\partial_N^{euc})+O(\|x\|^2),\label{eqn:ignore-temp-frame-2}\]
		where we have abbreviated $\partial^{euc}_i=\frac{\partial^{euc}}{\partial^{euc} x_i}$. Now working in the coordinate system provided by $P_N\circ R_{r,N}$, we may write $E_i=\sum_{j=1}^{N-1}a_{ij}(x)\frac{\partial}{\partial x_i}$ for some smooth functions $a_{ij}$ defined on this neighborhood in $S^{N-1}$. It is shown in \cite{pspin-second} that for any $1\le i,j,k\le N-1$, $\frac{\partial}{\partial x_k}a_{ij}(\n(r))=0$, by using that the relevant Christoffel symbols vanish at $\n(r)$, and we note that by construction $a_{ij}(\n(r))=\delta_{ij}$. In particular, $E_i$ coincides with $\frac{\partial}{\partial x_i}$ to first order in $x$, so that (\ref{eqn:ignore-temp-frame-1}) and (\ref{eqn:ignore-temp-frame-2}) hold with $\frac{\partial}{\partial x_i}$ replaced by $E_i$. Additionally, we see that
		\[E_i(G(x),E_j(x))|_{x=\n(r)}=\sum_{k=1}^{N-1}\frac{\partial}{\partial x_i}(G(x),a_{jk}(\n(r))\frac{\partial}{\partial x_k})\bigg|_{x=\n(r)}=(E_iG(x),E_j(x))|_{x=\n(r)}.\]
		Combining these results we obtain (\ref{eqn:local-frame-2}) and (\ref{eqn:local-frame-3}).
		
		Now we observe that denoting
		\[\rho_N(x,y)=\sum_{i=1}^{N-2}x_iy_i+rx_{N-1}y_{N-1}+r_*x_{N-1}\|y\|_*-r_*y_{N-1}\|x\|_*+r\|x\|_*\|y\|_*,\]
		that we have
		\[\E[\hat{g}_{i}^1(x)\hat{g}_{j}^r(y)]=\delta_{ij}\Phi_1(\rho_N(x,y))+(P_N^{-1}(x))_j((P_N\circ R_{r,N})^{-1}(y))_i\Phi_2(\rho_N(x,y)),\]
		for $x,y\in \R^{N-1}$ such that the functions appearing on the left-hand side are well defined. We observe similarly that for $1\le i\le N$ and $x,y\in \R^{N-1}$ around zero,
		\[P_N^{-1}(x)_i=x_i\delta_{i\neq N}+\delta_{iN}+O(\|x\|^2),\]
		\[(P_N\circ R_{r,N})^{-1}(y)_i=y_i(\delta_{i\neq (N-1)}+r\delta_{i(N-1)})+\delta_{i(N-1)}r_*+\delta_{iN}(r-r_*y_{N-1})+O(\|y\|^2),\]
		\[\rho_N(x,y)=\sum_{i=1}^{N-2}x_iy_i+rx_{N-1}y_{N-1}+r_*x_{N-1}-r_*y_{N-1}+r+O(\|x\|^2+\|y\|^2).\]
		
		We now simply recall eqn. 5.5.4 of \cite{AT}, which states that for an arbitrary centered Gaussian field, $h$, defined on an open subset of $\R^{m}$, and with smooth covariance function $C$, we have that
		\[\E[\frac{d^k}{dx_{i_1}\dots dx_{i_k}}h(x)\frac{d^l}{dy_{j_1}\dots dy_{j_l}}h(y)]=\frac{d^k}{dx_{i_1}\dots dx_{i_k}}\frac{d^l}{dy_{j_1}\dots dy_{j_l}}C(x,y). \label{covariance of derivative formula}\]
		Using this formula, (\ref{eqn:local-frame-1}-\ref{eqn:local-frame-3}), and the above expressions for $x,y\sim 0$, the lemma now follows from a long but routine computation.
	\end{proof}
	
	From this result, we observe the following corollary, which allows us to partition the components of (\ref{eqn:appendix-main-vector}) into independent blocks.
	\begin{corollary}
		\label{corr:appendix:independence}
		Let us take $r\in (-1,1)$, and the choice of orthonormal frame field $(E_i)_{i=1}^{N-1}$ given in Lemma \ref{lem:main-covariance-appendix}. Then all of the following random variables are independent
		\begin{enumerate}
			\item $(E_{i}G_j(\n),E_{j}G_i(\n),E_{i}G_j(\n(r)),E_{j}G_i(\n(r))),$ for $1\le i<j<N-1$,
			\item $(E_iG_i(\n)+\zeta(\n),E_iG_i(\n(r))+\zeta(\n(r))),$ for $1\le i<N-1$,
			\item $(G_i(\n),E_{N-1}G_i(\n),E_{i}G_{N-1}(\n),G_i(\n(r)),E_{N-1}G_i(\n(r)),E_{i}G_{N-1}(\n(r)))$ for $1\le i<N-1,$
			\item $(\zeta(\n),\zeta(\n(r)),G_{N-1}(\n),G_{N-1}(\n(r)),E_{N-1}G_{N-1}(\n),E_{N-1}G_{N-1}(\n(r))).$
		\end{enumerate}
	\end{corollary}
	We now go back to our original setting of (\ref{eqn:model-def}). That is, we take
	\[\Phi_1(x)=px^{p-1},\;\;\Phi_2(x)=\tau p(p-1) x^{p-2}, \;\;\Phi_3(x)=p\alpha x^{p-1}.\]
	We now prepare to derive some statements involving the conditional densities. To establish non-degeneracy in some cases, we will often use an observation from (\ref{eqn:model-def-explicit}), which shows that we may write
	\[\b{f}(x)=c_s\b{f}^{s}(x)+c_a\b{f}^a(x),\label{eqn:sym-antonym-decomp}\]
	where $c_s,c_a$ are constants, and $\b{f}^s$ and $\b{f}^a$ are independent Gaussian vector-valued functions which satisfy (\ref{eqn:model-def}) with parameter $\tau=1$ and $\tau=-(p-1)^{-1}$ respectively. Abstractly, such a decomposition is clear by noting that covariance function in (\ref{eqn:model-def}) is linear in $\tau$, but more explicitly we may write
	\[f^{s}_i(x)=\D^{euc}_iH(x),\;\;\; f^{a}_i(x)=\frac{1}{N}\sum_{k=1}^{N}x_k A_{ik}(x),\label{eqn:sym-and-antisymdef}\]
	with notation as in (\ref{eqn:model-def-explicit}). We collect the facts we need about this decomposition in the following lemma, which also verifies the equivalence of (\ref{eqn:model-def}) and (\ref{eqn:model-def-explicit}), as claimed in the introduction.
	
	\begin{lemma}
		\label{lem:model-equivalence}
		For $1\le i,j\le N$ and $x,y\in \SN$
		\[\E[f^a_i(x)f^a_j(y)]=\delta_{ij}(x,y)_N^{p-1}-\left(\frac{x_jy_i}{N}\right)(x,y)_N^{p-2},\label{eqn:ignore-1010}\]
		\[\E[f^s_i(x)f^s_j(y)]=p\delta_{ij}(x,y)_N^{p-1}+p(p-1)\left(\frac{x_jy_i}{N}\right)(x,y)_N^{p-2}.\label{eqn:ignore-1011}\]
		In particular, the random vector-valued function defined by (\ref{eqn:model-def-explicit}) satisfies (\ref{eqn:model-def}). Moreover, we have that $\lam(x)=\frac{p}{N}H(x)$, so in particular $\b{f}^a$ is independent of $\lam$.
	\end{lemma}
	\begin{proof}
		We recall that for a spherical $\ell$-spin model, $\tilde{H}$, we have that for $x,y\in \SN$
		\[\E[\tilde{H}(x)\tilde{H}(y)]=N(x,y)_N^\ell.\]
		In particular, by (\ref{covariance of derivative formula}) above, we that
		\[\E[\D_i^{euc}H(x)\D_j^{euc}H(y)]=\frac{\partial^2}{\partial x_i\partial y_j}N(x,y)_N^p\]
		which immediately gives (\ref{eqn:ignore-1010}). Similarly, we see that may compute that
		\[\E\left[\left(\frac{1}{N}\sum_{k=1}^{N}x_k A_{ik}(x)\right)\left(\frac{1}{N}\sum_{l=1}^{N}y_l A_{jl}(y)\right)\right]=\frac{1}{N}\sum_{k,l=1}^{N}x_ky_l(\delta_{ij}\delta_{kl}-\delta_{il}\delta_{jk})(x,y)_N^{p-2},\]
		which immediately gives (\ref{eqn:ignore-1011}). For the final claim we compute
		\[(x,\b{f}^a(x))_N=\frac{1}{N^2}\sum_{i,j=1}^{N}x_ix_jA_{ij}(x)=0,\]
		\[(x,\b{f}^s(x))_N=\frac{1}{N}\sum_{i=1}^{N}x_i\D_i^{euc}H(x)=pH(x),\]
		where in the first line we have used that $A_{ij}$ and in the second we have used that $H$ is a homogeneous polynomial of degree-$p$.
	\end{proof}
	
	We now proceed to the proof of Lemma \ref{lem:conditioning-density-statement}. We will repeatedly use the formulas for the conditional law of Gaussian distributions (see for example, (1.2.7) and (1.2.8) of \cite{AT}), which reduces us to routine calculation using the covariance computations in Lemma \ref{lem:main-covariance-appendix}.
	
	\begin{proof}[Proof of Lemma \ref{lem:conditioning-density-statement}] 
		We first observe that for $1\le i<N-1$:
		\[\E[(E_{N-1}G_{i}(\n(r)),E_iG_{N-1}(\n(r)),G_{i}(\n(r)))(E_{N-1}G_{i}(\n),E_iG_{N-1}(\n),G_{i}(\n))^t]=\]
		\[p(p-1)\begin{bmatrix}r^{p-3}(r^2-(p-1)r_*^2)&\tau r^{p-3}(r^2-(p-1)r_*^2)&-r^{p-2}r_*\\ 
			\tau r^{p-3}(r^2-(p-1)r_*^2)&r^{p-3}(r^2-\tau (p-1)r_*^2)&-\tau r^{p-2}r_*\\
			r^{p-2}r_*&\tau r^{p-2}r_* &(p-1)^{-1}r^{p-1}\end{bmatrix},\label{eqn:covariance-edge-0}\]
		\[\E[(G_i(\n),G_i(\n(r)))(G_i(\n),G_i(\n(r)))^t]=p\begin{bmatrix}1&r^{p-1}\\ r^{p-1}&1 \end{bmatrix}.\label{eqn:covariance-edge-1}\]
		In addition, we record the covariance matrix of the remaining special entries:
		\[\E[(\zeta(\n(r)),G_{N-1}(\n(r)),E_{N-1}G_{N-1}(\n(r)))(\zeta(\n),G_{N-1}(\n),E_{N-1}G_{N-1}(\n))^t]=\]
		\[\begin{bmatrix}
			pr^{p}\alpha&pr_*r^{p-1}\alpha&-pr^{p-2}(1-pr_*^2)\alpha\\
			-pr_*r^{p-1}\alpha&pr^{p-2}(r^2-r_*^2(p-1)\tau)&a_1(r)\\
			-pr^{p-2}(1-pr_*^2)\alpha& -a_1(r)&a_2(r)
		\end{bmatrix}\label{eqn:covariance-corner-1}\]
		where here
		\[a_1(r):=r_*(r\Phi_3'(r)+\Phi_3(r)-\Phi_2'(r))=pr_*r^{p-3}(p\alpha r^2-(p-1)(p-2)\tau ),\]
		\[a_2(r):=r_*^2[r_*^2\Phi_2''(r)-5r\Phi_2'(r)-r\Phi_1''(r)-2(\Phi_2(r)+\Phi_1'(r))]+(\Phi'_1(r)+\Phi_2(r))r^2+r\Phi_3(r)=\]
		\[r^{p-4}r_*^4p(p-1)(p-2)(p-3)\tau-r^{p-2}r_*^2p(p-1)(p+5p\tau -8\tau)+r^pp(p+2(p-1)\tau).\]
		From Corollary \ref{corr:appendix:independence} we have that for $1\le i\le N-1$, the random vectors $(G_i(\n),G_i(\n(r)))$ are pairwise independent centered Gaussian vectors, with covariance matrix given by the right-hand side of (\ref{eqn:covariance-edge-1}) for $1\le i<N-1$, and for $i=N-1$ (employing (\ref{eqn:covariance-corner-1})) by 
		\[\begin{bmatrix}p&pr^p-\tau p(p-1)r^{p-2}(1-r^2)\\
			pr^p-\tau p(p-1)r^{p-2}(1-r^2)&p\end{bmatrix}.\label{eqn:appendix:ignore-second-covariance}\]
		Thus we see that to verify the claims of Lemma \ref{lem:conditioning-density-statement} involving $(G_{E}(\n),G_{E}(\n(r)))$ it suffices to show that the matrices of (\ref{eqn:covariance-edge-1}) and (\ref{eqn:appendix:ignore-second-covariance}) are invertible when $r\in (-1,1)$. The first is obvious, so we focus on the second. We see that it suffices to verify that with
		\[\hat{g}_\tau(r):=r^p-\tau (p-1)r^{p-2}(1-r^2)\]
		we have that $|\hat{g}_\tau(r)|<1$ for $|r|<1$. As $\hat{g}_\tau(r)$ is linear in $\tau$, it suffices to show that $|\hat{g}_{-(p-1)^{-1}}(r)|<1$ and $|\hat{g}_{1}(r)|\le 1$. Explicitly we have that $\hat{g}_{-(p-1)^{-1}}(r)=r^{p-2}$, and that $\hat{g}_{1}(r)=pr^p-(p-1)r^{p-2}$. The inequality for the first case is obvious and for the second we note that for $r\ge 0$,
		\[r^{p}\le p r^{p}-(p-1)r^{p} \hat{g}_{1}(r)\le p r^{p-2}-(p-1)r^{p-2}=r^{p-2},\]
		and further noting that $\hat{g}_1(-r)=(-1)^p\hat{g}_1(r)$ established the case. In particular, this confirms that the matrix (\ref{eqn:appendix:ignore-second-covariance}) is invertible for all $r\in (-1,1)$, and completes the first claims of Lemma \ref{lem:conditioning-density-statement}.
		
		To demonstrate the remaining claims in Lemma \ref{lem:conditioning-density-statement}, we need to investigate the conditional distribution of $(\zeta(\n),\zeta(\n(r)))$. It is clear that this is a centered Gaussian vector, so we only need to verify that it is non-degenerate and compute its covariance.
		
		To begin we will verify non-degeneracy. To do so let us recall the decomposition $\b{g}=c_s \b{g}^s+c_a \b{g}^a$, from Lemma \ref{lem:model-equivalence}. The field $\b{g}^a$ is independent of $\zeta$ and $c_a\neq 0$. We have just shown that $(\b{G}^a(\n),\b{G}^a(\n(r)))$ is non-degenerate, and independent of $\zeta$, so we see that the conditional law of $(\zeta(\n),\zeta(\n(r)))$ is degenerate if and only if the unconditional law of $(\zeta(\n),\zeta(\n(r)))$ is non-degenerate, which follows easily from (\ref{eqn:covariance-corner-1}).
		
		We now set the notation
		\[\E_{\D,r}[-]=\E_{\D,r}[-|\b{G}(\n)=\b{G}(\n(r))=0],\]
		which is valid for $r\in (-1,1)$ by the non-degeneracy established above. With this notation we define
		\[\Sigma_U(r):=\E_{\D,r}[(\zeta(\n),\zeta(\n(r)))(\zeta(\n),\zeta(\n(r)))^t].\]
		The remainder of the proof will be spent computing an expression for $\Sigma_U(r)$, or more precisely $\Sigma_U(r)^{-1}$, which exists by the non-degeneracy shown above.
		
		To do this we denote the covariance matrix of the random vector \[(\frac{1}{\sq{p\alpha}}\zeta(\n),\frac{1}{\sq{p}}G_{N-1}(\n),\frac{1}{\sq{p\alpha}}\zeta(\n(r)),\frac{1}{\sq{p}}G_{N-1}(\n(r))),\label{eqn:ignore-vector}\] as $\bar{\Sigma}_U(r)$. This choice of normalization is chosen so that all entries have unit variance. We note that $p\alpha\Sigma_U(r)^{-1}$ coincides with the $2$-by-$2$ submatrix composed of the first and third rows and columns of $\bar{\Sigma}_U(r)^{-1}$, so we will begin by computing this inverse.
		
		If we denote
		\[c_1(r)=r^p,\;\; c_2(r)=r_*r^{p-1}\sq{\alpha},\;\; c_3(r)=r^{p-2}(r^2-(p-1)\tau(1-r^2)),\]
		then by (\ref{eqn:covariance-corner-1}) we have that
		\[C(r):=
		\begin{bmatrix}
			c_1(r)&c_2(r)\\
			-c_2(r)&c_3(r)
		\end{bmatrix}
		,\;\;\bar{\Sigma}_U(r)=\begin{bmatrix}
			I & C(r)\\
			C(r)^t & I\\
		\end{bmatrix}.\label{eqn:ignore-vic}\]
		We will compute the block of inverse of this matrix by employing the Schur complement formula, though to do this we will perform some preliminary computations. We first compute that
		\[I-C(r)^tC(r)=\begin{bmatrix}f_1(r)& f_2(r)\\ f_2(r) & f_3(r)\end{bmatrix},\]
		where here
		\[f_1(r):=1-r^{2p}-\alpha r^{2p-2}(1-r^2),\;\; f_2(r):=-(p-1)\tau \sq{\alpha} r^{2p-3}(1-r^2)^{3/2},\]
		\[f_3(r):=1-r^{2p-2}(1-r^2)\alpha-r^{2p-2}(r^2-(1-r^2)(p-1)\tau)^2.\]
		We may then introduce functions
		\[b(r)=f_1(r)f_3(r)-f_2(r)^2,\;\; k_1(r)=f_3(r),\;\; k_2(r)=c_2(r)f_2(r)-c_1(r)f_3(r),\]
		so that by the Schur complement formula
		\[\Sigma_U(r)^{-1}=\frac{1}{p\alpha b(r)}\begin{bmatrix}
			k_1(r)& k_2(r)\\ 
			k_2(r)& k_1(r)
		\end{bmatrix}.\label{eqn:def:U-external field form}\]
		Finally by direct computation we see that
		\[b(r)=(1-r^{2p-2})^2-(p-1)^2\tau^2r^{2(p-2)}(1-r^2)^2,\]
		\[k_1(r)=b(r)+r^{2p-2}(1-r^{2p-2}+\tau(p-1)(1-r^2)),\]
		\[k_2(r)=-r^{p}(1-r^{2p-2})-\tau(p-1) r^{3p-4}(1-r^2).\]
	\end{proof}
	
	We now define the remaining conditional density matrices we need. As we only need to understand their behavior as $r\to 1$ and $r\to 0$, we will avoid giving them as explicitly as we gave our expressions for $\Sigma_U$, writing them instead as a product of certain fixed matrices. 
	
	To begin, we denote $\n^1(r):=\n$ and $\n^2(r):=\n(r)$. We define for $k,l=1,2$ and $1\le i<N-1$
	\[\Sigma_{Z}^{k,l}(r):=\frac{1}{p(p-1)}\E_{\D,r}[(E_{i}G_{N-1}(\n^k(r)),E_{N-1}G_i(\n^k(r)))\times\]
	\[(E_{i}G_{N-1}(\n^l(r)),E_{N-1}G_i(\n^l(r)))^t].\label{eqn:ignore-plane}\]
	We note that this quantity is independent of $i$, and that the conditional covariance matrix of the vector \[\frac{1}{\sq{p(p-1)}}(E_{i}G_{N-1}(\n),E_{N-1}G_i(\n),E_{i}G_{N-1}(\n(r)),E_{N-1}G_i(\n(r)))\]
	on the event $(\b{G}(\n)=\b{G}(\n(r))=0)$ is given by the matrix
	\[\Sigma_Z(r):=\begin{bmatrix}\Sigma_{Z}^{1,1}(r)& \Sigma_{Z}^{1,2}(r)\\ \Sigma_{Z}^{2,1}(r) & \Sigma_{Z}^{2,2}(r)\end{bmatrix}.\label{eqn:def:SigmaZ}\]
	We note that we also have that $\Sigma_Z^{1,1}(r)=\Sigma_Z^{2,2}(r)$ and $\Sigma_Z^{1,2}(r)^t=\Sigma_Z^{2,1}(r)$. By Corollary \ref{corr:appendix:independence}, this is also the covariance matrix after conditioning, as these random variables are independent of $(\zeta(\n),\zeta(\n(r)))$.
	
	To give a more explicit representation, let us denote the unconditioned covariance matrix of
	\[\frac{1}{\sq{p(p-1)}}(E_{i}G_{N-1}(\n),E_{N-1}G_i(\n),E_{i}G_{N-1}(\n(r)),E_{N-1}G_i(\n(r)))\]
	as $\Sigma_A(r)$ and further define
	\[\Sigma_B(r):=\frac{1}{p(p-1)}\E[(E_{i}G_{N-1}(\n),E_{N-1}G_i(\n),E_{i}G_{N-1}(\n(r)),E_{N-1}G_i(\n(r)))\]\[(G_i(\n),G_i(\n(r)))^t],\]
	then by Corollary \ref{corr:appendix:independence}, (\ref{eqn:covariance-edge-1}), and the Schur complement formula we have that
	\[\Sigma_Z(r)=\Sigma_A(r)-\Sigma_B(r)\frac{1}{p(1-r^{2(p-1)})}\begin{bmatrix}1&-r^{p-1}\\-r^{p-1} &1 \end{bmatrix}\Sigma_B(r)^t.\]
	
	Similarly if we denote the (unconditional) covariance matrix of \[(\zeta(\n),\zeta(\n(r)),G_{N-1}(\n),G_{N-1}(\n(r)))\] as $\Sigma_C(r)$ and
	\[\Sigma_D(r):=\frac{1}{p(p-1)}\E[(E_{N-1}G_{N-1}(\n)-\zeta(\n),E_{N-1}G_{N-1}(\n(r))-\zeta(\n(r)))\times\]\[(\zeta(\n),\zeta(\n(r)),G_{N-1}(\n),G_{N-1}(\n(r)))^t],\]
	we may express the covariance matrix of \[(E_{N-1}G_{N-1}(\n)-\zeta(\n),E_{N-1}G_{N-1}(\n(r)))-\zeta(\n(r))\]
	conditioned on the event
	\[(\b{G}(\n)=\b{G}(\n(r))=0,\zeta(\n)=u_1,\zeta(\n(r))=u_2)\]
	as
	\[\Sigma_S(r)=\frac{1}{p(p-1)}\begin{bmatrix}p(p-1)(1+\tau)& a_2(r)\\ a_2(r)& p(p-1)(1+\tau) \end{bmatrix}-\Sigma_D(r)\Sigma_C(r)^{-1}\Sigma_D(r)^t,\label{eqn:S-def}\]
	and the mean as
	\[m^k_\ell(r)=\Sigma_D(r)\Sigma_C(r)^{-1}e_{\ell}.\label{eqn:m-def}\]
	
	We are now ready to give the proofs of Lemmas \ref{lem:jacobian-covariance-statement-no-energy} and \ref{lem:jacobian-covariance-statement-with-energy}.
	
	\begin{proof}[Proof of Lemmas \ref{lem:jacobian-covariance-statement-no-energy} and \ref{lem:jacobian-covariance-statement-with-energy}]
		We begin with the proof of Lemma \ref{lem:jacobian-covariance-statement-no-energy}. All of the independence statements follow from Corollary \ref{corr:appendix:independence}. The law of the elements of $(W^k_N(r),V^k_N(r))$ follows essentially by definition, and as the distribution of $(G^k_N(r))_{k=1,2}$ is unaffected by the conditioning, its law follows easily from Lemma \ref{lem:main-covariance-appendix}. This completes the proof of Lemma \ref{lem:jacobian-covariance-statement-no-energy}. 
		
		The proof of Lemma \ref{lem:jacobian-covariance-statement-with-energy} follows from Lemma \ref{lem:jacobian-covariance-statement-no-energy} using Corollary \ref{corr:appendix:independence}.
	\end{proof}

	Next we will give the proof of Lemma \ref{lem:vanishing rate of E}.
	
	\begin{proof}[Proof of Lemma \ref{lem:vanishing rate of E}]
		By employing (\ref{eqn:covariance-edge-0}) we see that $\Sigma_B(r)=O(r^{p-2})$ entrywise, so that
		\[\Sigma_Z(r)=\Sigma_A(r)+O(r^{2(p-2)}).\]
		Employing (\ref{eqn:covariance-edge-1}) we see that $\Sigma_A(r)=I+O(r^{p-3})$, so that in total we have that $\Sigma_Z(r)=I+O(r^{p-3})$. In view of the expressions for $Z_{N-1,i}^k(r)$ and $Z_{i,N-1}^k(r)$, this establishes the first two claims in (\ref{eqn:ignore-druid-1}).
		
		Employing (\ref{eqn:covariance-corner-1}) we see that $\Sigma_C(r)=\mathrm{diag}(p\alpha,p\alpha,p,p)+O(r^{p-2})$, where $\mathrm{diag}(x_1,\dots, x_n)$ denotes the diagonal matrix with entries $x_1,\dots x_n$. Noting that $a_1(r)=O(r^{p-3})$, we see that $\Sigma_{D}(r)=O(r^{p-3})$, so that $\Sigma_D(r)^t\Sigma_C(r)^{-1}\Sigma_D(r)=O(r^{2(p-3)})$. Now finally $a_2(r)=O(r^{p-4})$ we see that
		\[\Sigma_S(r)=(1+\tau)I+O(r^{p-4}),\]
		which shows the first claim of (\ref{eqn:ignore-druid-2}), with the second following similarly.
	\end{proof}
	
	We will now finally prove a vanishing result for these functions as $r\to \pm 1$. For this, let us denote for $1\le i\le N-2$
	\[\sigma_1(r)=\E_{\D,r}[E_{N-1}G_i(\n)^2]=p(p-1)[\Sigma_Z(r)]_{22},\label{eqn:def:sigma1}\]
	\[\sigma_2(r)=\E_{\D,r}[E_{N-1}G_{N-1}(\n)^2]=p(p-1)[\Sigma_S(r)]_{11}.\label{eqn:def:sigma2}\]
	\begin{lemma}
		\label{lem:appendix:covariance around 1}
		There is $C>0$ such that for $i=1,2$ and $r\in (-1,1)$
		\[\sigma_i(r)\le C(1-r^2).\]
	\end{lemma}
	\begin{proof}
		We note that $\sigma_i(r)\ge 0$ for $r\in (-1,1)$. As moreover, each $\sigma_i(r)$ is a rational function in $r$, we see that it suffices to show that for $i=1,2$
		\[\lim_{r\to \pm 1}\sigma_i(r)=0.\]
		Moreover, as both expressions are even in $r$ it suffices to only consider the case $r\to 1$. 
		
		For this we note that employing (\ref{eqn:covariance-edge-0}) the covariance matrix of the vector \[(E_{N-1}G_i(\n),G_i(\n(r)),G_i(\n))\] 
		is given by
		\[\begin{bmatrix}
			p(p-1)& p(p-1)r^{p-2}(1-r^2)^{1/2} & 0\\
			p(p-1)r^{p-2}(1-r^2)^{1/2}& p& pr^{p-1}\\
			0 & pr^{p-1} & p
		\end{bmatrix}.\]
		By this and Corollary \ref{corr:appendix:independence}, we have that
		\[\sigma_1(r)=p(p-1)-p^2(p-1)^2r^{2(p-2)}(1-r^2)\begin{bmatrix}
			p & pr^{p-1}\\
			pr^{p-1} & p
		\end{bmatrix}^{-1}_{22}=\]
		\[p(p-1)\left[1-\frac{(p-1)r^{2(p-2)}(1-r^2)}{(1-r^{2p-2})}\right].\label{eqn:def:sigma2-2}\]
		Direct computation with the above expressions yield that
		\[\lim_{r\to 1}\sigma_1(r)=p(p-1)\left[1-(p-1)\lim_{r\to 1}\left(\frac{1-r^2}{1-r^{2p-2}}\right)\right]=0.\]
		
		The computation of $\sigma_2(r)$ proceeds similarly. Using (\ref{eqn:covariance-corner-1}) we see that
		\[\sigma_1(r)=p(p+2\tau (p-1))-\frac{pa_1(r)^2}{p^2-(pr^{p}-\tau p(p-1)r^{p-2}(1-r^2))^2}.\label{eqn:def:sigma1-2}\]
		We further note the following limits
		\[\lim_{r\to 1}a_1(r)/\sq{1-r^2}=p(p+2\tau (p-1)),\]
		\[\lim_{r\to 1}\frac{1-r^2}{p^2-(pr^{p}-\tau p(p-1)r^{p-2}(1-r^2))^2}=\frac{1}{p^2(p+2\tau (p-1))}.\] 
		Employing these one may similarly obtain that $\lim_{r\to 1}\sigma_2(r)=0$.
	\end{proof}
	\section{The Kac-Rice formula\label{appendix:Kac-Rice-Proof}}
	
	In this section we derive the exact form of the Kac-Rice formula used in Lemma \ref{lem:Kac-Rice} above. This is essentially a modification of Theorem 6.4 of \cite{azais} to a general Riemannian manifold. The method to adapt such a result from the Euclidean case to a general Riemannian manifold is also discussed in (1.3.2) of Chapter 6 in \cite{azais}, though not in a generality sufficient for our case, so we include a statement and proof of our result for completeness.
	
	In preparation, we denote for a smooth manifold $M$, a vector field $f$ on $M$, a function $h:M\to \R^k$, and a choice of $B\subseteq \R^k$, the count
	\[N_{f,h}(B)=\#\{x\in M:f(x)=0,h(x)\in B\}.\]
	\begin{lemma}
		\label{lem:formal-KR}
		Let $M$ be a Riemannian manifold of dimension $d$, $f$ a random vector field on $M$ and $h:M\to \R^k$ a random function such that $(f,h):M\to TM\times \R^k$ is a smooth Gaussian random field. We assume that for each $x\in M$, the Gaussian vector $f(x)$ is non-degenerate. Then for any choice of a (piecewise) smooth orthonormal frame field $E=(E_i)_{i=1}^{d}$ on $M$, and open $B\subseteq \R^k$, we have that
		\[\E[N_{f,h}(B)]=\int_M \E[|\det(\D_Ef(x))|I(h(x)\in B)|f_E(x)=0]\varphi_{f_E(x)}(0)\omega_M(dx),\label{eqn:general-kac-rice}\]
		where $\omega_M$ denotes the Riemannian volume form on $M$, $\varphi_{f_E(x)}(0)$ denotes the density of the random vector $f_{i,E}(x)=(f(x),E_i(x))_M\in \R^d$, where $(*,*)_M$ denotes the Riemannian metric on $M$, and lastly $[\D_Ef(x)]_{ij}=E_if_{j,E}(x)$.
	\end{lemma}
	\begin{proof}
		We will begin with the case where $M$ is identified with an open subset $U$ of $\R^d$, with the standard Euclidean metric, and where $E$ is the standard Euclidean frame field. We first verify the conditions of Theorem 6.2 of \cite{azais} with $Z=f_E$. The only one which we have not taken as an assumption is that \[\P\bigg(\text{there is }x\in U\text{ with }f(x)=0\text{ and }\det(\D_E f(x))=0\bigg)=0.\label{eqn:ignore-compact-1}\]
		On the other hand employing their Proposition 6.5 (and an exhaustion of $U$ by compact subsets) we see that (\ref{eqn:ignore-compact-1}) follows from the non-degeneracy assumption on $f(x)$ for all $x\in U$. Thus by Theorem 6.4 of \cite{azais} we see that for any bounded continuous function $g:\R^k\to \R$, we have that
		\[\E[\sum_{x\in M:\D_E f(x)=0}g(h(x))]=\int_U \E[|\det(\D_E f(x))|g(h(x))|f_E(x)=0]\varphi_{f_E(x)}(0)\prod_{i=1}^ddx_i.\label{eqn:ignore-weak-euclidean-kr}\]
		On the other hand, as $B$ is open, we may write the indicator function of $B$ as the pointwise limit of an monotone increasing sequence of positive continuous functions (take for example $\epsilon^{-1}d(x,U^c)/(1+\epsilon^{-1}d(x,U^c))$ where here $d(x,A)=\inf_{y\in A}\|x-y\|$). By applying the monotone convergence theorem to (\ref{eqn:ignore-weak-euclidean-kr}) for these functions we obtain (\ref{eqn:general-kac-rice}) in our current case.
		
		To proceed, we now relax the assumption on the metric and frame field. We will denote by $\D^{euc}$ the standard Euclidean frame field on $\R^d$, and denote by $f^{euc}$ and $\D^{euc}f$ the quantities with respect to this basis. We observe that for $1\le i\le d$ we have that $E_i(x)=[g(x)^{-1/2}\D^{euc}]_i$ where $[g(x)]_{ij}=(E_i(x),E_j(x))$ is the metric tensor. We also have that $f_E(x)=g(x)^{-1/2}f^{euc}(x)$ and that
		\[\D_Ef(x)=g(x)^{-1}\D^{euc}f(x)+g^{-1/2}(x)(\D^{euc}g^{-1/2}(x))f^{euc}(x).\]
		From these relations we derive that $\varphi_{f_E(x)}(0)=\det(g(x))^{1/2}\varphi_{f^{euc}(x)}(0)$ and
		\[\E[|\det(\D_E f(x))|I(f(x)\in B)|f_E(x)=0]=\]\[\det(g(x))^{-1}\E[|\det(\D^{euc} f(x))|I(f(x)\in B)|f^{euc}(x)=0].\]
		As we have that $\omega_U(dx)=\det(g(x))^{1/2}\prod_{i=1}^{d}dx_i$, we see thus that the integral on the right-hand side of (\ref{eqn:ignore-weak-euclidean-kr}) is independent of both the choice of metric on $M$, and of the choice of orthonormal frame. Thus we have established the case where $M$ may be embedded as an open subset of Euclidean space.
		
		Now finally for the case of a general manifold $M$, we take a countable cover of $M$ by Euclidean neighborhoods $(U_i)_{i=1}^{\infty}$ and define $V_m=\bigcup_{i=1}^{m}U_i$. We observe that all non-empty finite intersections of the $U_i$ may be realized as open subsets of $\R^N$, so by the inclusion-exclusion principle, we see that (\ref{eqn:general-kac-rice}) holds when both sides are restricted to $V_m$. Again applying the monotone convergence theorem completes the proof.
	\end{proof}


	%\bibliography{main}{}
	%\bibliographystyle{sn-mathphys}      % mathematics and physical sciences
	
\end{document}